\documentclass{birkjour}

\usepackage{amsmath,amstext,amsfonts,amsthm,amssymb}
\usepackage{eucal}
\usepackage{diagrams}
\diagramstyle[scriptlabels]
\newarrow{EQ}{=}{=}{=}{=}{=}

\swapnumbers

\newtheorem{thm}[subsubsection]{Theorem}
\newtheorem{lem}[subsubsection]{Lemma}
\newtheorem*{Lem}{Lemma}
\newtheorem{prp}[subsubsection]{Proposition}
\newtheorem*{Prp}{Proposition}
\newtheorem{crl}[subsubsection]{Corollary}
{  \theoremstyle{definition}
           \newtheorem{dfn}[subsubsection]{Definition}%[section]
           \newtheorem{exm}[subsubsection]{Example}%[section]
           \newtheorem{rem}[subsubsection]{Remark}%[section]
           \newtheorem{rems}[subsubsection]{Remarks}%[section]
           \newtheorem*{Rem}{Remark}
           \newtheorem*{Rems}{Remarks}
}

\newcommand{\defn}{\ \\[2pt] \noindent {\bf Definition.\ }}
\newcommand{\eproof}{\ \hfill $\Box$}
 \newcommand{\chs}[1]{{{#1}\brack 2}}

\newcommand{\bs}{\backslash}
\newcommand{\ol}{\overline}

\newcommand{\cA}{\mathcal{A}}

\newcommand{\cC}{\mathcal{C}}

\newcommand{\cF}{\mathcal{F}}

\newcommand{\cL}{\mathcal{L}}
\newcommand{\cM}{\CMcal{M}}% moduli space as a coarse space 
\newcommand{\sM}{\mathfrak{M}} % moduli space as a stack
\newcommand{\cO}{\mathcal{O}}
\newcommand{\cP}{\mathcal{P}}
\newcommand{\cQ}{\mathcal{Q}}
\newcommand{\cS}{\mathcal{S}}
\newcommand{\cT}{\mathcal{T}}
\newcommand{\cU}{\mathcal{U}}
\newcommand{\cV}{\mathcal{V}}
\newcommand{\cW}{\mathcal{W}}
\newcommand{\cX}{\mathcal{X}}
\newcommand{\cY}{\mathcal{Y}}
\newcommand{\cZ}{\mathcal{Z}}

\newcommand{\id}{\mathrm{id}}
\newcommand{\pr}{\mathrm{pr}}

\newcommand{\Z}{\mathbb{Z}}
\renewcommand{\P}{\mathbb{P}}

\newcommand{\Q}{\mathbb{Q}}
\newcommand{\C}{\mathbb{C}}
\newcommand{\R}{\mathbb{R}}
\newcommand{\bM}{\overline{\cM}}
\newcommand{\bsM}{\overline{\sM}}
\newcommand{\bT}{\overline{\cT}}

\newcommand{\reg}{\mathrm{reg}}
\renewcommand{\kappa}{\varkappa}
\renewcommand{\epsilon}{\varepsilon}

\newcommand{\ad}{\operatorname{ad}}
\newcommand{\Out}{\operatorname{Out}}

\newcommand{\EQ}{\operatorname{Eq}}
\newcommand{\SEQ}{\operatorname{Seq}}
\newcommand{\Grp}{\mathtt{Grp}}
\newcommand{\Cat}{\mathtt{Cat}}
\newcommand{\Orbi}{\mathtt{Orbi}}
\newcommand{\Top}{\mathtt{Top}}

\newcommand{\Diff}{\operatorname{\mathcal{D}\!\mathit{iff}}}
\newcommand{\Atlas}{\mathtt{Atlas}}

\newcommand{\Hom}{\operatorname{Hom}}%{\mathrm{Hom}}
\newcommand{\Mor}{\operatorname{Mor}}
\newcommand{\Ob}{\operatorname{Ob}}%
\newcommand{\End}{\operatorname{End}}%{\mathrm{Hom}}

\newcommand{\Image}{\mathrm{Im}}
\newcommand{\Ker}{\mathrm{Ker}}

\newcommand{\an}{\mathit{an}}
\newcommand{\alg}{\mathit{alg}}

\newcommand{\nor}{\mathrm{nor}}
\newcommand{\naive}{\mathrm{nv}}

\newcommand{\Sp}{\mathsf{Sp}}  % category of spaces
\newcommand{\Do}{{D_0}}

\newcommand{\Ens}{\mathtt{Set}}
\newcommand{\op}{\mathrm{op}}
\newcommand{\Pre}{\mathtt{Pre}}
\newcommand{\Sch}{\mathtt{Sch}}
\newcommand{\Stacks}{\mathtt{Stacks}}
\newcommand{\Charts}{\mathtt{Charts}}

\newcommand{\Spec}{\mathrm{Spec}}
\newcommand{\Stab}{\mathrm{Stab}}
\newcommand{\Aut}{\operatorname{Aut}}
\newcommand{\Iso}{\operatorname{Iso}}

\newcommand{\MSt}{\Stacks/\Sp}

\newcommand{\isom}{\simeq}

\newcommand{\dirLim}{\underset{\longrightarrow}{\operatorname{Lim}}}
\newcommand{\dirtLim}{\underset{\longrightarrow}{{\cL}\!{\mathit{im}}}}
\newcommand{\teich }{Tei\-ch\-m\"ul\-ler}
\newcommand{\TS}{\teich\ space}
\newcommand{\LR}{\mathtt{LR}} % category of locally ringed spaces
\newcommand{\COM}{\mathtt{COM}} % category of commutative rings
\newcommand{\Adm}{\mathfrak{Adm}}
\newcommand{\cAdm}{\mathcal{A}dm}
\newcommand{\wh}{\widehat}
\newcommand{\wt}{\widetilde}

\newcommand{\Gal}{\mathrm{Gal}}

\begin{document}

\title%[Augmented \TS s]
{Augmented \TS s and Orbifolds }

\author{Vladimir Hinich}
\address{Department of Mathematics\\
 University of Haifa\\
 Mount Carmel,  Haifa 31905,  Israel}
 \email{hinich@math.haifa.ac.il}
 \author{Arkady Vaintrob}
 \address{Department of Mathematics\\
 University of Oregon\\
 Eugene, OR  97403, USA}
 \email{vaintrob@math.uoregon.edu}

\thanks{The work of the second author was partially supported by the NSF
  grant DMS-0407000}

% 32G15,  % Moduli of Riemann surfaces, Teichm\"uller theory
% 30F60,  % \teich\ theory
% 55N32,  % Orbifold cohomology 
% 57R18   % Topology and geometry of orbifolds 

\date{}

\begin{abstract}

We study complex-analytic properties of the augmented \TS s $\bT_{g,n}$ 
obtained by adding to the classical \TS s  $\cT_{g,n}$ points 
corresponding to Riemann surfaces with nodal singularities.
Unlike $\cT_{g,n}$, the space $\bT_{g,n}$ is not a complex 
manifold (it is not even locally compact). 
We prove however that the quotient of the augmented \TS\
by any finite index subgroup
of the \teich\ modular group has a canonical
structure of a complex orbifold.     
Using this structure we construct natural maps
from $\bT$ to stacks of admis\-sible coverings of stable Riemann surfaces.
This result is important for understanding the cup-product in stringy
orbifold cohomology. We also establish
some new technical results from the general theory of orbifolds which  
may be of independent interest.
\end{abstract}

\maketitle

\setcounter{tocdepth}{2}
\tableofcontents

\section{Introduction}

\label{sec:intro}

\subsection{Augmented \TS s}
\label{ss:ats}
The \TS\ $\cT_{g,n}$ is the space of   pairs $( (X,x_1,\ldots,x_n),\phi)$,
where $(X,x_1,\ldots,x_n)$ is a 
compact  complex curve (Riemann surface) 
of genus $g$ with  $n$   distinct marked points (which we also call
\emph{punctures})  and  
$$
\phi:   (S,p_1,\ldots,p_n) \to (X,x_1,\ldots,x_n)
$$ is
a \emph{marking} --- an isotopy class of orientation 
preserving diffeomorphisms with a fixed compact oriented 
surface $S$ with $n$ marked points $p_1,\ldots,p_n\in S$.
Lipman Bers in~\cite{bers} introduced
the augmented \TS\ $\bT_{g,n}$ by adding to $\cT_{g,n}$
points corresponding to Riemann surfaces with nodes.
(A marking of a nodal Riemann surface $X$ is an isotopy class of maps
 $\phi: S \to  X$, such that the preimages of nodes are simple closed curves
  on  $S$, see Definition~\ref{dfn:ATS}.)

Let $\Gamma_{g,n}=\pi_0(\Diff^+(S,p_1,\ldots,p_n))$ be the 
\teich\  modular group, i.e.\ the group 
of isotopy classes of 
orientation preserving diffeomorphisms $(S,p_1,\ldots,p_n)$.
(This group is also known as the mapping class group of the $n$-punctured
surface of genus $g$, cf.~\cite{Iv}).
We will frequently denote this group simply by $\Gamma$.
The modular group $\Gamma$ naturally acts on $\bT_{g,n}$ 
and the quotient $\Gamma\bs \bT_{g,n}$ is homeomorphic to 
the Deligne-Mumford-Knudsen compactification $\bM_{g,n}$ of the 
moduli space of Riemann surfaces of genus $g$ with $n$ punctures. 
One of Bers' goals was an attempt to introduce a natural complex
structure on $\bM_{g,n}$ and to prove its projectivity.
The existence of a normal complex structure on the quotient
$\Gamma\bs \bT_{g,n}$ was announced by Bers~\cite{bers,bers2,bers3},
but, to the best of our knowledge, no detailed proof of this result 
had been published.

Unlike the usual \TS, the space $\bT_{g,n}$ is not a 
manifold (it is not even locally compact). Still, the augmented \TS s 
 play an important role  in \teich\ theory (see~\cite{ab3}).
In particular they appear in the study of the Weil-Petersson metric on
$\cT_{g,n}$ (see~\cite{masur,wolpert,brock,brock-marg}).
One of the goals of this paper is to understand and study the
augmented \TS\ from the complex analytic point of view. 
Our main results suggest that the space $\bT_{g,n}$ 
can be viewed as a certain universal space
of coverings of stable Riemann surfaces of genus $g$  
ramified in at most $n$ points and from this point of 
view it should be thought of as a projective system of complex orbifolds.

\subsection{Results}
\label{sec:intr_results}
Our main result is the following theorem 
(a combined statement of~\ref{thm:orbifoldstructure},~\ref{sss:2groups}
and~\ref{covered-by}). 

\medskip

\noindent
{\bf Theorem.}
{\em Let $G$ be a finite index subgroup of the \teich\ modular 
group $\Gamma_{g,n}$, where $(g,n)$ is in the stable range (i.e.\  $2g-2+n > 0$).
  \begin{itemize}
  \item[(i)]
The quotient $G\bs \bT_{g,n}$ has a structure of a complex orbifold
such that $G\bs \cT_{g,n}$ is its open suborbifold.
\ 
In particular, $G\bs \bT_{g,n}$ is a compact normal complex space.

  \item[(ii)]
For every finite index subgroup $G' \subset G$ 
there exists a canonical morphism 
$G'\bs\bT\to G\bs\bT$ of 
the corresponding           orbifolds.
  \item[(iii)]
There exists a finite index subgroup $G' \subset G$ such that the orbifold
$G'\bs \bT_{g,n}$ is a manifold (i.e.\ each point has a trivial stabilizer).
  \end{itemize}
}

We will denote the quotient $G\bs \bT_{g,n}$ with this orbifold structure by
$[G\bs \bT_{g,n}]$.  For $G=\Gamma_{g,n}$ 
the resulting  orbifold  $[\Gamma_{g,n}\bs \bT_{g,n}]$ 
coincides with the Deligne-Mumford moduli stack%
\footnote{We use the fraktur font to distinguish stacks and  orbifolds 
(such as $\bsM_{g,n}$ and $\Adm_{g,n,d}$)  from  underlying coarse spaces 
denoted by the mathcal font (resp.\ $\cM_{g,n}$ and $\cAdm_{g,n,d}$).} 
$\bsM_{g,n}$ of stable curves of genus $g$ with $n$ punctures.
In Section~\ref{ss:gluing-ATS} we prove that the natural gluing operations
on the collection of stacks $\bsM_{g,n}$ of stable marked curves can be extended 
to give canonical operations on the collection of orbifolds  $[G\bs \bT]$.

As we prove in Section~\ref{ss:choice-g}, the first statement of this theorem 
also holds for certain finite  extensions 
$\wt{G}\to G\subset \Gamma_{g,n}$ of finite index subgroups of $\Gamma_{g,n}$ 
acting on $\bT_{g,n}$ via the homomorphism  $\gamma: \wt{G}\to \Gamma_{g,n}$.
This leads to our second main result---a discovery of a connection between 
the augmented  \TS\  $\bT_{g,n}$ and the moduli space\footnote{In fact, this 
is a Deligne-Mumford stack over $\mathbb{Z}[\frac{1}{d!}]$.
} 
 $\Adm_{g,n,d}$ of  admissible coverings $\pi:\wt{X}\to X$ of degree $d$, 
where $X$ is a stable complex curve of genus $g$ with $n$ punctures 
(see e.g.~\cite[Sect.~4]{ACV}).

Let $S$ be a compact oriented  surface of genus $g$ with $n$
punctures and let $\rho:\wt{S}\to S$ be a finite covering unramified
outside the punctures. 
For any stable complex curve $X$  of genus $g$ with $n$ punctures 
and a marking $\phi:S\to X$ the map $\rho$ induces an admissible 
covering $\wt{X}\to X$.  Thus, on the level of points, $\rho$ gives a map 
\begin{equation}
  \label{eq:rho}
v_\rho:\bT_{g,n} \to \cAdm_{g,n,d}~.
\end{equation}
In Section~\ref{sec:teich-vs-adm} we  show that this map can be elevated to 
a continuous map from  $\bT_{g,n}$ to 
the complex orbifold $\Adm_{g,n,d}$.
To do this we first construct a morphism of complex orbifolds
$[\wt{G}\bs\bT_{g,n}]\to\Adm_{g,n,d}$, where $\wt{G}$ is a finite extension
of a finite index subgroup of $\Gamma_{g,n}$. Then we compose
this morphism with the canonical map $\bT_{g,n}\to [\wt{G}\bs\bT_{g,n}]$.
(Note that, since  $ [\wt{G}\bs\bT_{g,n}]$ is not a quotient orbifold, the existence of this
map is  non-obvious. It is constructed in Section~\ref{sec:projection}
using parts (ii) and (iii) of the above theorem.) 

An important application of this result is a proof given in
Section~\ref{sec:cohom}  of associativity of stringy orbifold cohomology (see
below in~\ref{sec:motivation}). 

\

Since the projection $\left[G\bs\bT_{g,n}\right] \to \bsM_{g,n}$ 
is a finite morphism, the complex orbifold $\left[G\bs\bT_{g,n}\right]$
is projective. It is equipped with a tautological 
family of stable curves $\pi: \cX \to [G\bs \bT_{g,n}]$.
    Points of $\cX$ can be viewed as stable $G$-marked 
curves, where by a $G$-marking we understand 
a $G$-orbit in the set of all markings on a curve.
The orbifold $\left[G\bs\bT_{g,n}\right]$ allows to introduce, \emph{a
  posteriori},   a notion of a family of complex $G$-marked stable
curves.  By definition, this is a family 
of curves induced from the tautological family 
$\pi: \cX \to \left[G\bs \bT_{g,n}\right]$ via a map 
$S\to [G\bs \bT_{g,n}]$. For families of smooth curves, this
notion coincides with the one 
given  by Grothendieck in~\cite{groth}.
It would be nice to have an \emph{a priori} notion of such a family 
(defined over $\Z$ or at least over $\Q$) which would identify
$\left[G\bs\bT_{g,n}\right]$ with the (analytification of the)
corresponding moduli stack.  

Stable $G$-marked curves can be thought of as curves with generalized
level structures.  Indeed, for certain choices of the group $G$ this notion 
gives Prym level structures considered by Looijenga in~\cite{Loo}. 
Let $\wt{S}\to S$ be the universal Prym cover of  a compact oriented
 surface $S$ of genus $g$, i.e.\ a Galois covering
whose Galois group $H$ is the quotient of $\pi_1(S)$ 
by the normal subgroup generated by the squares of all elements of
$\pi_1(S)$.\footnote{Of course $H$ is isomorphic to $H_1(S,\Z/2)$} 
Let $G=\Gamma_{g,\chs{k}}$ be the subgroup of elements of
$\Gamma_{g,0}$  whose 
   lifts to $\wt{S}$ act on $H_1(\wt{S},\Z/k)$ as elements of $H$. 
The quotient $\cM_{g,{\chs{k}}}=G\bs\cT_g$ is the moduli space of smooth 
curves of genus $g$ with a level-$k$  Prym structure.
In~\cite{Loo} Looijenga studied the normalization  
$\bM_{g,\chs{k}}$ of the moduli space $\bM_g$ in the field of 
meromorphic functions on 
$\cM_{g,\chs{k}}$.
The main result of~\cite{Loo} 
is that, when $k$ is even and $k\geq 6$, the space 
$\bM_{g,\chs{k}}$ is smooth.

We use Looijenga's
result in our proof of part (iii) of the main theorem.
To do this we need to refine it in two ways. 
First, we generalize Looijenga's theorem to Riemann surfaces with punctures 
and the corresponding subgroup 
$G=\Gamma_{g,n,\chs{k}}$ of the modular group $\Gamma_{g,n}$.
Second, we show that when $k$ is even and $k \geq 6$
the orbifold $[G \bs \bT_{g,n}] $ is a complex manifold  
and, therefore, coincides with the Looijenga's space 
$\bM_{g,n,\chs{k}}$.
This provides a modular description of the  space $\bM_{g,n,\chs{k}}$.
   We will use this fact in Section~\ref{sec:projection}  to construct
canonical maps  $\pi_G:\bT_{g,n}\to\relax[G\bs\bT_{g,n}].$

\

The most natural way to 
introduce an orbifold structure on a topological space   
is to describe it as the moduli space of some geometric objects.
In the lack of a modular description,\footnote{V.~Braungardt in his thesis~\cite{brau} 
(see also~\cite{hs}) introduced a concept of a locally complex ringed space and  
proved that $\bT$ can be equipped with such structure which
is universal in a certain sense. However, the quotient stacks
$G\bs \bT$ produced this way
are \emph{non-separated} and therefore are 
very different from our complex orbifolds $[G\bs \bT]$.
We thank the referee for bringing these references to our attention.
}
we had to look for alternative ways.
We construct an orbifold structure on $G\bs \bT_{g,n}$ using 
our formalism of orbifold charts developed in Section~\ref{sec:satake}.

The traditional  approach of Satake~\cite{satake} works only for
effective orbifolds and is insufficient for our purposes.
We generalize Satake's description of orbifolds in terms of charts and
atlases to include non-effective orbifolds. 
We show that the resulting notion is equivalent by its
expressive power to the ``modern'' approaches to orbifolds 
based on the language of stacks and \'etale groupoids. 

In particular, our construction in Section~\ref{orbifold-from-atlas}
which associates a stack to an orbifold atlas 
is analogous to the well-known realization (due to Satake)
of an effective $C^\infty$-orbifold as a quotient of a manifold by
a compact Lie group.  
Our result, however, holds also for non-effective and complex
orbifolds. Even in the effective $C^\infty$-case our construction is
functorial and does not use partitions of unity. It is defined by a
universal property which is very convenient when dealing with
morphisms from an orbifold. 

\

Our proof of the main theorem uses yet another technical result
which may be of independent interest. This is Theorem~\ref{thm:gaga}
on analytification of some algebraic moduli stacks. Namely, we prove
that the analytifications of the stacks $\bsM_{g,n}$ and $\Adm_{g,n,d}$
represent the corresponding moduli functors in the complex analytic
category.

\subsection{The case $g=1,\ n=1$ }
We will illustrate the main theorem 
on the simplest interesting case 
of the one-punctured torus, i.e.\ when $g=1$ and $n=1$.

The moduli space $\cM_{1,1}$ of one-dimensional complex tori with one
mar\-ked point can be viewed as the space of lattices in $\R^2$ up to
similarity. 
Marking of a torus corresponds to a choice of a basis of the lattice.
Therefore the \TS\ $\cT_{1,1}$  is the set of similarity classes of
pairs of non-collinear vectors in $\R^2$  which can be identified
with the upper half-plane. 
The boundary of the compactified moduli space
$\bM_{1,1}$ consists of the single point corresponding to the 
degenerate elliptic curve $C$ (a pinched torus). 
Therefore, markings of $C$ correspond to isotopy classes 
of simple closed paths on the standard torus $S$.
So the boundary of $\bT_{1,1}$ can be identified with  
$\P^1(\Q)=Q\cup\{\infty\}$ (viewed as the set of pairs of relatively 
prime integers $(p,q)$ up to a common factor $\pm 1$). 
The base of the topology of $\bT_{1,1}$ near a boundary point 
is given by the collection of open disks tangent to the real line at
this point (plus the point itself). 

The \teich\ modular group $\Gamma_{1,1}$ is isomorphic to $\mathrm{SL}_2(\Z)$ 
and the quotient $\Gamma_{1,1} \bs \bT_{1,1}$  is the orbifold
$\bM_{1,1}$ whose underlying space is the Riemann sphere $\P^1(\C)$.
The quotient of $\bT_{1,1}$ by a finite index subgroup of $\mathrm{SL}_2(\Z)$ 
is a finite ramified covering of $\P^1(\C)$ and, therefore, has a 
canonical structure of a compact Riemann surface.
We will  show that analogous results hold  for arbitrary $g$ and $n$.

\subsection{Detailed description of the paper}
\label{sec:guide}

A significant part of the paper deals with general questions 
of orbifold theory. 

In algebraic geometry, the language of stacks is
the most adequate for moduli problems. 
We find it useful for studying orbifolds in other settings as well.

In Section~\ref{sec:orbi}  we present two different approaches to 
defining a 2-category of orbifolds: one based on stacks, the other on groupoids.  
Whereas the 2-category structure on stacks is standard,
the 2-category structure on groupoids is not the one that first comes
to mind.

We define an orbifold as a stack of geometric origin,
which means that it is equivalent to the stack associated to a
separated \'etale groupoid.%
\footnote{We cannot simply mimic the standard definition of an algebraic
stack (see~\cite{LMB}), since the categories of smooth or complex
manifolds do not have arbitrary fiber products which are needed to 
define representable morphisms.}

We prove  that the functor associating a stack to a groupoid
gives an equivalence between 2-categories
of separated \'etale groupoids and orbifolds.

Since $\Sp$ may not have arbitrary fiber products, 
we cannot define representable morphisms for general $\Sp$-stacks
and so it is impossible to define $\Sp$-orbifolds 
simply by modifying the standard definition of 
algebraic stacks (see~\cite{LMB}).  Instead we define an $\Sp$-orbifold 
as an $\Sp$-stack equivalent to the stack associated to an  $\Sp$-groupoid.
In the case when $\Sp$ is the category of schemes this approach gives
separated Deligne-Mumford stacks.

Our treatment of these questions is similar but not identical to 
the works of Metzler~\cite{metzler}, Noohi~\cite{noohi} and
Behrend-Xu~\cite{berxu}. 
The most significant difference between these papers and our approach 
is that we work in the category of smooth manifolds 
where general fiber products do not exist.
For this reason, we develop the theory so that only
fiber products along \'etale  morphisms are used.

In this section we also introduce
gerbes over orbifolds which 
will be used later in Section~\ref{sec:teich-vs-adm}. 

In Section~\ref{sec:satake} we present a more 
traditional approach to orbifolds. It is based on the notion of (generalized)
orbifold charts.  For effective orbifolds this approach 
goes back to the original definition of Satake~\cite{satake}.  
We generalize Satake atlases to include non-effective orbifolds. 
The main result here is a construction of an orbifold
from an atlas of generalized orbifold charts.

This gives us a flexibility to  use any of the three 
languages (of orbifold charts, groupoids or stacks)
depending on the circumstances.
In particular, the orbifold structure on the quotient $G\bs\bT_{g,n}$ 
will be given using a generalized orbifold atlas.

\

Throughout the paper we work with the  moduli stacks of stable complex
curves and admissible coverings in the complex-analytic category. 
Therefore we need to know that the analytification of the algebraic
Deligne-Mumford stacks
$\bsM$ and  $\Adm$ represent the corresponding functors (of families
of nodal Riemann surfaces and of families of admissible coverings) in
the analytic category. This is proved in Section~\ref{gaga}.

\

To construct an orbifold atlas for the quotient $G\bs\bT_{g,n}$, we  
start with an orbifold atlas for 
     $$\bM=\bM_{g,n}=\Gamma_{g,n}\bs\bT_{g,n}$$
and then construct corresponding charts upstairs on 
$G\bs\bT_{g,n}$. 
The existence of an orbifold atlas for $\bM$ follows from the
smoothness of the moduli stack $\bsM$,
but to be able to lift the charts to  $G\bs\bT_{g,n}$ 
we need an atlas on $\bsM$  whose charts satisfy some very special
properties. We call such charts \emph{quasiconformal}  
and prove that there exists a quasiconformal atlas $\bsM$ in 
Section~\ref{sec:quasiconf}. 

To construct such an atlas we use a version of the
Earle-Marden~\cite{marden} local holomorphic coordinates on the \TS\  
$\cT_{g,n}$.  
Let us recall  the definition  of these coordinates. 
Start with a maximally degenerate
stable Riemann surface $X_0$ in $\bM_{g,n}$.
This surface  is a union of $2g+n-2$ triply punctured spheres glued
together along  $3g+n-3$ pairs of punctures. 
For each of $3g+n-3$ nodes $q_i\in X_0$
choose a pair of ``coordinate'' functions $z_i,w_i$ 
that identify a neighborhood of 
$q_i$ with a neighborhood of the node of the curve
$V_i=\{(z_i, w_i)| z_i w_i =0\}\in \C^2$.
By replacing 
$V_i$ with $$V_{i,t_i}=\{(z_i, w_i) \in \C^2 | z_i w_i =t_i \}$$
we obtain a $3g+n-3$-parameter holomorphic family $X_t$ of nodal Riemann
surfaces. This gives a holomorphic map $\phi$
from a unit polydisk in $\C^{3g+3-n}$ to $\bsM$.
At every point where $\phi$ is \'etale, it defines an
orbifold chart of  $\bsM$. 
However $\phi$ may not be \'etale  everywhere
(see~\cite{cex} for a counterexample). 
To circumvent this problem and guarantee  existence of \'etale
charts at every point of $\bM$ we make a very special choice of 
local coordinates around punctures of $X_0$
(see Section~\ref{choice-of-lc}).

The existence of a quasiconformal atlas on 
$\bsM$ reflects two approaches to constructing
this space: one based on \TS s and 
another based on the theory of Deligne-Mumford stacks. 
This indicates that the appearance of quasiconformal charts here may
be not coincidental. 

\
In Section~\ref{sec:teich} we prove our main result---that the quotient of
the augmented \TS\ $\bT_{g,n}$ by a finite index subgroup $G$
of the modular group has a natural structure of a complex orbifold.
We do this by constructing an orbifold atlas on $G\bs \bT$
using the existence of a quasiconformal atlas on $\bsM$ proved in
Section~\ref{sec:quasiconf}.

In Section~\ref{sec:properties} we establish some properties
of the orbifold structure on $G\bs\bT$.
A special example of $G$-marked curves give
curves with level-$\ell$ structures. These curves
correspond to the subgroup 
$$
G=\Gamma^{(\ell)}=\Ker(\Gamma\rTo\Aut(H_1(S,\Z/\ell))).
$$

Since the orbifold structure on $[G\bs\bT]$ is given by
an ad hoc  construction and not by a universal property,
the existence of the quotient map $$\pi_G:\bT\to\relax[G\bs\bT]$$ 
is not guaranteed and requires special attention.
This is done in Section~\ref{sec:projection}.

First, for each finite index subgroup $G'\subset G$
we define a natural map of orbifolds
$\relax[G'\bs\bT]\to\relax[G\bs\bT]$.
Then, using a generalization of Looijenga's
analysis~\cite{Loo} we prove
that for any finite index subgroup $G$ of $\Gamma$ 
there exists a finite index subgroup $G'\subset G$ such that $[G'\bs\bT]$
is a complex manifold. 
Then, the quotient map $\pi_G$ can be defined as the composition
$$ \bT\rTo^{\pi_{G'}}\relax[G'\bs\bT]\rTo\relax[G\bs\bT].$$

In Section~\ref{ss:gluing-ATS} we show that the natural gluing
operations 
$$\bT_{g,n}\times \bT_{g',n'} \rTo \bT_{g+g',n+n'}
\quad \mathrm{and} \quad \bT_{g,n+2}\rTo \bT_{g+1,n}$$ 
descend to the maps of complex orbifolds
when we pass to quotients by finite index subgroups.

\

As we explained above,  given a covering $\rho:\wt{S}\to S$ of degree
$d$ (where $S$ is, as before, a compact oriented
surface of genus $g$ with $n$ punctures), 
one can naturally assign to each marked stable curve $(X,\phi)\in\bT_{g,n}$
an admissible covering $\wt{X}\to X$ of  degree $d$.

In Section~\ref{sec:teich-vs-adm} we  elevate the map~(\ref{eq:rho}) to
a morphism 
      $$v_\rho: \bT_{g,n}\to \Adm_{g,n,d}$$
of topological stacks (where $\bT_{g,n}$ has trivial stabilizers)
by composing a complex orbifold morphism
$[G\bs\bT]\to\Adm$ with the quotient map $\bT\to [G\bs\bT]$. 
Here $G$ is a symmetry group of the finite covering $\rho:\wt{S}\to S$. 
It is not a subgroup of the modular group $\Gamma$, it acts on 
$\bT$ via a natural homomorphism $\gamma:G\to\Gamma$ whose kernel and
the index of the image in $\Gamma$ are finite. 
Thus we need to deal with quotients of $\bT$ which are slightly more
general that the quotients modulo a finite index subgroup of
$\Gamma$. The quotient $[G\bs\bT]$ is a gerbe over 
$[\mathrm{Im}(\gamma)\bs\bT]$.

We also prove compatibility of the maps $v_\rho$ with the gluing operations
constructed in Section~\ref{ss:gluing-ATS}.

Even though the maps $v_\rho$ 
are morphisms of topological stacks,
we can view them as holomorphic maps by 
replacing $\bT$ with the projective system of complex orbifolds
$[G\bs\bT]$.

Finally, in Section~\ref{sec:cohom} we show how our results
about the spaces $G\bs\bT$ can be used in the study of 
stringy orbifold cohomology. This was the original motivation
for the work presented in this paper and we explain it in
a greater detail below.

\subsection{Motivation and application: orbifold cup-product}
\label{sec:motivation}
This paper is an offshoot of our project to study generalized multiplicative
orbifold cohomology theories~\cite{HV}.
The Chen-Ruan definition of the cup-product in (stringy) orbifold
cohomology (see~\cite{CR} and~\cite{FG}) uses cohomo\-logi\-cal
correction classes whose construction involves certain equivariant vector
bundles on the spaces of admissible coverings $\Adm_{g,n,d}$.
The space $\Adm_{g,n,d}$ has an open stratum corresponding to non-singular
curves and its boundary consists of products of spaces
 $\Adm_{g',n',d}$ for $g'\le g$ and $n'\le n$.
The associativity and commutativity of the orbifold cup-product are derived
in~\cite{CR} and~\cite{FG} from the fact that the fibers of these bundles on
$\Adm_{g,n,d}$ at certain boundary points are isomorphic. 
This would be immediate if the spaces $\Adm$ were connected.
However, this is far from being true.
For example, components of the open stratum of $\Adm_{g,n,d}$ 
correspond to conjugacy classes of actions of
the fundamental group of the curve on a $d$-element set.
An attempt to resolve this difficulty brought us to considering augmented 
\TS s.\footnote{The issue of non-connectivity of the spaces of admissible
coverings is also addressed in~\cite{JKK}, see Lemma 2.30. Unfortunately,
we were unable to understand the proof of this lemma.}

Let $S$ be a compact oriented
surface of genus $g$ with $n$ punctures  and  let  $\rho:\wt{S}\to S$
be a finite covering unramified outside of the punctures. 
For a stable Riemann surface $X$ of genus $g$ with $n$ punctures and a
marking  $\phi:S\to X$ we obtain an admissible covering of $X$ induced
from $\rho$. 
This leads to the map~(\ref{eq:rho}) of topological orbifolds.
Since the augmented \TS\ $\bT_{g,n}$ is contractible,
its image in $\Adm_{g,n,d}$ is connected. 
The boundary of $\bT_{g,n}$ consists of strata which are products of
$\bT_{g',n'}$ for smaller values of $g'$ and $n'$ and the maps $v_\rho$ respect
the decompositions of boundary strata of $\bT$ and $\Adm$. 
Therefore this construction allows to move verification of associativity and
commutativity of the orbifold cup-product away from a highly disconnected
space $\Adm_{g,n,d}$ to the contractible space $\bT_{g,n}$.

This is the idea of our approach to the stringy cup-product problem. 
To implement it, we have to be able to speak about continuous maps from the
space $\bT_{g,n}$ to $\Adm_{g,n,d}$. 
However, this task is non-trivial, since the former is a nasty 
topological space, whereas the latter is the space of $\C$-points of
a nice Deligne-Mumford stack. The common ground is found in 
the 2-category of complex orbifolds and is developed in this paper.

Applications of this construction to generalized stringy orbifold cohomology  
theories will be described in our forthcoming paper~\cite{HV}.

\subsection{Acknowledgments}
Parts of this paper were written during our stay at
IHES and MPIM. We are grateful to these institutions for hospitality
and for the excellent working environment.
We are grateful to O.~Gabber for pointing out the 
book~\cite{hakim}. We are grateful to Matteo Tommasini for having  found a 
gap in the original proof of Theorem~\ref{thm:realization-is-orbi}.

\section{Generalities on orbifolds}
\label{sec:orbi}

In this section we present our preferred way of working with orbifolds.
The language of algebraic stacks has long been the tool of choice for dealing
with  orbifolds in the context of algebraic geometry.  
We find it the most appropriate  in other categories as well.

In order to obtain different species of the  notion of an orbifold 
($C^\infty$, complex, algebraic), we have to choose an appropriate basic
category  $\Sp$ of manifolds or spaces 
and work with stacks over $\Sp$ (see Section~\ref{ss:list}).
The resulting notion of a $\Sp$-stack  is too general to be geometrically
meaningful in the same way as the corresponding notion  
of stack in algebraic geometry is too general. 
In order to  distinguish  geometrically meaningful stacks, 
we restrict our attention to  \'etale groupoids which have 
already been used for the description of orbifolds. 

Since $\Sp$ may not have arbitrary fiber products, 
we cannot define representable morphisms for general $\Sp$-stacks
and so it is impossible to define $\Sp$-orbifolds 
simply by modifying the standard definition of 
algebraic stacks (see~\cite{LMB}).  Instead we define an $\Sp$-orbifold 
as an $\Sp$-stack equivalent to the stack associated to an  $\Sp$-groupoid.
In the case when $\Sp$ is the category of schemes this approach gives
separated Deligne-Mumford stacks. 

Both stacks and \'etale groupoids form a 2-category.
Whereas the 2-category structure on stacks is standard, the 2-category 
structure on groupoids is not the one that first comes to mind. 
We define an orbifold as a stack of geometric origin,
which means that it is equivalent to the stack associated to a
separated \'etale groupoid.
We prove further that the functor associating a stack to a groupoid
gives an equivalence between 2-categories
of separated \'etale groupoids and orbifolds.

Our treatment is similar but not equivalent
to the recent expositions of Metzler~\cite{metzler}, Noohi~\cite{noohi} and
Behrend-Xu~\cite{berxu}. 
The most significant difference between these works and our approach 
is that our basic category of spaces is the category of 
smooth manifolds in which  general fiber products do not exist.
For this reason, we develop the theory which uses only
fiber products along \'etale  morphisms.

Having in mind the above-mentioned 2-equivalence, our approach to orbifolds 
via stacks is not so different from the 
widely  accepted approach based on groupoids. 
We prefer, however, the approach via stacks for various reasons.

In Section~\ref{sec:teich-vs-adm} we use a related
notion of a gerbe which is slightly non-standard. 
It is presented, along with other miscellanea, in~\ref{ss:misc}.  
Sheaves and vector bundles on orbifolds are defined in~\ref{sec:sheaves}.

\subsection{Basic categories of  ``spaces''}
\label{ss:list}

In what follows $\Sp$ will denote one of the following categories  of ``spaces.''
\begin{itemize}
\item[(i)]{The category of  
Hausdorff topological spaces.}
\item[(ii)]{The category of 
separated locally ringed topological spaces.}
\item[(iii)]{The category of $C^\infty$-manifolds.}
\item[(iv)]{The category of complex manifolds.}
\item[(v)]{The category of separated complex spaces.}
\item[(vi)]{The category of smooth separated schemes over a field.}
\item[(vii)]{The category of separated schemes over a base scheme.}
\end{itemize}

In each of these categories there exists  a notion of an \'etale morphism.
It is a local isomorphism for categories (i)-(v) 
and an \'etale morphism of schemes for cases (vi)-(vii).

We consider the category $\Sp$ endowed with the topology
defined by open covers  for $\Sp$  of type (i)-(v) 
and by \'etale morphisms for cases (vi)-(vii). 
We could equally consider the \'etale topology in all cases.

Note that  in all our categories of spaces there exist 
fiber products $X\times_ZY$ when one of the structure maps 
$X\to Z,\ Y\to Z$ is \'etale. 
Also,  the notion of a proper map makes sense for all these categories.

\subsection{Groupoids \emph{in} categories of spaces}

\subsubsection{Groupoids} 
\label{nota}
    A \emph{groupoid} (in category $\Ens$) is a small category whose
morphisms are invertible. Thus, a groupoid $G_\bullet=(G_0,G_1)$ 
can be specified by giving   a set $G_0$ of its objects, a set $G_1$
of its arrows, and operations:  
$$\iota:G_0\to G_1 \text{\ (identity)}, \ s,t:G_1\to G_0 \text{\ (source and 
target)}$$
 and the composition 
$$c:G_1\times_{G_0}G_1\to G_1$$
satisfying well-known axioms. 

Groupoids form a 2-category which we denote $\Grp$.  One-morphisms 
in $\Grp$ are functors between groupoids; 
2-morphisms are natural transformations between the corresponding functors. 
The 2-category $\Grp$ is strict: the composition
of functors is strictly associative.

It is sometimes convenient to  view 
groupoids as  (very special) simplicial sets. 
To  a groupoid $G_\bullet=(G_0,G_1)$ 
we assign  a simplicial set 
$(G_0,G_1,G_2,\ldots)$ whose $n$-simplices are
$n$-tuples of composable arrows in $G_1$. 
In this description
the source and the target maps $s,t:G_1\to G_0$ become the face maps $d_1$ 
and $d_0$; the composition of arrows $c$ becomes the face map 
$$d_1:G_2:=G_1\times_{G_0}G_1\to G_1.$$

\begin{dfn}
Let $\Sp$ be one of the categories of spaces from Section~\ref{ss:list}.

A \emph{groupoid  in}\footnote{In~\cite{LMB} the term {\sl
    \'espace  en groupo\"ides}  is used instead.}   
  $\Sp$  is 
a pair $G_\bullet=(G_1,G_2)$ of objects of $\Sp$, together
 with the structure maps $\iota,s,t,c$ as in~\ref{nota}, such that for
any $M\in\Sp$ the functor $\Hom(M,\underline{\ \ }\,)$ 
sends  this collection to a groupoid.
\end{dfn}
     (We assume above that the fiber product $G_2=G_1\times_{G_0}G_1$
exists in $\Sp$.)

\begin{dfn} 
Let $G_\bullet$ be a groupoid in $\Sp$.
\begin{itemize}
\item[(i)] 
$G_\bullet$ is called {\em separated} if the diagonal $(s,t):G_1\to
G_0\times G_0$  is proper. 
\item[(ii)] 
$G_\bullet$ is called  {\em \'etale} if the maps $s,t:G_1\to G_0$ are \'etale. 
\item[(iii)] 
$G_\bullet$ is called a $\Sp$-groupoid if it is \'etale and separated.
\end{itemize}
\end{dfn}

\begin{rem}
Separated groupoids are sometimes called {\sl proper groupoids}.
We follow Grothendieck's  terminology
(separatedness $=$ properness of the diagonal). 
\end{rem}

\begin{exm}
\label{transformation-g}
Let $G$ be a (discrete) group acting on $X\in\Sp$. 
The {\em transformation groupoid} $(G\bs X)_\bullet$ is defined by 
$$
(G\bs X)_0=X, \quad  (G\bs X)_1=G\times X,
$$
where the source map $s$ is the projection $G\times X \to G$
and the target $t$ is the action map $t: G\times X\to X$. 
If $G$ is finite then $(G\bs X)_\bullet$ is an $\Sp$-groupoid. 
If $\Sp$ is one of the non-algebraic categories~\ref{ss:list}.(i)-(v), then
$(G\bs X)_\bullet$ is an $\Sp$-groupoid also when
the action of $G$ on $X$ is discontinuous.
\end{exm}

\subsection{Groupoids \emph{over} categories of spaces. Stacks }
\label{ss:over}

By definition, a groupoid $G_\bullet$ in $\Sp$ represents a functor
from $\Sp$ to $\Grp$
$$
M\mapsto (\Hom(M,G_0),\Hom(M,G_1)).
$$
Since groupoids form a 2-category and not 
just a category, the notion of a functor to $\Grp$ is too rigid: 
the most natural constructions produce only 
pseudofunctors (see~\ref{ss:pseudofun})
to $\Grp$. This is why we need a relaxed  version of the notion 
of an $\Sp$-groupoid. 

The definitions~\ref{grpover} and~\ref{stack} below are special cases of 
Grothendieck's notions of a fibered category and of a stack,
see~\cite[Ch.~4]{sga1}. 
 The case when $\Sp$ is the category of schemes is described in~\cite{LMB}.

\begin{dfn}(see~\cite[Sec.~2]{LMB})
\label{grpover}
{\em A groupoid over}\footnote{%
{\sl Cat\'egorie  fibr\'ee en   groupo\"ides sur $\Sp$}
in the original terminology of \cite{LMB}.}   
 $\Sp$ is a category $\cX$ endowed with a functor $\pi:\cX\to\Sp$ such that 
\begin{itemize}
\item[(i)]For any $\alpha:U\to V$ in $\Sp$ and $x\in\cX$ with $\pi(x)=V$
there exists $a:y\to x$ such that $\pi(a)=\alpha$.

\item[(ii)] For any pair of morphisms $a:y\to x,\ b:z\to x$ in $\cX$
any map $\gamma:\pi(y)\to \pi(z)$ satisfying $\pi(a)=\pi(b)\gamma$
there exists a unique $c:y\to z$ such that $a=bc$ and $\pi(c)=\gamma$.
\end{itemize}
\end{dfn}
Groupoids over $\Sp$ form a 2-category denoted $\Grp/\Sp$.
A 1-morphism from $\pi:\cX\to\Sp$ to $\pi':\cX'\to\Sp$ is 
a functor $f:\cX\to\cX'$ strictly commuting with $\pi,\pi'$. 
A 2-morphism $\theta:f\to g$  between 
$f,g:\cX\to\cX'$ is  a natural transformation  
that sends $x\in\cX$ to a morphism $\theta(x):f(x)\to g(x)$ 
over $\id_{\pi(x)}$ in $\cX'$.

\subsubsection{Pseudofunctors. Cleavage}
\label{ss:pseudofun}
Let $\pi:\cX\to\Sp$ be a groupoid over $\Sp$. The fibers
$$\cX_U:=\pi^{-1}(U),\quad  \mathrm{for} \quad U\in\Sp,$$ 
are groupoids. For each
$\alpha:U\to V$ and for each $x\in \cX_V$ choose a lifting $a:\alpha^*(x)\to x$
of $\alpha$. This choice can be uniquely extended to a functor
$\alpha^*:\cX_V\to\cX_U$. 
Also, for each pair of composable arrows in $\Sp$,
one has a uniquely defined isomorphism
$\theta_{\alpha,\beta}:\alpha^*\beta^*\to(\beta\alpha)^*$ 
These isomorphisms $\theta$ satisfy a standard
compatibility  condition shown on diagram~(\ref{compatibility-of-thetas}).

The   above collection  $(\cX_U,\alpha^*,\theta_{\alpha,\beta})$ 
defines a {\em pseudofunctor}
$\Sp^\op\to\Grp$; it would be a genuine functor if 
$\theta_{\alpha,\beta}$ were the identity for all $\alpha$ and $\beta$.

Vice versa, given a collection of groupoids $\cX_U$ 
for   each $U\in\Sp$,  together with functors 
$$ \alpha^*:\cX_V\rTo\cX_U$$ 
for each   morphism $\alpha:U\to V$ and equivalences  
\begin{equation}
\label{thetas}
\theta_{\alpha,\beta}:\alpha^*\beta^*\to(\beta\alpha)^*
\end{equation}
for each pair of composable arrows $\alpha,\beta$ of $\Sp$,
such that the diagram
\begin{equation}
\label{compatibility-of-thetas}
\begin{diagram}
\alpha^*\beta^*\gamma^* & \rTo^{\theta_{\alpha,\beta}\cdot\gamma^*} & 
(\beta\alpha)^*\gamma^* \\
\dTo^{\alpha^*\cdot\theta_{\beta,\gamma}} & & 
\dTo^{\theta_{\beta\alpha,\gamma}} \\
\alpha^*(\gamma\beta)^* & \rTo^{\theta_{\alpha,\gamma\beta}} & 
(\gamma\beta\alpha)^*
\end{diagram}
\end{equation}
is commutative for each triple of composable arrows,
one can ``glue'' a groupoid $\pi:\cX\to\Sp$ by the formulas
\begin{align*}
\Ob\ \cX=\coprod\Ob\ \cX_U;\quad \pi(x)=U\Leftrightarrow x\in\cX_U;
\\
\Hom_{\cX}(x,y)=\coprod_{\alpha:U\to V}\Hom_{\cX_U}(x,\alpha^*(y)).
\end{align*}

\begin{dfn}
A choice of  functors $a^*:\cX_V\to\cX_U$ for each $a:U\to V$ in $\Sp$
and of compatible equivalences~(\ref{thetas}) is called {\em a cleavage}
of a groupoid $\pi:\cX\to\Sp$ (in SGA1: un clivage). 
Thus every groupoid over $\Sp$ admits a cleavage, and cleaved
groupoids over  $\Sp$ are the same as pseudofunctors
$\Sp^\op\to\Grp$. 
\end{dfn}
Any groupoid $G_\bullet$ in $\Sp$ represents a functor
$\Sp\to\Grp$. This, together with the trivial cleavage
$\theta_{\alpha,\beta}=\id$,  defines a groupoid
over $\Sp$. Thus, the notion of groupoid {\em over} $\Sp$ generalizes that 
of groupoid {\em in} $\Sp$.

\

Groupoids over $\Sp$ play the role of ``presheaves of groupoids'' 
on $\Sp$. 
Stacks can be viewed as ``sheaf of groupoids''.

\begin{dfn}
\label{stack}(see~\cite{LMB}, Sect.~2--3)
A \emph{stack} (of groupoids) $\cX$ over $\Sp$ is a groupoid over $\Sp$
satisfying the following two conditions.
\begin{itemize}
\item[(i)]
For any two objects $x,y\in\cX_U$ the assignment
$$ \alpha:V\rTo U\mapsto \Hom(\alpha^*(x),\alpha^*(y))$$
is a sheaf on $\Sp/U:=\{V\to U|V\in\Sp\}$.
\item[(ii)]
For any covering $\alpha_i:V_i\to U$ in $\Sp$ 
the groupoid $\cX_U$ is equivalent to the groupoid $\cX(\{V_i\})$
of ``local data'' whose objects are collections 
$$\left(x_i\in\cX_{V_i},
\theta_{ij}:x_i|_{V_{ij}}\rTo x_j|_{V_{ij}}\right),
\mathrm{\ where\ } V_{ij}:=V_i\times_UV_j,
$$
with compatible $\theta_{ij}$ and whose morphisms are isomorphisms of these
collections. 
\end{itemize}
\end{dfn}

Stacks over $\Sp$ form a 2-category $\MSt$ which is a strictly
full 2-subcategory of $\Grp/\Sp$.

Let $M\in\Sp$. The functor on $\Sp$ represented by $M$ is a stack. It is
called the stack represented by $M$. 
An 1-morphism $M\to\cX$ is given by an object of $\cX_M$.

\subsubsection{Associated stack} 
For every $\cX\in\Grp/\Sp$ we can associate an $\Sp$-stack $[\cX]$ 
which is constructed   in two steps.   First
one sheafifies all $\Hom$-sets and then 
``glues'' new objects from the local data as in
Definition~\ref{stack} (see the details in~\cite[Lemme 3.2]{LMB}).

The stack associated to an $\Sp$-groupoid $M_\bullet$
will be denoted $[M_\bullet]$. 
If $M_\bullet=(G\bs X)_\bullet$, we will write $[M_\bullet]=[G\bs X]$ 
rather than  $[(G\bs X)_\bullet]$.

The following lemma gives an explicit description of 
the groupoid $[M_{\bullet}]_U$.
\begin{lem}
\label{lem:explicit-asso}
Let $U\in\Sp$.
\begin{itemize} \item[(i)]
Objects of $[M_{\bullet}]_U$ are morphisms
$V_{\bullet}\to M_{\bullet}$ of groupoids where $\alpha:V\to U$ is
\'etale surjective in $\Sp$ and the groupoid $V_{\bullet}$ in $\Sp$ is 
defined by the formulas
$$
V_n=V\times_U\ldots\times_UV \text{\em \ ($n+1$  factors)}.
$$

\item [(ii)]
 Given two objects, $\alpha:V_{\bullet}\to M_{\bullet}$ and
$\alpha':V'_{\bullet}\to M_{\bullet}$, a morphism from $\alpha$ to $\alpha'$
is a morphism between two functors from $V_{\bullet}\times_UV'_{\bullet}$
to $M_{\bullet}$.
\end{itemize}
\end{lem}
This is a direct application of the construction of the associated stack,
see the proof of Lemma 3.2 in~\cite{LMB}.
\qed

Proposition~\ref{prp:hom-asso} below gives a similar explicit description
of the
groupoid $\Hom([X_\bullet],\relax[Y_\bullet])$, where 
$X_\bullet$ and $Y_\bullet$ are $\Sp$-groupoids.

\begin{dfn}
A map $f:Z_\bullet\to X_\bullet$ is called an {\em acyclic fibration} if
the map $f_0:Z_0\to X_0$ is \'etale surjective and the commutative square
$$
\begin{diagram}
Z_1 & \rTo^{f_1} & X_1 \\
\dTo^{(s,t)} & & \dTo^{(s,t)} \\
Z_0\times Z_0 & \rTo^{f_0\times f_0}& X_0\times X_0
\end{diagram}
$$
is Cartesian.\footnote{Acyclic fibrations are special cases of weak
  equivalences of groupoids,  as defined in~\cite{moepronk}.}
\end{dfn}

Let $f:Z_\bullet\to X_\bullet$ be an acyclic fibration and let 
$g:X'_\bullet\to X_\bullet$ be a morphism. Then the ``naive''
fiber product 
$$Z'_\bullet=Z_\bullet\times^\naive_{X_\bullet}X'_\bullet$$
defined by 
$$ Z'_i=Z_i\times_{X_i}X'_i,\ i=0,1,$$
gives rise to an acyclic fibration $f':Z'_\bullet\to X'_\bullet$.

\begin{prp}
\label{prp:hom-asso}
The groupoid $\Hom([X_\bullet],\left[Y_\bullet\right])$ 
has the following explicit description.
\begin{itemize}
\item[(i)]
The objects of $\Hom([X_\bullet],\left[Y_\bullet\right])$ 
are diagrams of $\Sp$-groupoids 
$$X_\bullet\lTo^s Z_\bullet\rTo^f Y_\bullet$$
where $s$ is an acyclic fibration.
\item[(ii)] A morphism from 
$X_\bullet\lTo^s Z_\bullet\rTo^f Y_\bullet$ to
$X_\bullet\lTo^{s'} Z'_\bullet\rTo^{f'} Y_\bullet$
in $\Hom([X_\bullet],\left[Y_\bullet\right])$ 
is given by a $2$-morphism between the two compositions
$f\circ\pr_1$ and $f'\circ\pr_2$ from
$Z_\bullet\times^\naive_{X_\bullet}Z'_\bullet$ to $Y_\bullet$.
\end{itemize}
\end{prp}
\begin{proof}
By the universality of the associated stack, we need just
to describe $\Hom(X_\bullet,\left[Y_\bullet\right])$.
A map $F:X_\bullet\to\relax[Y_\bullet]$ defines a composition
$\hat{F}:X_0\to\relax[Y_\bullet]$. By Lemma~\ref{lem:explicit-asso}
there exists an \'etale surjective map $s_0:Z_0\to X_0$ and a map $f_0:Z_0\to
Y_0$ so that the pair $(s_0,f_0)$ presents $\hat{F}$.

Consider the space
\begin{equation}
\label{eq:Z1}
Z_1=(Z_0\times Z_0)\times_{X_0\times X_0}X_1.
\end{equation}
This determines an acyclic fibration $s:Z_\bullet\to X_\bullet$.
We claim that $F$ canonically determines
(and is canonically determined
by) a map $f:Z_\bullet\to Y_\bullet$ extending $f_0$.

The pair $(s_0,f_0)$ gives for each $U\in\Sp$ a functor $F(U)$ which
acts on objects by
\begin{equation}
\label{image-F}
     (x:U\to X_0) \ \rMapsto \ (U\lTo U\times_{X_0}Z_0\rTo Z_0\rTo^{f_0} Y_0).
\end{equation}

Let us describe the action of $F$ on the arrows. To each arrow in 
$(X_\bullet)_U$ (that is to each map $x:U\to X_1$) $F$ assigns a
morphism between two images of $sx,\ tx:U\to X_0$ given as
in~(\ref{image-F}). The second part of Lemma~\ref{lem:explicit-asso}
says that this amounts to a map $U\to Z_1$ where $Z_1$ is defined 
by~(\ref{eq:Z1}).

This proves the first part of the proposition.
The second part is straightforward.
\end{proof}

\subsubsection{Fiber products}
\label{ss:fibre-products}

Since groupoids over $\Sp$ form a 2-category, 
we will use the following natural 2-categorical fiber   
product operation.

\defn
The \emph{fiber product} of a diagram of  1-morphisms in $\Grp/\Sp$
   $$ \cX\rTo^f\cZ\lTo^g\cY$$
is the groupoid in $\Grp/\Sp$
whose objects over $U\in\Sp$ are triples $(x,y,\theta)$, where
 $x\in\cX(U),\ y\in\cY(U)$ and $\theta:f(x) \tilde{\to} g(y) $ 
is an isomorphism; 
morphisms are compatible pairs of morphisms in $\cX$ and $\cY$.

This fiber product has the expected properties.
\begin{lem}
\label{lem:fibre-products}
Let $\cF$ be a fiber product of a diagram 
$\cX\to \cZ \leftarrow \cY$. Then
\begin{itemize}
\item[(i)] If $\cX,\ \cY,\ \cZ$ are $\Sp$-stacks then
$\cF$ is as well an $\Sp$-stack.
\item[(ii)] The associated stack $[\cF]$ is a fiber product of the
  diagram  
$$[\cX]\to [\cZ]\leftarrow [\cY].$$
\end{itemize}
\end{lem}
\begin{proof} 
The statement (i) is immediate and (ii) follows from (i). 
\end{proof}

\subsection{Orbifolds}
\label{ss:orbi}

In this section we define $\Sp$-orbifolds, where $\Sp$ is one of the
categories of spaces from Section~\ref{ss:list}.

\begin{dfn}
\label{def:orbi}
A stack $\cX$ over $\Sp$ is called an $\Sp$-\emph{orbifold} if it is
equivalent to the stack $[X_{\bullet}]$ associated  to an (\'etale separated)
$\Sp$-groupoid $X_{\bullet}$.

The full 2-subcategory of $\Sp$-orbifolds in the 2-category of $\Sp$-stacks
will be denoted by $\Sp$-$\Orbi$ (or simply $\Orbi$).
\end{dfn}

Let $\Sp$ be the category of schemes over a fixed base scheme.
The standard definition of Deligne-Mumford stack requires the diagonal 
to be quasi-compact (i.e., the preimage 
under the diagonal map
of any quasi-compact open subset is quasi-compact). 
The stacks having proper diagonal are 
called \emph{separated stacks.}

Definition of  orbifolds as equivalence classes of separated \'etale
groupoids belongs to Moerdijk. We prefer looking at a groupoid as a 
specific {\em presentation} of an $\Sp$-orbifold
in the sense of the following definition.

\begin{dfn} Let $\cX$ be an $\Sp$-orbifold. An $\Sp$-groupoid $X_\bullet$
together with a map 
$$\alpha: X_\bullet\rTo\cX$$
of groupoids over $\Sp$ is called a presentation of $\cX$ if it induces
an equivalence $[X_\bullet]\to\cX$.
\end{dfn}
A presentation $\alpha:X_\bullet\to\cX$ of an $\Sp$-orbifold $\cX$ is
uniquely determined
by a morphism $\alpha_0:X_0\to\cX$ and by an 
equivalence $\alpha_1:X_1\to X_0\times_{\cX}X_0$.

\subsection{Representable morphisms}

\begin{dfn}\label{representable}
An $\Sp$-orbifold $\cX$ is called \emph{representable} if for 
every $U\in\Sp$ the groupoid $\cX(U)$ is discrete (i.e. 
the group of automorphisms of  every object in $\cX(U)$
is trivial).
\end{dfn}

If $\Sp$ is one of the non-algebraic categories~\ref{ss:list}.(i)-(v), then
representable orbifolds are functors represented by objects of $\Sp$; 
representable orbifolds for $\Sp$ of type (vi) or (vii) 
correspond to algebraic spaces.

\begin{dfn}
A morphism $f:\cX\to \cY$ of $\Sp$-orbifolds 
   is called \emph{representable} if 
for any morphism $a:Y\to\cY$ such that  $Y\in\Sp$ and 
the fiber product $\cX\times_{\cY}Y\in\Sp\text{-}\Orbi$ exists,
this fiber product is representable. 
\end{dfn}
It is clear that  in order to check that a morphism $f:\cX\to \cY$ is
representable,  it is sufficient to prove that 
 for some presentation $Y_\bullet$ of $\cY$ the fiber product
 $\cX\times_{\cY}Y_0$ is a representable orbifold. 

\begin{prp}
Let $\cX$ be an $\Sp$-orbifold.
\begin{itemize}
\item[(i)] 
The diagonal $\cX\to\cX\times\cX$ is representable.
\item[(ii)]
 Let $f:X\to\cX$ be a morphism in $\Sp\text{-}\Orbi$ such that $X$
   belongs to $\Sp$.
Then $f$ is representable.
\end{itemize}
\end{prp}
\begin{proof}
Choose a presentation $X_\bullet$ of $\cX$. 
  Then the fiber product 
$$\cX\times_{\cX\times\cX}(X_0\times X_0)$$
is equivalent to the stack associated to 
$$
X_\bullet\times_{X_\bullet\times X_\bullet}(X_0\times X_0)=
X_0\times_{X_\bullet}X_0=X_1.
$$
This proves the first statement.

The second statement
follows from the equality
$$X\times_{\cX}X_0=\cX\times_{\cX\times\cX}(X\times X_0).$$
\end{proof}

\subsubsection{Properties of representable morphisms}
\label{sss:properties-rep}

\defn
A property (class) $P$ of morphisms in $\Sp$ is called \emph{local} if
for each Cartesian diagram
$$
\begin{diagram}
T & \rTo & X \\
\dTo^{f'} & & \dTo^f \\
Z & \rTo^g & Y
\end{diagram}
$$
the following hold:
\begin{itemize}
\item $f\in P$ and $g$ is \'etale implies that $f'\in P$ and
\item $f'\in P$ and $g$ is \'etale surjective implies that $f\in P$.
\end{itemize}

Let $P$ be a local property of morphisms in $\Sp$.
We say that a representable morphism $f:\cX\to\cY$ of $\Sp$-orbifolds 
satisfies  $P$ if its base change $f': X_0\to Y_0$ satisfies $P$,
where $Y_0\to\cY$ is obtained from a presentation
$Y_\bullet$ of $\cY$.

\

The following classes of morphisms are 
local:  smooth (= submersive), \'etale, 
\'etale surjective, proper, open embedding,
and finite (=proper with finite fibers).

\

Locality of \'etale surjective morphisms is important for the following 
description of the 2-category of orbifolds.
\begin{prp}
\label{2-equivalence}  
The 2-category $\Sp$-$\Orbi$ is equivalent to the 2-category
whose objects are $\Sp$-groupoids and
morphisms are defined as in Proposition~\ref{prp:hom-asso}.
\end{prp}
\begin{proof}
If $X_\bullet$ is a presentation of $\cX$, the corresponding map
$X_0\to\cX$ is \'etale surjective since its base change with respect to
the morphism $X_0\to\cX$ is $s:X_1\to X_0$ which is \'etale and admits 
a section.

Vice versa, assume $a:Y\to\cX$ is \'etale surjective. 
Consider
 $X_0=Y$,  and $X_1=Y\times_{\cX}Y$. The orbifold $X_1$ is representable and,
since it is \'etale over $X_0\in\Sp$, 
it belongs to $\Sp$.\footnote{This is true even in
  the case when $\Sp$ is the category of schemes,  see~\cite[6.17]{knutson}.}
Therefore we found a $\Sp$-groupoid $X_\bullet$ presenting $\cX$.

Let $X_0\to \cX$ and $Y_0\to\cX$ be \'etale surjective and let 
$Z_0=X_0\times_{\cX}Y_0$.
Let $X_\bullet,\ Y_\bullet$ and $Z_\bullet$ be the presentations of $\cX$
constructed as above from the maps $X_0\to\cX$, $Y_0\to\cX$, $Z\to\cX$.
Then the maps $Z_\bullet\to X_\bullet$ and $Z_\bullet\to Y_\bullet$
are acyclic fibrations in the sense of~\ref{ss:over}. 
\end{proof}

Thus, the 2-category $\Orbi$  can be defined in terms of  of $\Sp$-groupoids.
By using language of stacks we do not gain new ``expressive power''.
However, this language has the 
same advantages in dealing with $\Sp$-orbifolds as it has in the context
of algebraic geometry.

Proposition~\ref{2-equivalence} 
implies that the $1$-category obtained from $\Orbi$ by identifying 
isomorphic morphisms, is equivalent to the localization of the category
of $\Sp$-orbifolds by the collection of acyclic fibrations. The latter
category is what Moerdijk~\cite{moepronk} calls the category of
orbifolds.

\subsection{Some examples and constructions}
\label{ss:misc}

\subsubsection{Points and the coarse space}
\label{ss:coarse}
Let $\Sp$ be of one of the non-algebraic categories of spaces \ref{ss:list}.(i)-(v).
For an $\Sp$-orbifold $\cX$ define 
$|\cX|$ as the set of
connected components of the groupoid $\cX(\text{point})$. 
If $\cX$ is represented
by an $\Sp$-groupoid $X_{\bullet}$, one has a natural surjection
$X_0\to |\cX|$. The set $|\cX|$ endowed with the quotient topology
is called {\em the coarse space of } $\cX$.
Open subsets of $|\cX|$ are in a one-to-one correspondence 
with (equivalence classes of) open substacks of $\cX$.

If $\Sp$ is  of algebraic type (vi) or (vii),
the points of $\cX$ are defined as classes of 
equivalent objects of $\cX(\Spec\ K)$,  where $K$ is a field and the
equivalence allows to extend the 
field $K$.

The set of points $|\cX|$ is endowed with the Zariski topology, 
whose open sets are
 defined by points $|\cU|$ of open suborbifolds $\cU$ of $\cX$
(see details in~\cite{LMB}).

If $\Sp$ is the category of complex manifolds
the coarse space of an $\Sp$-orbifold
has a natural structure of  a complex space.
If  $\Sp$ is the category of schemes of finite type over a locally
Noetherian base, the coarse moduli space is an algebraic space
by a result of Keel-Mori, see~\cite{keelmori}.

\subsubsection{Global quotient}
Let $X\in\Sp$ and let  $G$ be a finite group acting on $X$. 
Then the  $\Sp$-orbifold $[G\bs X]$ associated to the transformation  
groupoid $(G\bs X)_\bullet$ (see~\ref{transformation-g})
is called {\em the global quotient orbifold}.

\subsubsection{Change of the base category}
\label{change-basecat}

Let $F:\Sp_1\to \Sp_2$ be a functor 
between two categories of spaces
from the list~\ref{ss:list} that preserves
\'etale morphisms, coverings and proper morphisms, as well as the fiber 
products.  Then the functor $F$ extends to the corresponding
categories of $\Sp$-groupoids.  Using Proposition~\ref{prp:hom-asso}
and Proposition~\ref{2-equivalence}, 
we obtain, up to $2$-equivalence, a functor
$$ F:\Sp_1\text{-}\Orbi\rTo\Sp_2\text{-}\Orbi.$$

Examples of this construction
provide various forgetful functors. 
A less obvious example is the functor assigning to a scheme of finite
type over $\C$ its analytification, see~\cite[expos\'e XII]{sga1}.

\

One of the goals of this paper is to construct maps
from the augmented \TS\ $\bT$ (which is a topological space) to 
stacks of admissible coverings $\Adm$ 
considered either as a Deligne-Mumford stack or as an orbifold 
in  the category of complex spaces (see~\ref{sss:moduli}). 
The above construction allows one to define 
the desired map as a $1$-morphism in the category of topological orbifolds. 

This may seem a weak notion; it is sufficient, however, to be able to
pull back a vector bundle on $\Adm$ to $\bT$ (see~\ref{sec:sheaves}).

\subsubsection{Moduli stacks}
\label{sss:moduli}

In this paper a few moduli stacks play an important role.

According to~\cite{dm} and~\cite{knudsen} 
the functor assigning to each scheme $S$ the groupoid
of families of stable curves of genus $g$ 
with $n$ punctures over $S$, is represented by a smooth projective
Deligne-Mumford stack.  
We denote it $\bsM_{g,n}$. Its open substack $\sM_{g,n}$ represents the 
groupoid of smooth families. 

The stack of admissible coverings $\Adm_{g,n,d}$ assigns to a scheme $S$
the groupoid of admissible coverings of degree $d$ of $S$-families of
stable curves of genus $g$  
unramified of $n$ points (see~\cite[Sect.~4]{ACV}).
This is a proper Deligne-Mumford stack having a projective
coarse moduli space, see~\cite{Mochizuki}.
Similarly, given a finite group $H$, we denote by $\Adm_{g,n}(H)$
 the stack of admissible $H$-coverings.

These are 
algebraic orbifolds (i.e.\ $\Sp$ is the category of schemes)
in the sense of our definition. We will prove in Section~\ref{gaga}
that the analytifications 
  of  the stacks $\sM$, $\bsM$, $\Adm$ represent the corresponding
groupoids of analytic families (of    stable curves or of admissible coverings). 

\subsubsection{Gerbes}
\ \\
\defn
A morphism $f:\cX\to\cY$ of $\Sp$-orbifolds is called {\em a gerbe} if
\begin{itemize}
\item $f:\cX\to\cY$ is surjective.
\item $\Delta:\cX\to\cX\times_{\cY}\cX$ is surjective.
\end{itemize}

\noindent
The first condition means that for any object $y\in\cY_U$ there exists
a covering $V\to U$ and an object $x\in\cX_V$ such that $f(x)$ is isomorphic
to $y_V$. The second condition means that given a pair of objects
$x_1,\ x_2$ in $\cX_U$ and an isomorphism $\theta:f(x_1)\to f(x_2)$ in 
$\cY_U$, there exists a covering $V\to U$ and an isomorphism 
$\eta:x_{1V}\to x_{2V}$ such that $f(\eta)=\theta$.

A gerbe $f:\cX\to\cY$ is called \emph{split} if there exists a morphism $s:\cY\to\cX$
such that the composition $f\circ s$ is equivalent to $\id_{\cY}$.

A typical example of a split gerbe is given by a finite group 
trivially acting on a manifold. Here is a non-split example. Let
$\wt{G}\to G$ be a surjective homomorphism of finite groups. Let $G$
act on a manifold $X$. Then the morphism 
$$[\wt{G}\bs X]\rTo\relax[G\bs X]$$
is a gerbe which is not necessarily split.

Note that a base change of a gerbe is a gerbe and that for any gerbe
$\cX\to\cY$ there exists a covering $\cY'\to \cY$ such that the
base change $\cX'\to\cY'$ splits. All this immediately follows from the
definition.

\subsection{Sheaves and vector bundles on orbifolds}
\label{sec:sheaves}

A sheaf (or a vector bundle) on an orbifold $\cX$ is given by a compatible
collection of sheaves (vector bundles) on each \'etale 
neighborhood $f:X\to\cX$.
Here is an appropriate definition.

\begin{dfn}\label{dfn:sheaf}
A \emph{sheaf} $F$ on an orbifold $\cX$ is a collection of the following data:
\begin{itemize}
\item 
Assignment, for each \'etale morphism $f:X\to\cX$, of a sheaf $F_f$
on $X\in\Sp$.
\item 
An isomorphism of sheaves 
      $$\theta_{f,g,\phi,\alpha}:\phi^*(F_g)\to F_f$$
for each quadruple $(f,g,\phi,\alpha)$, where $f:X\to\cX$ and $g:Y\to\cX$
are \'etale morphisms, $\phi:X\to Y$ is a morphism in $\Sp$ and
$\alpha:f\to g\circ\phi$ a morphism in $\cX(X)$.
\end{itemize}
The isomorphisms $\theta$ should be compatible with respect to
compositions, i.e.\  for any  morphisms 
$h:Z\to\cX$, $\psi:Y\to Z$, $\beta:g\to h\circ\psi$ we have
$$ \theta_1\circ\phi^*(\theta_2)=\theta_{12},$$
where $\theta_1 = \theta_{f,g,\phi,\alpha}$, $\theta_2
= \theta_{g,h,\psi,\beta}$,  
$\theta_{12} = \theta_{f,h,\psi\circ\phi,(\beta\phi)\circ\alpha}$ 
and 
$$\beta\phi:g\circ\phi\to h\circ\psi\circ\phi$$ 
is induced by $\beta$.
\end{dfn}

A vector bundle on orbifolds is defined similarly.

Let $X_\bullet$ be a presentation of $\cX$. A sheaf (resp., a vector bundle) 
$F$ on $\cX$ gives a sheaf (a vector bundle) $F_0$ on $X_0$ together with
an isomorphism $\theta:s^*(F_0)\to t^*(F_0)$ of sheaves on $X_1$ 
satisfying the cocycle condition on $X_2$. It is a standard fact that
the above assignment is an equivalence of categories. In particular,
if $\cX=[G\bs X]$, where $G$ is a finite group, and $X\in\Sp$, then
sheaves (resp., vector bundles) on $\cX$ are the same as $G$-equivariant 
sheaves (vector bundles) on $X$.

\subsubsection{Inverse image}

Given a morphism of orbifolds $f:\cX\to\cY$ one can choose presentations
$X_\bullet$ and $Y_\bullet$ of $\cX$ and $\cY$ so that $f$ lifts to
a map 
$$f_\bullet:X_\bullet\to Y_\bullet$$ 
of $\Sp$-groupoids. 
Then a sheaf (resp., a vector bundle) $F$ on $\cY$ is given by a sheaf 
(a vector bundle) 
$F_0$ on $Y_0$ together with the descent data 
(an isomorphism $\theta:s^*(F_0)\to t^*(F_0)$ satisfying
the cocycle condition). The inverse image $f_0^*(F_0)$ together with the
inverse image descent data define a sheaf (a vector bundle) on $X$.
One can easily check that the result does not depend on the choice
of presentations for $\cX$ and for $\cY$. This defines the inverse
image functor $f^*$.

\section{Satake orbifolds}
\label{sec:satake}

In this section the category of spaces $\Sp$ 
is either the category of $C^\infty$-manifolds or 
of complex manifolds.

Originally orbifolds were  defined by Satake~\cite{satake} using 
the  language of orbifold charts. This approach  works only for
effective orbifolds which is not sufficient for our purposes. 

In this section we present a generalization of
Satake's description of orbifolds in terms of charts and
atlases which also works for non-effective orbifolds
and has some other advantages.
We will show that this generalized Satake definition of $\Sp$-orbifolds
is equivalent to the one based on the language of stacks from
Section~\ref{ss:orbi}.
Even though the definition in terms of stacks is more natural, 
we need to use charts and atlases in Section~\ref{sec:teich}
in order to construct  a complex orbifold structure on quotients of
the augmented \TS. 

In Section~\ref{geography} we define Satake orbifold atlases.
Our definition is more general than the original one given by Satake
in~\cite{satake}. Our atlases, in addition to orbifolds charts,
contain information about admissible maps between the charts.
This allows us to incorporate non-effective orbifolds. We prove that
every $\Sp$-orbifold has such an atlas.

  In Section~\ref{orbifold-from-atlas}, conversely, we show that
any Satake orbifold specified by
a collection of (generalized) orbifold charts 
and admissible morphisms between them
corresponds to an $\Sp$-orbifold.
This orbifold is constructed as  $2$-colimit (in an appropriate sense) 
of the global quotients defined by the charts. 

Our method has several advantages 
over the standard construction of an equivalence class of groupoids
from a Satake orbifold (see e.g.~\cite{moepronk}).
First, we define the associated  $\Sp$-orbifold by a universal
property which is very convenient  in applications.
Second,  our procedure works with non-effective orbifolds as well as
with effective ones.
And third, the same construction works 
both for $C^\infty$ and complex orbifolds.

\subsection{Geographical approach: charts and atlases}
\label{geography}

Recall the following fact.
\begin{lem}
\label{lem:orbi-neighborhood}
Let $X_\bullet$ be an $\Sp$-groupoid. Let $x\in X_0$ and 
     $$G=\Aut(x)=\{\gamma\in X_1|s(\gamma)=t(\gamma)=x\}.$$
For any open neighborhood $V$ of $x\in X_0$ there exists 
an open neighborhood $U\subset V$ of $x$ so that
the restriction of $X_\bullet$ to $U$  is isomorphic to a quotient
groupoid $(G\bs U)_\bullet$.
\end{lem}
\begin{proof}
See proof of Theorem~4.1 in~\cite{moepronk}.
\end{proof}

This lemma implies  that any $\Sp$-orbifold can be covered by open
suborbifolds of the form $[G\bs U]$, 
where $G$ is a finite group. 

A pair $(U,G)$ as above is called an orbifold chart of $\cX$ 
(see  a more formal definition below). An orbifold 
chart $(U,G)$ is called {\em effective} if the
action of $G$ on $U$ is effective. 

In~\cite{satake} Satake defined an orbifold ($V$-manifold in his terminology)
as a topological space endowed with an atlas of effective
orbifold charts.   We will call such objects {\em effective
Satake orbifolds}.
Satake proved that every effective Satake orbifold 
can be presented as a quotient of a manifold by a compact group acting
with finite stabilizers. In~\cite{moepronk}  Moerdijk and Pronk
deduce from this that an effective Satake 
orbifold can be presented by a $C^\infty$-groupoid.

In Section~\ref{orbifold-from-atlas} 
we define general (not necessarily effective)
Satake orbifolds and construct an orbifold atlas for arbitrary orbifold
in the sense of Section~\ref{sec:orbi}. We also present a new construction
that associates to a general Satake orbifold
a $C^\infty$ (or complex) orbifold.

We begin with formal definitions of orbifold charts and atlases.

\subsubsection{Abstract orbifold charts}
     \defn
An \emph{abstract orbifold chart} is a pair $(V,H)$ where $V\in \Sp$ and $H$ is
a finite group acting on $V$. 
A \emph{morphism of abstract orbifold charts}
      $$f:(V,H)\rTo(V',H')$$
is a pair $(f_V,f_H),$ where $f_V:V\to V'$ is a morphism in 
$\Sp$  and $f_H:H\to H'$ is a group homomorphism, such that
\begin{itemize}
\item the map $f_V$ is $f_H$-equivariant and

\item the induced map of orbifolds
$$ [H\bs V]\rTo \relax[H'\bs V']$$
is an open embedding 
(see Section~\ref{sss:properties-rep}).
\end{itemize}

 An abstract orbifold chart $(V,H)$ is called \emph{effective} if $H$
 acts effectively on $V$. 

 The category of abstract orbifold charts will be denoted by $\Charts$. 
\begin{rems} \
\label{rems:open-emb}
\begin{enumerate}
\item[(i)] 
The second condition
in the definition of morphism of charts can be 
reformulated as follows. After a base change, the
map $[H\bs V]\to\relax[H'\bs V']$  turns into
$[H\bs (H'\times V)]\to V'$. 
The latter map is an open embedding if and only if 
\begin{itemize}
\item 
the kernel of the map $f_H:H\to H'$ acts freely on $V$ and 
\item
the induced map from the quotient space $H'\times^HV$ to $V'$ is an  
open embedding.
\end{itemize}
\item[(ii)]
If $f$ is a map of abstract orbifold charts, 
then the map $f_V$ is \'etale  because it is
the composition of an open embedding $V\to H'\times V$, 
the \'etale morphism 
$$H'\times V\to \relax[H\bs H'\times V]$$
and  the  open embedding described in the previous remark.
\end{enumerate}
\end{rems}

\begin{dfn}\label{dfn:satake-orbi}
Let  $X$ be a Hausdorff topological space. 

 An \emph{orbifold chart} of $X$ is a  collection $(V,H,\pi:V\to X)$
 where $(V,H)\in\Charts$  and $\pi$ is a
continuous map identifying the quotient $V/H$ with an open subset of $X$.
An orbifold chart $(V,H,\pi)$ of $X$ is 
called effective if the abstract orbifold chart $(V,H)$ is effective.

A \emph{morphism} 
$$f:(V,H,\pi)\to (V',H',\pi')$$ 
of orbifold charts of $X$
is a morphism $(f_V,f_H)$ of the abstract orbifold charts satisfying the
compatibility $\pi=\pi'\circ f_V$.

The category of orbifold charts of $X$  will be denoted  $\Charts/X$.
\end{dfn}

Note that a morphism $f:(V,H,\pi)\to(V',H',\pi')$ of effective 
orbifold charts is uniquely determined by its first component $f_V$.

Let $(V,H,\pi)$ be an orbifold chart. Any element $h\in H$ defines 
{\em the inner automorphism} $h$ of  $(V,H,\pi)$ by the formulas
$$ h_V(x)=h(x),\ h_H(g)=hgh^{-1}.$$

The effective version of  these notions
is considerably simpler due to the
following property of effective orbifold charts.

\begin{lem}
\label{MoePronkA1}
Let $f,g:(V,H,\pi)\to\relax(V',H',\pi')$ be two injective 
maps between connected effective orbifold charts. 
Then there exists $h\in H'$ so that $g=h\circ f$.
\end{lem}
\begin{proof}
See Proposition A.1 in Moerdijk-Pronk~\cite{moepronk}.
\end{proof}

The following example shows this  does not hold
in general. Let $H$ act trivially
on $V$ and let $\phi$ be a non-inner automorphism of $H$. Then the pair
$(\id_V,\phi)$ is an automorphism of $(V,H,\id_V)$ which cannot be
obtained from $(\id_V,\id_H)$ by conjugation.

As a special case of Lemma~\ref{MoePronkA1} we deduce that if a chart 
$(V,H,\pi)$ is effective, the semigroup of endomorphisms
$\End(V,H,\pi)$ identifies with $H$.

We start with the (more or less standard) definition 
of {\em effective} orbifold atlases.

\begin{dfn}
An \emph{effective orbifold atlas} of a Hausdorff topological space $X$ is a 
collection of  effective orbifold charts on $X$
covering $X$, such that for any two charts $(V',H',\pi')$ and
$(V'',H'',\pi'')$ with $x\in\pi'(V')\cap\pi''(V'')$ there exists a
chart $(V,H,\pi)$ in the collection and a pair of injective 
morphisms from  $(V,H,\pi)$ to  $(V',H',\pi')$ and  $(V'',H'',\pi'')$
respectively  so that $x\in\pi(V)$.
\end{dfn}

The notion of equivalent atlases and of the maximal atlas in the effective
case are defined in a standard way.

\begin{dfn}\label{def:effective}
A topological space $X$ with a family of equivalent effective orbifold atlases 
is called {\em an effective Satake orbifold}. 
\end{dfn}

\
Below we present a  general definition of (not necessarily
effective) a Satake orbifold.
To be able to work with non-effective atlases, we will have to 
specify the admissible morphisms between the orbifold charts
explicitly.

Our general definition reduces to~\ref{def:effective} in the effective case
(see Remark~\ref{eff:equivalence-of-def} below).

The category of the orbifold charts $A$ will satisfy the following properties.

\begin{dfn}
\label{chart-cat}
 A category $A$ is called {\em a chart} category if 
\begin{itemize}
\item For each $a\in A$ all endomorphisms of $a$ in $A$ are invertible.
\item For each $a,b\in A$ the set $\Hom_A(a,b)$ is a (may be, empty)
$\Aut(b)$-torsor.
\end{itemize}
\end{dfn}

Note that any arrow $f:a\to b$ in a chart category $A$ defines
a homomorphism 
     $$ \Aut(f):\Aut(a)\rTo \Aut(b) $$ 
uniquely characterized by the property
$$ f\circ u=\Aut(f)(u)\circ f\textrm{ for }u\in\Aut(a).$$

\begin{dfn}
\label{satake-orbifold-atlas}
An \emph{orbifold atlas}
 of a Hausdorff topological space $X$ consists of
\begin{itemize} 
\item
 A chart category $A$.
\item
A functor $c:A\to\Charts/X$  which sends
$a\in A$ to the chart
\begin{equation}
  \label{eq:atlas}
c(a)=(V_c(a),H_c(a),\pi_c(a))\in\Charts/X.  
\end{equation}

\item
A collection of isomorphisms $\iota:\Aut(a)\to H_c(a)$ 
compatible with the action of both groups on $V_c(a)$ such that
$\phi\in\Aut(a)$ induces the inner automorphism of $c(a)$ given by the
element $\iota(\phi)\in H_c(a)$.
\end{itemize}
The above data are assumed to satisfy the following properties.
\begin{itemize}
\item[(i)] The images of the charts $c(a)$ cover the whole $X$.
\item[(ii)]
For any $x\in X$ belonging to the images of two charts $c(a')$ and $c(a'')$,
there exists $a\in A$ with a pair of arrows $a\to a',\ a\to a''$, such that
$x$ belongs to the image of $c(a)$.
\end{itemize}
A \emph{morphism of orbifold atlases}  $(A,c)\to (A',c')$ is 
a fully faithful functor 
$$f:A\to A'$$ 
of the corresponding chart categories 
together with an isomorphism of 
functors 
$$c\rTo^\isom c'\circ f.$$

Two orbifold atlases are called  \emph{equivalent} if they can be
connected by a sequence of morphisms  in the above sense.
\end{dfn}
\

We will usually suppress the subscript $c$ in 
equation~(\ref{eq:atlas}) and will write simply
$$
c(a)=(V(a),H(a),\pi(a))\ \mathrm{or\ even\ }
c(a)=(V(a),H(a),\pi).
$$

\begin{rem}
\label{eff:equivalence-of-def}
Let $X$ be an effective Satake orbifold defined by a set $A$ of connected  
effective orbifold charts. Then by~\ref{MoePronkA1} the subcategory 
of $\Charts/X$ defined by the set $A$ 
of orbifold charts and all injective morphisms between them,
is a chart category. Thus, our definition~\ref{satake-orbifold-atlas} reduces
to the standard definition 
\ref{def:effective} in the effective case.
\end{rem}

\subsubsection{An orbifold atlas of an orbifold} 
\label{orb-to-satake}
\noindent {\bf Proposition.\ }
Any  $\Sp$-orbifold admits an orbifold atlas in the sense of 
Definition~\ref{satake-orbifold-atlas}. 
\begin{proof}
For an orbifold $\cX$, define an atlas category $A$ as follows
(see~\cite[6.1]{double}).  The objects of $A$  are 
triples $(V,H,\hat{\pi})$, where $(V,H)$ is an abstract 
orbifold chart and $\hat{\pi}:[H\bs V]\to\cX$ is an open embedding.
A morphism 
$$f:(V,H,\hat{\pi})\to(V',H',\hat{\pi'})$$
 is a triple  $(f_V,f_H,\theta),$
where $(f_V,f_H)$ is a morphism of abstract orbifold charts and
$\theta$ is an isomorphism between $\hat\pi$ and $\hat{\pi'}\circ\hat{f}$,
      where
$$
  \hat{f}:[H\bs V]\to\relax[H'\bs\ V']
$$
is the map of orbifolds induced by $(f_V,f_H)$. 
Let $X$ be the coarse space for $\cX$.
Define the functor
$$  c:A\to\Charts/X $$ 
by assigning to the triple $(V,H,\hat\pi)\in A$ 
the orbifold chart $(V,H,\pi)$, where $\pi$ is the composition of the
projection $V\to H\bs V$  with the map $H\bs V\to X$ induced by
$\hat\pi$. 
Let us check that  $c:A\to\Charts/X$  is an orbifold atlas for $\cX$.
According to Lemma~\ref{lem:orbi-neighborhood}, the 
conditions (i) and (ii) of Definition~\ref{satake-orbifold-atlas} are satisfied.

Now we will show that $A$ is a chart category. 
Let $a=(V,H,\hat\pi)$ be an object of $A$.
Since $\hat\pi:[H\bs V]\to\cX$ is an open embedding, 
the category $\Hom([H\bs V],\cX)$ 
is equivalent to the category $\Hom([H\bs V],\left[H\bs V\right])$. 
Thus, $\End(a)$ is isomorphic to the group of automorphisms of the
identity functor  $[H\bs V]\to\relax[H\bs V]$ and by
Lemma~\ref{lem:explicit-asso}  we have 
$$\End(a)=\Aut(a)=H.$$ 
A similar  argument
proves that $\Hom(a,b)$ is an $\Aut(b)$-torsor.
\end{proof}

\subsection{An orbifold from a Satake orbifold atlas}
\label{orbifold-from-atlas}

In this  section we will show that to each Satake orbifold there  
corresponds an orbifold in the sense of Definition~\ref{def:orbi}. 
This correspondence is natural in a sense which 
we will not try to make precise 
(because we have not introduced a 2-category structure on Satake orbifolds). 

\begin{thm}
\label{satake-groupoid}
There exists a natural construction which assigns  to a Satake 
orbifold $X$ an $\Sp$-orbifold $[X]$, such that 
an orbifold chart $(V,H,\pi)$ of $X$ gives an
open embeddings of orbifolds 
$$\hat{\pi}:[H\bs V]\rTo \relax[X].$$  
In particular, the coarse space of the orbifold $[X]$ is homeomorphic to the 
underlying space of $X$.
\end{thm}

In a certain sense, this result is a converse to
Proposition~\ref{orb-to-satake}. 
When $X$ is an effective $C^\infty$ Satake orbifold,
it follows from Theorem 4.1 of~\cite{moepronk}.

We will use Satake orbifolds in the study of quotients of \TS s 
in Section~\ref{sec:teich}. 
The construction of a complex orbifold from an orbifold atlas of
will be used in Section~\ref{sec:teich-vs-adm} 
to obtain  a map from the  augmented \TS\ $\bT_{g,n}$ to the stack of
admissible coverings $\Adm_{g,n,d}$. 

\

\emph{The proof} of Theorem~\ref{satake-groupoid} occupies the rest 
of this section.
The idea is very simple.  If a manifold $X$ is covered by 
open subsets $U_\alpha,\ \alpha\in A$,
then under some natural assumptions on $A$,  $X$ can be 
described  as the direct limit  of the collection $U_\alpha$. 
In our situation the realization $[X]$ of a Satake orbifold $X$ will
be defined as a ($2$-) colimit of its orbifold charts. The most
difficult part of this project consists of proving that the resulting
stack is an orbifold. 

\subsubsection{Direct limit of stacks}
\label{dirlimofstacks}

Recall the notion of direct limit of $\Sp$-stacks.
Since the stacks form a 2-category, it makes more sense to talk about
weak functors into the (2-) category of stacks. Let $I$ be a category.
A functor 
     $$F:I\rTo\MSt$$
 is defined as a fibered category 
$$\pi:\cF\rTo I^\op\times\Sp$$
 such that for each $i\in I$ the fiber
$$\cF_i\rTo \Sp$$
is an $\Sp$-stack.

Following~\cite[VI.6.3]{sga4}  we define
$\dirLim(F)$ as the localization of 
the total category $\cF$ with respect to the morphisms of the type $\alpha^*$
where $\alpha\in\Mor(I)$. The resulting localization is still a category over
$\Sp$. 
As we show below, $\dirLim(F)$ is fibered over $\Sp$;
we denote the associated $\Sp$-stack by $\dirtLim(F)$.
The category $\dirtLim(F)$ can be described in terms of pseudofunctors
as follows.
The composition 
$$\Pi=\pr_2\circ\pi:\cF\to\Sp$$
is a fibered category. Choosing a cleavage, we get 
a pseudofunctor $\Sp^\op\to\Cat$. 
Its composition with the 
total localization functor $\Cat \to \Grp$
gives a pseudofunctor $\Sp^\op\to\Grp$, i.e., 
a cleaved  groupoid over $\Sp$
which is denoted by $\dirLim(F)$. 

In general, we have no reason to expect that 
the direct limit $\dirtLim(F)$ is an orbifold, 
even if the fibers  $\cF_i$ are all $\Sp$-orbifolds.
This happens, however, in some cases.  
Let, for example, $X\in\Sp$ and let a finite group $G$ act on $X$. These
data define an obvious functor 
$$\wt{X}:BG\to\MSt$$
from the classifying groupoid $BG$ of $G$ to $\Sp$ and, therefore, to 
$\Sp$-stacks. The direct limit $\dirLim(\wt{X})$ is the functor from $\Sp$
to groupoids represented by the quotient groupoid $(G\bs X)_\bullet$; 
the associated stack $\dirtLim(\wt{X})$ is $[G\bs X]$.

\subsubsection{Constructing orbifold from an atlas}
\label{def:realization}

Let $X$ be a Hausdorff topological space
and let $c:A\to\Charts/X$ be an orbifold atlas for $X$,
where $A$ is a chart category. The composition
\begin{equation}
\label{c-composition}
 \cV: A\rTo^c\Charts/X\rTo^{\pr_1}\Sp\rTo\MSt
\end{equation}
assigns to $a\in A$ the $\Sp$-stack represented by $V(a)\in\Sp$.

We define the \emph{realization} $[X]$ as $\dirtLim(\cV)$.
This is an $\Sp$-stack  which depends 
on $X$ and on the choice of the atlas  $c:A\to\Charts/X$ of $X$.

We will prove later that $[X]$ is essentially independent of the choice
of the atlas.

\

Let $c:A\to\Charts/X$ be an atlas and let $I$ be a finite subset of $\Ob(A)$.
Define $A_I$ as the full subcategory of $A$ which consists
of objects $a\in A$ satisfying the condition
$$ \Hom(a,i)\ne\emptyset \textrm{ for each } i\in I.$$
Define $X_I$ as the intersection of the images
of the charts corresponding to the elements of $I$. For each
$a\in A_I$ the chart $c(a)$ has its image in $X_I$. 
This  gives a functor 
$$c_I:A_I\to\Charts(X_I).$$
 
\begin{lem}
\label{atlas-I}
The pair $(A_I,c_I)$ is an orbifold atlas of $X_I$.
\end{lem}
\begin{proof}
The only thing we have to check
is that the images of the charts of $A_I$ cover the whole $X_I$. 
This  follows from
property~\ref{satake-orbifold-atlas} (ii) of orbifold 
atlases by induction on the cardinality of $I$.
\end{proof}

\begin{dfn}\label{covering}
A collection of arrows $f_i:a_i\to b$ 
in $A$ with the same target  $b\in A$ 
is called {\em a covering} if the 
maps $V(f_i):V(a_i)\to V(b)$  cover $V(b)$.
\end{dfn}

\begin{lem}
\label{covering-B}
Let $B$ be a subset of $\Ob(A)$ and let $a\in A$.  Assume that 
     $$\Image\ c(a)\subset \bigcup_{b\in B} \Image\ c(b).$$ 
Then the collection of maps $f:x\to a$ 
from elements $x$ which can be mapped
 into an element of $B$ is a covering.
\end{lem}
\begin{proof} Let $v\in V(a)$ and let $x=\pi(v)$. There exists $b\in B$ such 
that $x\in\Image\ c(b)$. Then by 
property~\ref{satake-orbifold-atlas} (ii) of orbifold atlases there exists 
$d\in A$, a pair of maps $\alpha:d\to a$ and $\beta:d\to b$, and 
$w\in V(d)$ such that $x=\pi(w)$. This implies that the elements 
$v$ and $V(\alpha)(w)$ belong to the same $H(a)$-orbit. This means that
by replacing  $\alpha$ with its $H(a)$-conjugate, we can assure that
$v=V(\alpha)(w)$. 
\end{proof}

Now, we are ready to prove that the realization 
does not depend on the choice of an atlas.

\begin{prp}
\label{prp:independence}
Let 
$$   B\rTo A\rTo^c\Atlas/X $$
be a morphism of orbifold atlases of $X$.
Suppose, as above, that 
$$ \cV:A\to\Sp\to\MSt $$ 
sends $a\in A$ to $V(a)$ and let $\cV_B$
be the restriction  of  $\cV$ to $B$. 

Then the map of the realizations
$$ \dirtLim(\cV_B)\rTo\dirtLim(\cV)$$
is an equivalence.
\end{prp}
\begin{proof}
We define a subatlas $\bar{B}$ in $A$ by the formula
$$ \bar{B}=\{a\in A|\exists b\in B: \Hom(a,b)\ne\emptyset\}.  $$
Since $\bar{B}$ contains the image of $B$ in $A$, the morphism of atlases
is the composition
$$   B\to \bar{B}\to A\rTo^c\Atlas/X.  $$

We will prove that the functors 
$$  B\to \bar{B} \mathrm{\ \ and\ \ } \bar{B}\to A  $$
induce an equivalence of the corresponding direct limits.

For each $a\in\bar{B}$ choose an arrow $f:a\to b$ with $b\in B$.
Given a compatible collection of maps $V(b)\to\cX$ for $b\in B$, we will get
a collection of maps $V(a)\to \cX$ for $a\in\bar{B}$. To prove 
that it is  automatically compatible, we will check that if $g:a\to
b'$ is another arrow with $b'\in B$, the compositions 
$$
V(a)\to V(b)\to\cX \mathrm{\ and\ } 
V(a)\to V(b')\to\cX
$$ 
are canonically isomorphic.

We claim that $a$ is covered by the arrows $u:x\to a$ 
which can be placed in a commutative diagram~(\ref{covering-bb}),
where  $b''\in B$.

\begin{equation}
\label{covering-bb}
\begin{diagram}
   x & & & \\
\dTo^u & \rdTo(3,2) & & \\
a & &  & b'' \\
\dTo^f & \rdTo(3,2)^{g\quad\quad} &\ldTo(3,2)  & \dTo \\
b & &  & b'
\end{diagram}
\end{equation}

Lemma~\ref{atlas-I}, for $I=\{b,b'\}$, together with 
Lemma~\ref{covering-B} applied to the atlas $\bar{B}_I$ and to the subset 
$\Ob\ B_I$  guarantee that $a$ is covered by arrows $u:x\to a$,
such that $x$ can be mapped to an element $b''\in B$ which,
in turn, can be sent  to $b$ and to $b'$. Since $A$ is a chart
category, the arrows $b''\to b$ and  $b''\to b'$ can be chosen so that
the diagram~(\ref{covering-bb}) becomes commutative.

The  equivalence between the compositions
$$ V(x)\rTo V(b'')\rTo V(b)\rTo\cX$$
and
$$V(x)\rTo V(b'')\rTo V(b')\rTo\cX$$
is now immediate. 
  Since the maps $u:x\to a$ cover $a$, this gives the required
equivalence between the compositions 
$$V(a)\to V(b)\to\cX  \text{\ \ and \ \ }
\ V(a)\to V(b')\to\cX\ .$$

Now let us prove a similar statement for the functor $\bar{B}\to A$.
Any object $a\in A$  can be covered by objects of $\bar{B}$. Thus,
for any stack $\cX$ a map $V(a)\to\cX$ is uniquely defined by a compatible
collection of maps $\alpha_f:V(b_f)\to\cX$ for each $f:b_f\to a$ with 
$b_f\in\bar{B}$. Given a compatible collection of maps $V(b)\to\cX$
for $b\in\bar{B}$, the collection of maps $V(a)\to\cX$ so defined will
be automatically compatible by Lemma~\ref{covering-B}. 
This proves that the functor
$\dirtLim\cV_{\bar{B}}\to \dirtLim\cV$ is an equivalence.
\end{proof}

\begin{crl}
\label{crl:globalchart}
Assume that a Satake  orbifold $X$ admits a global chart, i.e. a chart
$(V,H,\pi)$ 
with surjective $\pi:V\to X$. Then the realization $[X]$ is naturally
equivalent to $[H\bs V]$.
\end{crl}
\begin{proof}
Let $c:A\to\Charts/X$ be an orbifold atlas of $X$ and let $a\in A$
define a global chart $c(a)=(V,H,\pi)$. 
The embedding $BH\to A$  which sends the unique
object of $BH$ to $a$ gives an embedding of orbifold
atlases. The realization $\dirtLim\cV_{BH}$ is precisely  $[H\bs V]$.
\end{proof}

\subsubsection{The induced atlas}
\label{induced-atlas}

Let $c:A\to\Charts/X$ be an orbifold atlas of $X$ and let $U$ be an open 
subset of $X$. The \emph{induced orbifold atlas}  of $U$ is the functor
$$ c_U:A\to\Charts/U $$
defined as the composition of $c:A\to\Charts/X$ 
with the restriction functor 
$$\Charts/X\to\Charts/U$$ 
which sends a chart $(V,H,\pi)$ to  $(\pi^{-1}(U),H,\pi|_{\pi^{-1}(U)})$. 

By definition, a canonical morphism of the realizations 
$[U]\to \relax[X]$ is defined.

For example, if $U$ is the image in $X$ of a chart $c(a)=(V,H,\pi)$,
then by~\ref{crl:globalchart} the realization $[U]$ is equivalent to
$[H\bs V]$. 

\

Later we will need the following
explicit description of  fiber products in $\Charts/X$.
\begin{lem}
\label{fiberproduct_charts}
Let $f_1$ and $f_2$ be two morphisms
of orbifold charts
$$  f_i:(V_i,H_i,\pi_i)\rTo (V,H,\pi),\ i=1,2.  $$ 
Define the triple $(V_{12},H_{12},\pi_{12})$
by the formulas
$$
 V_{12}=V_1\times_VV_2,\ H_{12}=H_1\times_HH_2, \ \pi_{12}=\pi_1\circ\pr_1.
$$
Then the projections 
$$\pr_i:(V_{12},H_{12},\pi_{12})\to(V_i,H_i,\pi_i),       \ i=1,2,  $$ 
are morphisms of orbifold charts.
\end{lem}
\begin{proof}
 The morphisms $f_i:V_i\to V$ are \'etale, therefore, the fiber product 
$V_{12}$ exists in $\Sp$ and the projections
$\pr_i:V_{12}\to V_i$ are \'etale. Thus we have to verify
that the maps $[H_{12}\bs V_{12}]\to\relax[H_i\bs V_i]$ 
are open embeddings.  According to 
Remark~\ref{rems:open-emb}.(i), we need to check
two conditions. The first one,
that the kernel of the map $\pr_1:H_{12}\to H_1$ acts freely on $V_{12}$,
immediately follows  from the similar property of the map $f_2$. 
The second  is that the map 
$$
\alpha:H_1\times^{H_{12}}V_{12}\to V_1
$$
 is an open embedding.  This map is
  \'etale since $H_{12}$ acts freely on $H_1\times V_{12}$ and $\pr_1$
is \'etale. Thus, it is enough to check that $\alpha$ is injective. 

Assume that  we have
$$ h_i\in H_1,\quad (x_i,y_i)\in V_{12}\text{ for }i=1,2  $$
such that $h_1(x_1)=h_2(x_2)$.
We  need to find $(u,v)\in H_{12}$ such that
$$ h_1=h_2u \mathrm{\ and \ } (x_2,y_2)=(u,v)(x_1,y_1).  $$
Since we must set $u=h_2^{-1}h_1$, we need to show that 
there exists $v\in H_2$ satisfying the conditions
\begin{equation}
\label{eq:check1}
 f_2(v)=f_1(h_2^{-1}h_1)\mathrm{\ and\ } y_2=vy_1.
\end{equation}
Applying $f_1$ to the equality $h_1(x_1)=h_2(x_2)$ one obtains
$$ 
  f_1(h_1)f_2(y_1)=f_1(h_1)f_1(x_1)=f_1(h_2)f_1(x_2)=f_1(h_2)f_2(y_2).
$$
Since $f_2$ is a morphism of orbifold charts,  the above equation
implies the existence of $v\in H_2$  which satisfies 
$$ f_1(h_1)=f_1(h_2)f_2(v),\quad y_2=vy_1.$$
This is equivalent to equation~(\ref{eq:check1}).
\end{proof}

\subsubsection{Intersection of charts}
\label{intersection-of-charts}

Now we will describe an operation that assigns 
to every pair of objects in the chart category $A$
an orbifold chart. For manifolds it corresponds to the usual operation
of intersection of charts. 

\

\noindent {\bf Proposition-Definition.\ }
Let $c:A\to\Charts/X$ be an orbifold atlas.
There exists a natural operation that assigns to a pair of objects
$a_1, a_2 \in A$ a chart $c(a_1\cap a_2)$ together with morphisms
   $$  \pr_i:c(a_1\cap a_2)\to c(a_i),\ i=1,2,  $$
which  satisfy the following universal property.

For  each pair of morphisms 
      $$\alpha_1:b\to a_1,\ \alpha_2:b\to a_2$$
there exists a canonical morphism of charts 
     $$  c(\alpha_1,\alpha_2):c(b)\to c(a_1\cap a_2)  $$
such that 
      $$  \pr_i \circ c(\alpha_1,\alpha_2) = V(\alpha_i), \ i=1,2. $$
\ \\[2pt]
We call the chart
   $$ c(a_1\cap a_2)=(V(a_1\cap a_2),H(a_1\cap a_2),\pi(a_1\cap a_2))$$
the \emph{intersection} of $a_1$ and $a_2$.

\

This operation resembles a direct product operation, but it is not 
a direct product. We call it intersection because of the lack of a
more appropriate  term.

\

\begin{proof}
Let $c_i:=c(a_i)=(V_i,H_i,\pi_i),\ i=1,2$ be two orbifold charts. 
We will construct a chart $(V_{12},H_{12},\pi_{12})$ together with a
pair of maps  
$$\pr_i:(V_{12},H_{12},\pi_{12})\rTo(V_i,H_i,\pi_i)$$ 
satisfying the universal property.

Consider an open subset  $U=\pi_1(V_1)\cap\pi_2(V_2)$ of $X$. 
We can view  it as a Satake orbifold with the induced 
atlas of  orbifold charts, see~\ref{induced-atlas}.
The charts $(U_i,H_i,\pi_i)$, where $U_i=\pi_i^{-1}(U)$,
have the same image $U$ in $X$. Therefore, by 
Corollary~\ref{crl:globalchart}, the maps 
$[H_1\bs U_1]\to \relax[U]$ and $[H_2\bs U_2]\to\relax[U]$ 
are equivalences. 
Consider the 2-fiber product 
$$V_{12}=U_1\times_{\left[U\right]}U_2,$$
where the maps $U_i\to\relax[U]$ are defined as the compositions
$$ U_i\to \relax[H_i\bs U_i]\rTo \relax[U].$$

Since $V_{12}$ is a representable $\Sp$-orbifold,  we may assume that $V_{12}\in\Sp$. 

The group $H_2$ acts freely on $V_{12}$ with the 
quotient $U_1$. Similarly, the group $H_1$ acts freely on $V_{12}$
with the  quotient $U_2$. These actions commute and define an action
of $H_{12}=H_1\times H_2$ on $V_{12}$.  
Thus we have constructed an orbifold chart $(V_{12},H_{12},\pi_{12})$
with $\pi_{12}$ being the composition of the projection to $U_1$ and
$\pi_1$. We claim this is the chart we need.

Let $c=(W,H,\rho)$ be a chart. A map from $c$ to $c(a_1\cap a_2)$ is given
by a pair of maps $W\to V_{12}$ and $H\to H_1\times H_2$. 
It is uniquely defined by a triple $(f_1,f_2,\theta)$ where 
         $$f_i:c\to (U_i,H_i,\pi_i)$$  
are morphisms of charts and  $\theta: \psi_1\to \psi_2$ 
is an isomorphism between the two compositions
\begin{equation}
\label{eq:iso}
     \psi_i: W\rTo U_i\rTo \relax[H_i\bs U_i]\rTo \relax[U],\quad i=1,2.
\end{equation}

Let $\alpha_i:b\to a_i,\quad i=1,2,$ be two arrows in $A$.
We have  a pair of morphisms  
      $$ c(\alpha_i):c(b)=(W,H,\rho)\to(V_i,H_i,\pi_i). $$
Since $\rho(W)\subseteq U=\pi_1(V_1)\cap\pi_2(V_2)$, the morphisms
$c(\alpha_i)$ factor through $(U_i,H_i,\pi_i)$. 
By definition of realization, each of the compositions
\begin{equation}
\label{eq:iso1}
     \psi_1: W\rTo U_1\rTo \relax[H_1\bs U_1]\rTo \relax[U]
\end{equation}
and
\begin{equation}
\label{eq:iso2}
     \psi_2: W\rTo U_2\rTo \relax[H_2\bs U_2]\rTo \relax[U]
\end{equation}
is canonically isomorphic to the composition
$$ W\rTo \relax[\rho(W)]\rTo \relax[U].$$
Therefore, one has a canonical
choice of isomorphism between (\ref{eq:iso1}) and (\ref{eq:iso2}), so that
a map $c(b)\to c(a_1\cap a_2)$ is defined.
\end{proof}

\

Now we are ready to prove that the realization $[X]$ of a Satake
orbifold is an orbifold in the sense of Definition~\ref{def:orbi}.
This will be the last step in the proof of Theorem~\ref{satake-groupoid}.

\begin{thm}
\label{thm:realization-is-orbi}
Let $X$ be a Satake orbifold. Then its realization $[X]$ is an orbifold.
\end{thm}
\begin{proof}
Let $c: A\to\Charts/X$  be an orbifold atlas of $X$,
$\cV:A\to\Sp$ be the obvious functor assigning $V(a)$ to $a\in A$.
We wish to present the stack
$\dirtLim(\cV)$ by an $\Sp$-groupoid. The problem here is in the fact
that the definition of $\dirLim(\cV)$ includes localization of the total
category which may destroy representability.

Fortunately, the intersection operation~\ref{intersection-of-charts}
allows one to present the localization in a very explicit way.

As it was done  in~\ref{dirlimofstacks}, we 
will interpret the functor 
$$ \cV:A\rTo\Sp,\ a\mapsto V(a) $$
as a category $\cX$ fibered over $A^\op\times\Sp$.
The fibers $\cX_{a,M}$ at $(a,M)\in A^\op\times\Sp$ are discrete;
one has $\cX_{a,M}=\Hom(M,V(a))$ for connected $M\in\Sp$.

The category $\cX$ considered as a fibered category over $\Sp$ comes
from a category \emph{in} $\Sp$ (which we denote by the same letter) defined
as follows
\begin{itemize}  
\item The objects of $\cX$ is $\displaystyle\coprod_{a\in A} V(a)$.

\item The morphisms of $\cX$ is
$\displaystyle\coprod_{\alpha\in\Mor(A)} V(s(\alpha))$

\item The map $s\!:\Mor(\cX)\to\Ob(\cX)$ restricted to the $\alpha$-component
is $\id_{V(s(\alpha))}$.
\item The map $t:\Mor(\cX)\to\Ob(\cX)$ restricted to the $\alpha$-component
is $V(\alpha)$.
\end{itemize}
Here, as before, $s\alpha$ 
and $t\alpha$ denote the source and the target of an arrow $\alpha$.

Now we will  present an \'etale groupoid $\cY$ in $\Sp$ such that the
corresponding  fibered category over $\Sp$ is obtained from $\cX$
by the full localization of the fibers. 
Thus $\cY$ will represent $\dirtLim(\cV)$. 

Define the groupoid $\cY$ as follows.
\begin{itemize} 
\item $\Ob(\cY)=\Ob(\cX)$.
\item $\displaystyle\Mor(\cY)=\coprod_{a_1,a_2\in A}V(a_1\cap a_2)$ 
(we are using here the notation of~\ref{intersection-of-charts}).
\item The maps $s,t:\Mor(\cY)\to\Ob(\cY)$ are just the projections
$\pr_1,\ \pr_2$ from 
$V(a_1\cap a_2)$ to $V(a_1)$ and to $V(a_2)$.
\end{itemize}

The structure maps $s,t$ are \'etale. 
The composition in $\cY$ is given by the canonical projections
$$  c(a_1\cap a_2)\times_{c(a_2)}c(a_2\cap a_3)
\rTo c(a_1\cap a_3),\ a_1,a_2,a_3\in A
$$
defined as follows.

Consider the chart
$$c=(V,H,\pi)= c(a_1\cap a_2)\times_{c(a_2)}c(a_2\cap a_3).$$
 A map from $c$ to $ c(a_1\cap a_3)$ is uniquely determined
by maps $c\to c(a_i),\ i=1,2,$
and an isomorphism between the two maps from $V$ 
to $[\pi_1(V_1)\cap\pi_3(V_3)]$.
The maps $c\to c(a_i),\ i=1,2,$ are defined by 
$$\pr_1:c(a_1\cap a_2)\to c(a_1) \mathrm{\ and\ }
\pr_2:c(a_2\cap a_3)\to c(a_3).$$

To get an isomorphism of the two maps from $V$ we notice
that all the realizations involved, $[\pi_i(V_i)]$ and their double and
triple intersections, are global quotient orbifolds by~\ref{crl:globalchart} 
and their inclusions are open embeddings of orbifolds. 
Since $\pi(V)$ belongs to the triple intersection 
$\pi_1(V_1)\cap\pi_2(V_2)\cap\pi_3(V_3)$, the isomorphisms between
the two maps from $V$ to $[\pi_1(V_1)\cap\pi_2(V_2)]$ 
and between the
two maps from $V$ to $[\pi_2(V_2)\cap\pi_3(V_3)]$ 
induced by the maps
$c\to c(a_1\cap a_2)$ and $c\to c(a_2\cap a_3)$
can be realized as isomorphisms between two pairs of maps from $V$ 
to $[\pi_1(V_1)\cap\pi_2(V_2)\cap\pi_3(V_3)]$.
Their composition, composed with the open embedding of
$[\pi_1(V_1)\cap\pi_2(V_2)\cap\pi_3(V_3)]$ into
$[\pi_1(V_1)\cap\pi_3(V_3)]$, yields the required datum.

\

The canonical map $\iota:\cX\to\cY$ of $\Sp$-categories is defined as follows.
It is identical on the objects. For any morphism $\alpha:a\to b$ 
in $A$ a canonical map $\iota_\alpha:c(a)\to c(a\cap b)$ corresponds
to the pair $(\id_a,\alpha)$. This induces a map
$V_{\iota_\alpha}:V(a)\to V(a\cap b)$ which assembles into
the map $\iota:\Mor(\cX)\to\Mor(\cY)$.

We claim that, after passage to associated stacks, $\cY$
becomes the full localization of $\cX$.

One has
$$\Ob\cX(M)=\Ob\cY(M)=\coprod_{a\in A}\Hom(M,V(a))=\{(a,f)|a\in A,
f:M\to V(a)\}.$$

Furthermore,
$$\Mor\cX(M)=\!\!\!\coprod_{\alpha\in\Mor(A)}\!\!\!\Hom(M,V({s\alpha}))=
\{(\alpha,f)|\alpha\in\Mor(A),f\!:\!M\to V({s\alpha})\},$$
where
$$
s(\alpha,f)=(s\alpha,f),       \mathrm{\ and \ }     t(\alpha,f)=
(t\alpha,V(\alpha)\circ f:M\to V(s\alpha)\to V(t\alpha)).
$$

Similarly,
$$
\Mor\cY(M)=\!\coprod_{a,b\in A}\!\Hom(M,V(a\cap b))=
\{(a,b,f)|a,b\in A,f:M\to V(a\cap b)\},
$$
where
\begin{equation}
  \label{eq:star1}
s(a,b,f)=(a,V(\pr_1)\circ f), \mathrm{\ and \ } t(a,b,f)=(b,V(\pr_2)\circ f).
\end{equation}
In the formula~(\ref{eq:star1})
the maps $\pr_1:c(a\cap b)\to c(a)$ and $\pr_2:c(a\cap b)\to c(b)$ 
are the standard projections.

The functor $\iota:\cX(M)\to\cY(M)$ assigns to an arrow $(\alpha,f)$ in 
$\cX(M)$ the arrow $\quad$ $(s\alpha,t\alpha,V(\id,\alpha)\circ f)$
where  $c(\id,\alpha):c(a)\to c(a\cap b)$ is defined by the maps
$$\id:a\to a, \mathrm{\ and\ } \alpha:a\to b$$ 
in $A$.

For a pair of arrows $\alpha:d\to a$ and $\beta:d\to b$ in $A$ a map
$V(\alpha,\beta):V(d)\to V(a\cap b)$ is \'etale. Moreover, the maps 
$V(\alpha,\beta)$ form an \'etale covering of $V(a\cap b)$. Let
$f_{\alpha,\beta}:M_{\alpha,\beta}\to V(d)$ be the map obtained from 
$f:M\to V(a\cap b)$ via the base change along $V(\alpha,\beta)$
and let $U(\alpha,\beta):M_{\alpha,\beta}\to M$ be the base change
of $F(\alpha,\beta)$. 

A direct calculation shows
that
\begin{equation}
\label{morY}
(a,b,G(\alpha,\beta)\circ f)\circ\iota(\alpha,f_{\alpha,\beta})=
\iota(\beta,f_{\alpha,\beta}).
\end{equation}

Let $G$ be a stack of groupoids over $\Sp$
and let $F:\cX\to G$ be a functor of fibered categories.
We claim there exists a unique 
$\bar{F}:\cY\to G$ such that $F=\bar{F}\circ\iota$.
On objects $\bar{F}$ must coincide with $F$, since
$\Ob(\cX(M))=\Ob(\cY(M))$ for all $M\in\Sp$.

From~(\ref{morY}) it follows how $\bar{F}$ should act on morphisms:
given $(a,b,f)$ with $a,b\in A$ and $f:M\to V(a\cap b)$ an element of
$\Mor\cY(M)$ we have for each $\alpha:d\to a,\ \beta:d\to b$
\begin{equation}
\label{morYtoG}
\bar{F}(a,b,G(\alpha,\beta)\circ f)=
F(\beta,f_{\alpha,\beta})\circ F(\alpha,f_{\alpha,\beta})^{-1}.
\end{equation}
Since $G$ is a stack and $G(\alpha,\beta):M_{\alpha,\beta}\to M$ form a covering,
$\bar{F}(a,b,f)$ is uniquely defined.  To prove its
existence, we have to check that $\bar{F}$ defined by~(\ref{morYtoG})
commutes with compositions. 
This is a straightforward calculation.

This completes the proof of Theorem~\ref{thm:realization-is-orbi}.
\end{proof}

Now we can finish the proof of Theorem~\ref{satake-groupoid}.
Since $[X]$ is represented by the groupoid $\cY$, with
        $$\Ob(\cY)=\coprod_{a\in A}V(a),$$ 
for every $a\in A$ we have an open suborbifold $U_a$
represented by the $\Sp$-groupoid $G_\bullet(a)=(G_0(a),G_1(a))$ with
$$G_0(a)=V(a) \mathrm{\  and \ } G_1(a)=V(a\cap a).$$
Since the set of arrows $G_1$ can be identified with
  $$V(a)\times_{\left[H_a\bs V_a\right]}V(a)=H(a)\times V(a),$$
we see that the open suborbifold can be identified 
with the global quotient $[H(a)\bs V(a)]$.
\eproof

\ 

Note the following important  corollaries of this theorem.

\begin{crl}
\label{crl:universal-property}
Let $X$ be a Satake orbifold and  let $[X]$ be its realization.
Let $\cY$ be an arbitrary orbifold.
Then a  map $f:[X]\to\cY$ of orbifolds  is determined by the
following data:
\begin{itemize}
\item 
For each $a\in A$ a map $$f_a:[H(a)\bs V(a)]\to\cY$$ 
(where, as usual, $(V(a),H(a),\pi(a))=c(a)$ is the chart corresponding to
$a$).
\item For each morphism $\alpha:a\to b$ in $A$
a $2$-morphism 
      $$\theta_\alpha:f_a\to f_b\circ\left[c(\alpha)\right]$$ 
where  for a morphism $\phi$ of orbifold charts we denote
by $[\phi]$ the corresponding map of quotient orbifolds.
\end{itemize}
These data are required to satisfy obvious
 compatibility condition for $\theta_\alpha$. 
\end{crl}

\begin{prp}
The categories of sheaves  (or categories of vector bundles) on a
Satake orbifold $X$ and its realization $[X]$ are canonically equivalent. 
\end{prp}
\begin{proof}
This result follows from our construction of the realization
of a Satake orbifold.
\end{proof}

\section{Algebraic moduli versus analytic moduli}
\label{gaga}

\subsection{Two ways of passing from algebraic to analytic families}
The two ways of looking at orbifolds discussed in Sections~\ref{sec:orbi}
and~\ref{sec:satake}---as groupoid-valued functors on 
a certain category of spaces (manifolds) 
and as geometric objects represented by groupoids in the category of 
spaces---suggest two possible ways of passing from one category of 
manifolds to another.
We are particularly interested in the passage from the category of 
 schemes  (of finite type over $\C$) to the category of analytic spaces.

The first way of passing from schemes to complex spaces is to replace
a functor on the category of 
 schemes with a functor on complex spaces. For example, 
the functor of families of stable curves over schemes becomes
the functor of families of stable Riemann surfaces.\footnote{This procedure
is not defined for an arbitrary Deligne-Mumford stack.} 

The second way is the change of the base category 
mentioned in~\ref{change-basecat}.  That is for a Deligne-Mumford 
stack $\cX$ represented by a groupoid $X_\bullet$
we can  apply the analytification functor
which produces  a groupoid $X_\bullet^{\an}$ representing an orbifold in the 
analytic category.

Of course, since the first procedure
is not even formally defined, we cannot
expect that  these two processes always give
the same result.

However, as we show in this section, for the moduli spaces of stable
punctured curves and of admissible coverings the two procedures are
equivalent. 

Let $\bsM_{g,n}$ be the stack of stable complex  curves
 of genus $g$  with $n$ punctures 
and let $\Adm_{g,n,d}$ (resp.\ $\Adm_{g,n}(H)$) be
the stack of admissible  coverings of degree $d$ 
(resp.\ of admissible $H$-coverings).

These stacks are proper Deligne-Mumford 
stacks and, therefore, $\Sp$-orbifolds where $\Sp$ is the category of schemes.

In this section we prove the following result.
\begin{thm}
\label{thm:gaga}
The analytification of the stack $\bsM_{g,n}$
(resp., of $\Adm_{g,n,d}$, resp., of $\Adm_{g,n}(H)$)
represents  the functor of  analytic families of stable 
curves of genus $g$ with $n$ punctures (resp., of 
admissible coverings of degree $d$ of curves of genus $g$ 
unramified outside of $n$ points, resp., of $H$-admissible coverings).  
\end{thm}

Here is a plan of the proof. First, following M.\,Hakim~\cite{hakim},
we define algebraic families  whose base is an arbitrary
 locally ringed topological space. 
The analytifications of the algebraic moduli stacks
automatically represent the corresponding algebraic families
with bases in the category of complex-analytic spaces. 
The rest follows from the fact proved in Section~\ref{sec:an-an}
that any complex-analytic family of
stable complex curves is necessarily projective (and therefore algebraic).
This is  a generalization of the well-known fact that every compact
complex manifold of dimension one  is projective.
Thus our proof does not work for families of algebraic varieties of dimension
higher than one.

\subsection{Analytic families of algebraic curves}
\label{sec:analytic_fams}
When we  speak about families of varieties parametrized by an 
analytic space, we usually mean an analytic family of the corresponding
complex analytic spaces. Sometimes it is important to  have both
``analytic'' and ``algebraic'' directions. This can be done using the notion
of a family of schemes parametrized by a ringed topological space introduced
by M.\,Hakim~\cite{hakim} (in a much greater generality).

We present below a definition of a family of objects of a stack $\cX$ 
parametrized by a locally ringed topological space. 
For the stack  $\cX=\bsM_{g,n}$  this gives the notion
of an analytic family of algebraic curves.
Similarly,  for  $\cX=\Adm_{g,n,d}$ we get the notion of an 
analytic family of algebraic admissible coverings.
 
\

\begin{dfn}
Let $(X,\cO)$ be a locally ringed site and let $\cX$ be a stack of groupoids
on the category of affine schemes.
Let $\Pre$-$\cX(X,\cO)$ be the fibered category over $X$ whose
fiber  over $U\in X$ is the groupoid 
     $$\Pre\mathrm{-}\cX(X,\cO)_U=\cX(\cO(U)).$$
Denote by $\cX(X,\cO)$ the groupoid of  global sections of the stack
associated to the fibered category $\Pre$-$\cX(X,\cO)$.
Objects of $\cX(X,\cO)$ are called {\em families of objects of $\cX$
parametrized by $(X,\cO_X)$.} 
\end{dfn}
For $\cX=\Sch$ this gives to the notion of
a scheme over $(X,\cO)$ (see~\cite{hakim}); 
for $\cX=\bsM_{g,n}$ we get the notion of an analytic family of stable
algebraic curves, and  $\cX=\Adm_{g,n,d}$ 
we obtain the notion of an analytic family of algebraic 
admissible coverings.

Thus, by definition, \emph{a scheme over $(X,\cO)$} is  given
by a collection of the following data.
\begin{itemize}
\item An open covering $\{U_i\}$ of $X$;
\item A collection of schemes $Y_i$ over $\Spec\ \cO(U_i)$;
\item A compatible collection of isomorphisms of the pullbacks of $Y_i$
and of $Y_j$ to $\Spec\ \cO(U_{ij})$.
\end{itemize}

A similar description can be given 
for $\cX=\bsM_{g,n}$ or $\Adm_{g,n,d}$.

\begin{prp}
\label{anfamily-algobjects}
Let $\cX$ be an algebraic  Deligne-Mumford
stack of finite type over $\C$. Then
the functor $(X,\cO)\mapsto\cX(X,\cO)$  from
the category of analytic spaces to the category of
groupoids is representable by the analytification of $\cX$.
\end{prp}
\begin{proof}
The analytification $\cX^\an$ is defined as follows. Let $\cX$ be presented
by a groupoid $X_\bullet$ where $X_i$, $i=0,1$ are schemes of finite
type over $\C$.  Then  $\cX^\an$ 
is defined as the stack associated to the groupoid $X^\an_\bullet$.

The  statement of the proposition  follows immediately from 
the following facts.
\begin{itemize}
\item A map $\Hom_{\LR}((X,\cO),\Spec A)\to\Hom_{\COM}(A,\Gamma(X,\cO))$ is a
bijection. Here the left-hand side $\Hom$ is taken in the category of
locally ringed spaces and the right-hand side $\Hom$ in the category of
commutative rings.
\item A map $X\to M$ from an analytic space $X$ to a scheme of locally
finite type over $\C$ in the category of locally ringed spaces lifts
canonically to a map $X\to M^\an$ of analytic spaces.
\end{itemize}
\end{proof}

Thus, according to Proposition~\ref{anfamily-algobjects}, the
complex-analytic stack $\bsM_{g,n}^\an$ represents (algebraic) families
of stable  curves of genus $g$ 
with $n$ punctures parametrized by complex-analytic spaces. A similar
claim is true for $\Adm_{g,n,d}^\an$ and $\Adm_{g,n}(H)^\an$. 

\subsection{Analytic families of analytic curves}
\label{sec:an-an}

Here we prove that any analytic family of stable curves (or of stable
admissible coverings or of stable admissible $H$-coverings) 
is algebraic, i.e.\ it can be obtained as the analytification of
an algebraic family.  Together with Proposition~\ref{anfamily-algobjects}
this will give Theorem~\ref{thm:gaga}.

\

\begin{thm}
\label{anal-is-alg-curves}
Any analytic family $(\pi: X\to S,\sigma_1,\ldots, \sigma_n)$ 
of stable punctured curves is projective. In particular, it is an
analytification of an algebraic family over $S$.  
\end{thm}
\begin{proof}
Let $\omega_\pi$ be the relative dualizing sheaf of $\pi$. 
In  the analytic category it was defined in~\cite{RRV} as $\pi^!(\cO_S)$ 
where the functor
      $$\pi^!=D_X\circ\pi^*\circ D_S $$ 
is obtained from the inverse image functor by dualization.
The morphism $\pi$ is a locally complete intersection morphism, therefore
it follows (see e.g.~\cite{ha})  that $\omega_\pi$ is an invertible sheaf.
It satisfies the base change formula%
\footnote{%
A detailed proof of the  base change formula 
for  Cohen-Macaulay morphisms of locally Noetherian schemes
is given in in~\cite[Theorem~3.6.1]{C} 
It is based on a local description of $\omega_\pi$ in terms
of Ext functors and on the base change for Ext functors. 
For local complete intersections  it is given by the
compatibility lemma~\cite[Lemma 2.6.2]{C} whose proof 
remains valid in the analytic setting as well.
}  
$\omega_{\pi_T}=g^*(\omega_\pi)$ for a Cartesian diagram
\begin{equation}
\begin{diagram}
X_T & \rTo^g & X \\
\dTo^{\pi_T} & & \dTo_\pi \\
T & \rTo^f & S
\end{diagram}\quad.
\end{equation}
Let $D_i=\sigma_i(S)$ be the divisor in $X$ 
that corresponds to the $i$th marked point.
We claim that  the invertible sheaf 
 $$L=\left(\omega_\pi \otimes \cO_X(-\sum_{i=1}^n D_i)\right)^{\otimes 3} $$
 gives rise to a finite morphism 
      $$j:X\to \P((\pi_*(L))^*).$$
Indeed, since the restriction of $L$ to  every fiber $X_s, \ s\in S$,
is very ample,  the map $j$ is well-defined and its restriction to
$X_s$ is a closed embedding. Since $\pi$ is proper, $j$ is also proper
and thus is finite.   This implies that $\pi$ is projective.  

Let $X^\alg$ be the scheme over $S$ whose analytification is
isomorphic to $X$. According to the ``relative GAGA'' 
(see Theorem VIII.3.5 in~\cite{hakim})
the categories of coherent sheaves on $X$ and on $X^\alg$ are equivalent.

This implies that the sections $\sigma_i:S\to X$
are algebraic and also that any analytic automorphism of $X$ comes
from an automorphism of $X^\alg$. This completes the proof.
\end{proof}

Now we can easily obtain a similar result
for families of stable admissible coverings. 
\begin{thm}
\label{anal-is-alg-adm}
Any 
analytic family of stable admissible coverings 
      $$(C\to X\to S,\sigma_1,\ldots, \sigma_n)$$ 
is projective. The same is true for  families of admissible $H$-coverings.
\end{thm}
\begin{proof} The family $(X\to S, \sigma_1,\ldots, \sigma_n)$ is 
analytification of an algebraic family 
$$(X^{\alg}\rTo S,\sigma_1,\ldots, \sigma_n)$$
by Theorem~\ref{anal-is-alg-curves}. 
Since the covering $C$ of $X$ is given by a coherent sheaf of algebras, 
the result follows by  the ``relative GAGA''~\cite[Theorem
VIII.3.5]{hakim}. 

Finally, to deal with the case of  $H$-coverings, we notice 
that the balancedness condition in the definition of admissible $H$-covering
(see Section 4.3.1 of~\cite{ACV}) only involves geometric points.
Therefore  this condition is the same for analytic and algebraic version.
\end{proof}

\section{ \TS s and quasiconformal charts of $\bsM$}
\label{sec:quasiconf}
In the beginning of this section we introduce the \TS s
$\cT_{g,n}$ and $\bT_{g,n}$ and present some facts about them
which will be needed later.
Then we construct on $\bsM=[\Gamma_{g,n}\bs\bT_{g,n}]$
an orbifold atlas whose charts satisfy 
some very special properties. We call such charts 
\emph{quasiconformal}.
Our construction uses a version of the Earle-Marden~\cite{marden} local
holomorphic coordinates on the \TS\ $\cT_{g,n}$.  
In Section~\ref{sec:teich} we will use this quasiconformal atlas to 
construct an orbifold atlas on quotient of the augmented \TS\ $\bT$ 
by finite-index subgroups of the modular group $\Gamma$.

\subsection{\TS s $\cT_{g,n}$ and $\bT_{g,n}$}

Here we recall definitions and standard facts about the \TS s
$\cT_{g,n}$ and $\bT_{g,n}$.

Let us fix a compact oriented surface $S$  of genus $g$
with $n$ boundary components $L_1, \ldots, L_n$ and smooth parametrizations
        $$\lambda_i: S^1 \to L_i$$
 compatible with the orientation of $S$.
Here and below, we assume that the surface $S$ is of hyperbolic type, i.e.\
$2g+n-2>0$. 

\

\begin{dfn}
\label{dfn:ATS}
Let $(X,p_1,\ldots,p_n)$ be a stable complex  curve $X$ of
arithmetic genus $g$ with $n$ punctures $p_i\in X$. A \emph{marking} of
the punctured  curve $(X,p_1,\ldots,p_n)$ is  a continuous map
      $$\phi: S\to X$$
satisfying the following properties
\begin{enumerate}
\item[(i)] 
The preimage $\phi^{-1}(p_i)$  of the $i$th puncture $p_i\in X$  is
the $i$th boundary  component $L_i \subset S$.
\item[(ii)] 
The preimage $\phi^{-1}(q)$ of every node  $q\in X$
 is a simple closed curve in $S$.
\item[(iii)] The map $\phi$ induces a homeomorphism 
      $$\phi^{-1}(X_{\reg})\to X_{\reg},$$ 
where 
$$
   X_\reg = X - X_\mathrm{sing} - \{p_1,\ldots,p_n\},
$$
         is the complement of the sets of nodes and punctures of $X$.
\end{enumerate}

Two markings $\phi, \phi' : S\to X $ are called \emph{isotopic}
if $\phi'=\phi \circ f$, where $f$ is a diffeomorphism
of $S$, such that 
\begin{equation}
  \label{eq:diff_bnd}
f_{L_i}=\mathrm{Id}_{L_i}, \ i=1,\ldots, n,  
\end{equation}
and $f$ is isotopic to the identity in the class of diffeomorphisms
satisfying~(\ref{eq:diff_bnd}).
\end{dfn}

\begin{dfn}
A punctured stable curve $X$ with an isotopy class of markings $[\phi]$ is
called a \emph{marked curve}.

The set $\bT_{g,n}$  of isomorphism classes of marked curves of
genus $g$ with $n$ punctures is called the \emph{augmented \TS}.
\end{dfn}

\begin{Rem}Sometimes, when we wish to stress the functorial dependence
of the augmented \TS\ on $S$, we will use the notation $\bT(S)$ instead of
 $\bT_{g,n}$. Of course, $\bT(S)$ depends, up to non-canonical isomorphism,
only on the genus of $S$ and on the number of its boundary components.
\end{Rem}

The points $(X,[\phi])$ of $\bT_{g,n}$, where $X$ is a non-singular 
complex curve, form
the usual \TS\ $\cT_{g,n}$.

In order to introduce a topology on $\bT_{g,n}$ we need the following notion.

\begin{dfn}
\label{dfn:contraction} 
Let $(X,x_1,\ldots,x_n,\phi)$ and $(Y,y_1,\ldots,y_n,\psi)$ be two
marked stable curves.
A continuous map $f: X \to Y$ is called a \emph{contraction} if  it
satisfies the following conditions.
\begin{itemize}
\item[(i)] 
$f(x_i)=y_i$ for $i=1,\ldots,n$.
\item[(ii)] 
$f$ induces a homeomorphism $f^{-1}(Y_\reg) \to Y_\reg$.
\item[(iii)] 
 For every node $y\in Y$ its preimage $f^{-1}(y)$ is
  either a node  of $X$ or a simple closed loop.
\item[(iv)] 
The marking $\psi$ of $Y$ is isotopic to $f \circ \phi$.
\end{itemize}

For unmarked punctured curves  $(X,x_1,\ldots,x_n)$ 
and $(Y,y_1,\ldots,y_n)$ a contraction is defined as any 
continuous map $f: X \to Y$ satisfying conditions (i)---(iii).
\end{dfn}

The following sets form a basis of the topology of the augmented \TS.
Choose a marked curve 
$$(Y,y_1,\ldots,y_n,[\psi]) \in \bT_{g,n},$$
a number $\varepsilon > 0$ and an open subset $N$ of $Y$ containing 
all the nodes of $Y$.  
The neighborhood $\cU_{N,\varepsilon} \subset \bT_{g,n}$ 
is  defined as the 
set of all  
$$(X,x_1,\ldots,x_n,[\phi])\in \bT_{g,n}$$
for which there exists a contraction $f: X \to Y$
such that the restriction of $f$ to $f^{-1}(Y-\bar{N})$ is
$(1+\varepsilon)$-quasiconformal.

\subsubsection{Modular group action}
\label{sec:modular}

 Let 
$$\Gamma_{g,n}=\pi_0(\Diff^+(S/\partial S))$$ 
be the \teich\  \emph{modular group}, i.e.\ the group  of isotopy classes 
of  orientation preserving diffeomorphisms of $S$ identical on the 
boundary  $\partial S$.
(This group is also known as the mapping class group of the $n$-punctured
surface of genus $g$, cf.~\cite{Iv}).
We will usually denote this group by $\Gamma(S)$ or simply by $\Gamma$.
The modular group $\Gamma$ naturally acts on $\bT_{g,n}$ and on $\cT_{g,n}$
as follows:
\begin{equation}
  \label{eq:action}
[\gamma](X,[\phi]):=(X,[\phi\circ \gamma^{-1}]),
\end{equation}
where $[\gamma]\in \Gamma$ is a mapping class represented by a diffeomorphism
$\gamma$ and $\phi:S\to X$ is a marking of $X$.

This action allows the following description 
of markings of a nodal curve
$X_0$ in terms of markings of nearby smooth curves. Let $X$ be a smooth curve
and let $X_0$ be a nodal curve.

Assume there is a contraction of $X$ to $X_0$  that contracts
several disjoint simple closed curves $C_1,\ldots, C_r$ on $X$.

\begin{prp} \label{prp:dehn}
There is a natural bijection between the set of isotopy classes of
markings of the nodal curve $X_0$ and the set of 
$G$-orbits in the set of isotopy classes of markings of $X$, where
$G$ is a subgroup of $\Gamma$ generated by the Dehn twists around the curves 
$C_1,\ldots,C_r$.
\end{prp}

\qed

 \

We will use the following classical results about the \TS s
$\cT_{g,n}$ and $\bT_{g,n}$ and the action of the modular group on
them (for details see \cite{ab3} and references there).

\begin{thm}
  \begin{itemize}
  \item[(i)]
The space  $\cT_{g,n}$ has a structure of a complex manifold of complex
dimension $3g+n-3$ diffeomorphic to an open ball in $\R^{6g+2n-6}$.

  \item[(ii)]
The quotient $ \Gamma\bs\cT_{g,n}\ $ is isomorphic, as a complex space, 
to $\cM_{g,n}$.

\item[(iii)]
The quotient space $ \Gamma\bs\bT_{g,n}\ $ is homeomorphic to $\bM_{g,n}$.
  \end{itemize}
\end{thm}

Here and below $\cM_{g,n}$ (resp.,  $\bM_{g,n}$) denotes the complex
 space associated to the moduli stack of compact Riemann surfaces of
 genus $g$   with $n$ marked points (resp., its Deligne-Mumford
 compactification).

\subsection{Complex structure of  $\cT_{g,n}$}

We present below a modular description of the complex space $\cT_{g,n}$
(see~\cite{groth,earle,engber}).

\begin{dfn}
Let $B$ be a complex space. A \emph{family of smooth
curves} of genus $g$ with $n$ punctures over
the base $B$ is a flat proper morphism of complex spaces 
$\pi:C\to B$  with $n$ sections $\sigma_i: B\to C$,
such that fibers of $\pi$ are complex  curves of genus $g$ and
the images of the sections  $\sigma_i, \ i=1,\ldots,n$, 
are pairwise disjoint.
\end{dfn}

To each $b\in B$ we assign a set 
$$P_b=\pi_0(\Diff^+(S,C_b)),$$
 where
$C_b=\pi^{-1}(b)$. Since $\pi$ is topologically a locally trivial fibration,
these sets assemble into a covering 
$$p:P\rTo B$$ with fibers $P_b$.

\begin{dfn}
A \emph{marking}  of a family of smooth curves $\pi:C\to B$ is 
a section of the associated covering 
$$p:P(\pi)\rTo B.$$

If  $G\subset \Gamma$ is  a subgroup of the modular group,
a section of the covering 
$$p_G: G\bs P(\pi) \rTo B$$ 
is called
a $G$-\emph{marking} of the family $\pi$.
\end{dfn}

\

The following result proved in~\cite{earle,engber} generalizes 
the theorem of  Grothendieck on modular description of the 
\TS\ $\cT_g=\cT_{g,0}$.

\begin{thm} \label{th:teich_mod}
For $2g+n>2$, the functor
$$
B\mapsto F(B),$$
where
$F(B)$ is the set of isomorphism classes of  
marked curves of genus $g$ with $n$ punctures over $B$,
is representable by a complex manifold.
The representing object is isomorphic to the \TS\ $\cT_{g,n}$.
\end{thm}

\subsection{Quasiconformal atlas for $\bsM$}
\label{ss:qc}

In this section we prove the existence of an atlas on $\bM$ with
especially nice orbifold charts. These charts, which we call
\emph{quasiconformal}, satisfy a collection of properties described
in~\ref{sss:QCdef}. Our approach is based on the plumbing construction
of Earle and Marden~\cite{marden}. This construction produces a family
of stable curves over a polydisk starting with a maximally degenerate
curve $X_0$ and a collection of local coordinates near the nodes of $X_0$. 
This family of curves is not everywhere locally universal, i.e.\ it does
not necessarily give an orbifold chart for the moduli space $\bM$
(see a counterexample in~\cite{cex}).

However, as we show in this section,  open subsets of those coordinate
polydisks which do form an orbifold chart cover the whole moduli space
and therefore give an orbifold atlas with required properties.

To prove that the charts obtained from the plumbing construction
cover the whole moduli space, we proceed as follows. First,
for each stable curve $X$ we  describe very special plumbing data
$(X_0,\ z_i)$,  where $X_0$ is a maximally degenerate stable curve of
genus $g$ with $n$ punctures (all such curves have $m=3g+n-3$ nodes)
and $z_1,\ldots, z_{2m}$ are local parameters near the nodes of $X_0$.

This data gives rise to a family of curves 
\begin{equation}
  \label{eq:family0}
  \pi: \cX \rTo U
\end{equation}
whose base $U$ is an neighborhood of the origin in $\C^{3g+n-3}$.
This family, which we construct in~\ref{choice-of-lc},
contains $X$ and has the property that  the geodesics
(in the hyperbolic metric) which cut $X$ into a union of
``pairs of pants''  in local coordinates $z_i$  have equations $|z_i|=s_i$.

The family~(\ref{eq:family0}) is induced from the universal family
over the moduli stack $\bsM$ via a map $U\to \bsM$ which gives rise to an 
orbifold chart 
$$\hat{\beta}:[A\bs U]\to\bsM.$$

To prove \'etalness of $\hat{\beta}$ we first show in~\ref{ks-calculation} 
that, when $X$ is non-singular, the restriction of the 
family~(\ref{eq:family}) to a certain subspace of $U$ of \emph{real}
dimension $m$ is the Fenchel-Nielsen family (see Section~\ref{sss:fn-family}).
This gives \'etalness in the non-singular case
and   in~\ref{sss:etality-nodal} we deduce from it the general case.

\subsubsection{Quasiconformal charts on $\bsM$}
\label{sss:QCdef}

We start with a definition of quasiconformal charts.

Let $U$ be an open subset of $\C^m$ with an action of a finite group $A$.
Let  
$$\hat\beta:[A\bs U]\to\bsM $$ 
be an open embedding and let 
$$\beta:U\to\bM $$ 
be the corresponding map to the coarse moduli space.
Denote by 
$$\pi:\cX\to U$$ 
the family of nodal curves on $U$ induced by
$\beta$ and let $U_0$ be the smooth locus of $\pi$:
      $$ U_0=U\times_{\bsM}\sM.$$
The complement $U-U_0$ will be called the singular locus (of $\pi$).
For $t\in U$ we  denote by $X_t$ the fiber $\pi^{-1}(t)$.

\

The construction of~\ref{orb-to-satake} provides 
$\bM$, the coarse moduli space 
of the smooth complex orbifold $\bsM$, with an orbifold atlas. 
Below, in our construction of an orbifold atlas for $G\bs\bT_{g,n}$, we will need
charts satisfying some nice properties.  We call such charts {\em quasiconformal}. 

\

\noindent
{\bf Definition.}
An orbifold chart $(U,A,\beta)$ of the complex orbifold $\bsM$ is
called {\em quasiconformal} if it satisfies the  following conditions (QC1)--(QC6).
\begin{itemize}
\item[(QC1)] The manifold $U$ is analytically equivalent to a contractible
neighborhood of $0$ in $\C^m$, so that 
the singular locus $U-U_0$ 
corresponds to the union of (some) coordinate hyperplanes. In
particular, if $U\ne U_0$, the intersection of the components
of the  singular locus is stable 
under the $A$-action. 
We assume that there exists a point $z\in U$ fixed by $A$. If $U\ne U_0$, we
assume that $z$ lies in the intersection of the components of the  singular locus.

\item[(QC2)] For  every $t\in U$ there exists an open neighborhood 
$U^t$ of $t$ in $U$ and a 
\emph{quasiconformal contraction} --- a
continuous map  
$$c^t:\cX^t\to X_t,$$
where $\cX^t$ is the 
restriction of $\cX$ to $U^t$,  such that 
for every fiber $X_s,\ s\in U^t$, the restriction
      $$c^t_s=c^t|_{X_s}:X_s\rTo X_t$$
is  a contraction (see Definition~\ref{dfn:contraction}). 
In addition, the map $c^t$  is \emph{quasiconformal} in the following
sense. 

Let $\phi_t:S\to X_t$ be a marking; choose a neighborhood $N$ of the nodes 
of $X_t$ and $\varepsilon>0$. Then there exists a small neighborhood
$U^\delta$ of $t$ in $U^t$ such that for any $s\in U^\delta$ and for any 
marking $\phi_s:S\to X_s$  for which $\phi_t$ is isotopic
$c^t_s\circ\phi_s$,  the contraction $c^t_s:X_s\to X_t$ is
$(1+\varepsilon)$-quasiconformal outside the preimage of $\ol{N}$.
\item[(QC3)] For  every $t\in U$ there  exist neighborhoods $\cO_i\ni x_i$ 
of the nodes $x_i,\ i=1,\ldots,r$, of the curve $X_t$  such that 
\begin{itemize}
\item[(a)] The maps 
$$ c^t:\cX^t\rTo X_t\textrm{ and } \pi^t:\cX^t\rTo U^t$$
define an analytic isomorphism
\begin{equation}
\label{dir-product}
(c^t)^{-1}(X_t-\bigcup\cO_i)\rTo U^t\times(X_t-\bigcup\cO_i).
\end{equation}
\item[(b)]
For  every $i=1,\ldots,r,$ the map
\begin{equation}
\label{dir-product-2}
(c^t)^{-1}(\cO_i)\rTo U^t
\end{equation}
is analytically isomorphic to the standard projection  
$$P_i \to D^m $$ 
from  
$$ P_i=\{(u,v,t_1,\ldots,t_m)\in D^2\times D^m|uv=t_i\} $$
to the standard polydisk $D^m\subset \C^m$. 
\end{itemize}
\item[(QC4)] 
For any $s\in U^t, \ u\in U^s\cap U^t\cap U_0$
there exists a homeomorphism 
     $$\theta:X_u\to X_u$$
  isotopic to the identity, such that
$$ c_u^t\circ\theta=c_s^t\circ c_u^s$$

\item[(QC5)] One has $U=U^z$.
\item[(QC6)] 
For a node $x$ of $X_z$ let $D_x$ be the space
$$ D_x=\{t\in U|(c_t^z)^{-1}(x)\textrm{ is a point }\}.$$
Then $D_x$  is a component of the singular locus
and  every component of the singular locus is obtained
in this way.
\end{itemize}

\
\begin{Rems}
\

1. 
Note that if the condition (QC2) is valid for some marking
$\phi_t$ of $X_t$, then it is valid for all markings of $X_t$.
Also, since $\sM=[\Gamma\bs\cT]$, the condition (QC2) is empty for 
$t\in U_0$.

2. 
Existence of continuous 
contractions $c^t$ in (QC2) is not a very restrictive
condition.  What makes it non-trivial is the 
requirement that $c^t$ is quasiconformal.

3.
The property (QC3) means that the family of curves over
$U$ is constant outside neighborhoods of the nodes
and is equivalent to the family given by the 
plumbing construction (see~\ref{ssub:basicfamily}) 
in the neighborhoods of the nodes. 

4. 
The property (QC6) identifies   the set of
components of the singular locus with the set of nodes of $X_z$.
A marking $\phi: S \to X_z$ of $X_z$
allows to identify the fundamental group of $U_0$ with the subgroup 
  of the modular group $\Gamma$ generated by the Dehn twists around
$\phi^{-1}(x)$, where $x$ runs through the nodes of $X_z$.

5. 
Below we will construct a collection of quasiconformal charts for
$\bsM$ using a plumbing construction and will prove  that they give an
orbifold atlas of $\bsM$.  This means that, in a certain sense, all 
sufficiently small orbifold charts of $\bsM$ are quasiconformal. 
\end{Rems}

The notion of a quasiconformal chart serves a bridge between the
\teich\   and the stack-theoretic approach to the description
of the moduli space of stable curves.

\begin{thm}\label{thm:atlas}
The moduli stack  $\bsM$ of stable curves admits an orbifold atlas 
of quasiconformal charts.  
\end{thm}

Proof of this theorem occupies the rest of this subsection
(\ref{ssub:basicfamily}---\ref{atlas-for-bM}).

\subsubsection{Plumbing construction}
\label{ssub:basicfamily}

Fix a  maximally degenerate stable  curve $X_0$ of genus $g$ with $n$
punctures,  i.e.\ $X_0$ has a maximal possible number of nodes 
   $$m=3g+n-3.$$  
Let $x_1,\ldots,x_m \in X_0$ be the  nodes of $X_0$.
For each node $x_i\in X_0$ fix an open neighborhood $x_i\in V_i \subset X_0$,
such that  these sets $V_i$ are pairwise disjoint, do not contain punctures
and  each $V_i$ is a union of two subsets 
   $$V_i=U_i\cup U_{i+m}, \ i=1,\ldots,m,$$
meeting at the point $x_i$ and 
homeomorphic to the open unit disk $D\subset \C$. 
Finally, 
let $$z_k: D \to X_0, \ k=1,\ldots, 2m,$$ be holomorphic maps
such that $z_k$ 
gives a homeomorphism between $D$ and $z_k(D)=U_k$ and
     $$z_i(0)=z_{i+m}(0)=x_i, \ i=1,\dots,m.$$
Using the choices of the curve $X_0$ and of $2m$ 
local coordinate functions $z_i$,
we will construct a  family $\cX$ of stable punctured 
curves over the polydisk $D^m$ as follows. 

Take an open subset $\cY\subset X_0\times D^m$ given by
$$
\cY=X_0\times D^m - \bigcup_{i=1}^m W_i,
$$
where
$$
W_i=\{(x,t_1,\ldots,t_m)\in X_0\times D^m \ |\  x= z_i(z) \text{\ or\ }
x= z_{i+m}(z) \text{\ for\ } |z|\le |t_i|\}
$$
and
$$
P_i = \{(u,v,t_1,\ldots,t_m)\in D^2\times D^m \ | \ uv=t_i\}~.
$$
We glue the manifolds $\cY$ and $P_i$ using the equivalence relation
generated by the following conditions.
\begin{itemize}
\item The point of $\cY$
with coordinates
$(z_i(z),t_1,\ldots,t_m)$ is equivalent to the point of $P_i$ with coordinates
$(z,t_i/z,t_1,\ldots,t_m)$
\item The point of $\cY$ with coordinates
$(z_{i+m}(z),t_1,\ldots,t_m)$ is equivalent to the point of $P_i$ 
with coordinates $(t_i/z,z,t_1,\ldots,t_m)$.
\end{itemize}

One easily sees that the quotient of $\cY\sqcup P_1\sqcup\ldots\sqcup P_m$
by the equivalence relation described above is Hausdorff; it is, therefore,
a complex manifold
which we denote by $\cX$. It is fibered over $D^m$; its fiber $X_t$
over $t=(t_1,\ldots,t_m)$ is obtained from the original nodal curve $X_0$
by ``holomorphic plumbing'' which replaces a neighborhood of the node $x_i$,
for which $t_i\ne 0$, locally 
parametrized by a neighborhood of the node of the curve
$uv=0$, with a piece of the smooth curve $uv=t_i$.

\

The fiber $X_t$ of the  above family 
is smooth if and only if all the coordinates of $t$ are nonzero.

Introduce the following notation
$$\Do=D-\{0\}, \ B_0=(\Do)^m \text{\ and \ } B=D^m$$
and let
\begin{equation} \label{eq:family}
\pi:\cX\rTo B
\end{equation}
be the family of curves constructed above.
The restriction of $\pi$ to $B_0$ 
gives the family 
$$\pi_0:\cX_0\to B_0$$ 
of smooth curves.

According to the results of Section~\ref{gaga}, 
      the stack  $\bsM$ represents complex
families of nodal curves with punctures. Thus, the family~(\ref{eq:family})
defines a map 
     $$\hat\beta:B\to \bsM.$$

As was shown in~\cite{cex}, the map $\hat\beta$ is not necessarily
\'etale. We will show however, that for any stable punctured curve $X$ 
there exists a choice of a maximally degenerated curve $X_0$, 
together with a choice of local coordinates near the nodes
 so that, for some point $t\in B$,  the map $\hat\beta$ is \'etale at $t$ and 
$\beta(t)$ is presented by $X$.

For the point and the plumbing data 
chosen as above, consider the group $A=\Aut(\hat\beta(t))$. 
According to Lemma~\ref{lem:orbi-neighborhood}, there exists a contractible 
neighborhood $U$  such that $(U,A,\hat\beta)$ 
gives an open embedding $[A\bs U]\to \bsM$.
We also  assume  that $U$ does not 
intersect  coordinate hyperplanes which do not contain $t$.
The singular locus of $(U,A,\beta)$ is the union of coordinate
hyperplanes containing $t$.  
The collection of quasiconformal contractions
is given by the standard contraction of the family 
   $$ \{(z,w,t)\in\C^3|\ |z|\leq 1,\  |w|\leq 1,\  |t|\leq 1,\ zw=t\}$$
over the closed disk $|t|\leq 1$ to the fiber at $t=0$.

\subsubsection{Construction of the family}
\label{choice-of-lc}

Let $(X,x_1,\ldots,x_n)$ be a  
punctured curve with $r$ nodes. 

We endow the complement
$$X_\reg=X - \{\text{nodes and punctures}\}$$
with the canonical complete hyperbolic metric. Choose a
maximal collection of simple disjoint geodesics 
$$C_i, i=1,\ldots,m-r$$ 
on $X$ such that their complement
$$X_\reg - \bigcup_iC_i$$
is a disjoint union of pairs of pants $P_j$, $j=1,\ldots,2g-2+n$.

Note that each geodesic $C_i$ has a natural (angular) parametrization.
To each boundary component of each pair of pants $P_j$ we glue a
punctured disk, so that the angular parametrizations on the common circle
coincide. As a result, we get an embedding of each pair of pants $P_j$
into a triply punctured sphere $S_j$; each punctured disk glued to a pair
of pants $P_j$ defines an open embedding $z:D_0\to S_j$ which is
{\em almost} the local coordinate near the puncture we need.

Here is the reason we will have to make a small adjustment to the embeddings 
$z:D_0\to S_j$. If $w:D_0\to S_k$ is the other local coordinate corresponding
to the same geodesic $C_i$, the gluing formula is $zw=1$, whereas
we were supposed to get $zw=t$ with $|t|<1$.

The lemma below claims that each open embedding $z_i:D_0\to S_j$ can be 
extended to an open embedding $Z_i:D'_0\to S_j$ of a greater
punctured disk. Then we can substitute
$z_i$ with $Z_i(1+\varepsilon)$ so that the geodesic $C_i$ will be given by the
equation $|z|=\frac{1}{1+\varepsilon}$ and the images of the unit disks will 
still have no intersection.

\begin{Lem}
Let $X$ be a bordered Riemann surface and $C$ be its boundary
component endowed with the intrinsic metric. Glue a unit disk $D$ to $X$
so that the common boundary component acquires the same angular coordinate
from $X$ and from $D$.
Let $\wh{X}$ be the resulting Riemann surface.
Then the map $D\to \wh{X}$ extends to an open embedding $D'\to\wh{X}$
of a strictly greater disk $D'\supseteq D$ having the same center.
\end{Lem}
\begin{proof}The claim is clear if $X$ is a half-annulus
$A=\{z|\ 1\leq |z|<c\}$. Then $\wh{X}$ identifies with the disk
$\{z|\ |z|<c\}$ strictly containing the unit disk.

Now, if $X$ is arbitrary, let
$$X^d=X\cup_C\bar{X},$$
be the double of $X$ with respect to $C$,
where $\bar{X}$ is the antiholomorphic copy of $X$.
Then the Nielsen extension of $\bar{X}$ 
at $C$ embeds into $X^d$ and has form $A\cup\bar{X}$ where $A$
is a half-annulus having $C$ as the boundary and embedded into $X$.
This gives the required extension.
\end{proof}

The $r$ nodes of the original curve $X$ identify some pairs of punctures
of\  $\coprod_jS_j$. This gives a maximally degenerated curve $X_0$
having $r$ ``original'' nodes and $m-r$ new nodes, endowed with local 
coordinates 
$$z_1,\ldots,z_{m-r},w_1,\ldots,w_{m-r}$$
near the $m-r$ new nodes. We can choose the $2r$ coordinates near $r$
``original'' nodes in an arbitrary way. The curve $X$ is obtained from $X_0$
by the plumbing construction with parameters $t=(t_1,\ldots,t_m)$
where the geodesic $C_i$ in the corresponding pair of local coordinates
is given by the equations $|z|=\sqrt{t_i},\ |w|=\sqrt{t_i},\ $
and $t_i=0$ for $i>m-r$.

\subsubsection{The case of a  smooth curve}
\label{ks-calculation}

Assume that $X$ has no nodes.
We assume $X=X_t$ for some $t\in B_0$. The family
$\pi_0:\cX_0\to B_0$ of Riemann surfaces defines a map 
$T\beta:T_tB_0\to T_{\beta(t)}\sM\ $
of complex vector spaces. We want to prove that this map is an isomorphism
if $\pi$ is the family constructed in~\ref{choice-of-lc}.

The tangent space $T_{\beta(t)} \sM\ $ identifies with the cohomology
$H^1(X_t,T)$ where $T$ is the sheaf of vector fields vanishing at the
punctures. The image of a 
vector $v\in T_t(B)=\C^m$ is described by an explicit \v{C}ech 1-cocycle.
Thus the problem reduces to proving that some \v{C}ech 1-cocycles are not
coboundaries. This is, however, difficult to calculate explicitly,
and this is not true for a general choice of local coordinates---see
a counterexample in~\cite{cex}.

\subsubsection{The Fenchel-Nielsen family}
\label{sss:fn-family}
Recall  the construction of the Fenchel-Nielsen coordinates on the \TS.
As above, we have chosen a maximal collection of free loops on the basic
surface $S$. For each $(X,\phi)\in\cT_{g,n}$ a collection 
of geodesics is therefore defined. Their lengths give a
(real-analytic) map
\begin{equation}
\label{length-FN}
L:\cT_{g,n}\rTo \R_+^m
\end{equation}
(Fenchel-Nielsen length coordinates).
Fix $l=(l_1,\ldots,l_m)\in\R_+^m$. The preimage $L^{-1}(l)$ is a 
$\R^m$-torsor with the action of the $i$-th component of $\R$ given by 
cutting of a Riemann surface along the $i$-th geodesic, twisting the 
boundary components one with respect to the other, and gluing them back.

The map $L$ has a section which allows one to define what is classically
known as Fenchel-Nielsen  coordinates. This coordinate system consists of
 $m$ length coordinates~(\ref{length-FN}) 
 and $m$ angular Fenchel-Nielsen coordinates 
$\theta_1,\ldots,\theta_m$, chosen so that the shift by $2\pi$ 
along each coordinate corresponds to the Dehn twist. In what follows we will
use modified angular coordinates $\tau_i=\frac{l_i}{2\pi i}\theta_i$.

For a fixed value $l\in\R_+^m$ the Riemann surfaces from $L^{-1}(l)$
can be organized in a family with the base $\R^m$ --- this family is 
sometimes called the Fenchel-Nielsen deformation. 

Kodaira-Spencer theory~\cite{Kod}  provides for any $X\in L^{-1}(l)$
an $\R$-linear map $\R^m\to H^1(X,T)$, $T$ being the sheaf of
vector fields vanishing at the punctures of $X$.\footnote{Fortunately,
 Kodaira and Spencer developed their theory 
for $C^\infty$ families of complex manifolds!}

We denote the images of the coordinate vectors by 
$\frac{\partial}{\partial\tau_i}\in H^1(X,T).$

\begin{Lem}
The vectors $\frac{\partial}{\partial\tau_i},\ i=1,\ldots,m$, form
a basis of $H^1(X,T)$ over $\C$.
\end{Lem}
\begin{proof}
This follows from the Wolpert's formula~\cite[8.3]{ImaTani}
$$\omega_{\mathrm{WP}}=\sum_{i=1}^m d\tau_i\wedge dl_i$$
for the Weil-Petersson form on the \TS.
Since $\omega_{\mathrm{WP}}$ is nondegenerate, $d\tau_i$,
and therefore $\frac{\partial}{\partial\tau_i}$ are linearly independent.
\end{proof}

Now we can explain what is special 
about our choice of local coordinates.

The pullback of the family $\pi_0:\cX_0\to B_0$ along the map
$$
u:\R^m\rTo B_0
$$  
defined by the formula $u(x_1,\ldots,x_m)=(t_1e^{2\pi ix_1},\ldots,
t_me^{2\pi ix_m})$, is the Fen\-chel-Nielsen family.

Consider the diagram of  maps of tangent spaces
(``chain rule'') 
$$ \R^m \rTo^{Tu} \C^m \rTo^{T\beta}T_{\beta(t)}\sM=H^1(X,T).$$
The composition $T\beta\circ Tu$ sends the standard basis $\{e_i\}$
of $\R^m$ into a $\C$-basis  $\{\frac{\partial}{\partial\tau_i}\}$
of $T_{\beta(t)}\sM$. Since the map $Tu$ also sends the basis
of $\R^m$ into a basis of $\C^m$, the map $T\beta$ is an isomorphism.

This proves that in the case of $X$ smooth the special chart we have 
defined in~\ref{choice-of-lc} is \'etale at $t\in B_0$ for which $X=X_t$. 

The case of nodal curves is considered below.

\subsubsection{Proof of the \'etalness for nodal curves}
\label{sss:etality-nodal}
Let $X$ be a nodal punctured curve and let $\pi:\cX\to B$
be the family of curves built by the plumbing construction with the
special choice of the local coordinates as in~\ref{choice-of-lc}.
Assume that $X=X_t$ for $t\in\bar{B}$.
We have to check that the map of the tangent spaces
$$T(\beta):T_t B \rTo T_{\beta(t)} \bsM$$
is an isomorphism.
Since the dimensions of the vector spaces coincide, it is sufficient
to prove the injectivity. Let\ $t=(t_1,\ldots,t_m)$ and let
$$v=(v_1,\ldots,v_m)\in T_t B $$ 
belong to the kernel of $T(\beta)$. 
The target of  $T(\beta)$ is the collection
of deformations of $X_t$ over $\C[\epsilon]/(\epsilon^2)$. 
Triviality of such a deformation means 
in particular that all nodes of $X_t$  are 
preserved under the deformation; in other words,  one has 
$$ t_i=0 \Longrightarrow v_i=0.$$

Assume for simplicity that 
$t_1=\ldots=t_k=0=\tau_1=\ldots=\tau_k$ and $t_i\ne 0$ for $i>k$.
The normalization 
$X_t^\nor$ of $X_t$
is a smooth curve with $n+2k$ punctures
(all preimages of the nodes become punctures). 
Let $X'_0$ be obtained from $X_0$ by ungluing the first $k$
nodes and turning them into $2k$ punctures.%
\footnote{To unglue a single node  
$q$ we choose a neighborhood $U\ni q$ which does not contain  
other nodes, normalize $U$ and paste the result back.
We assign labels $n+1,n+2,\ldots, n+2k$ 
to the new $2k$ punctures in an arbitrary way.}
Let $B'=D^{n-k}$ and   let   
 $$ \pi':\cX'\to B'$$ 
be the family of curves obtained from $X'_0$ by the plumbing
construction with the same special choice of the local coordinates near
the punctures as specified in~\ref{choice-of-lc}. 
Then $X_t^\nor$ appears in this family
as the fiber of $\pi'$ at 
$t'=(t_{k+1},\ldots,t_n)\in B'$.
Therefore, the tangent vector $$v'=(v_{k+1},\ldots,v_n) \in T_{t'} B'$$ 
belongs to the kernel of the  map 
  $$T(\beta'):T_t' B' \to T_{\beta'(t')} \bsM,$$
where $\beta':B'\to\bsM_{g,n+2k}$ is  the map inducing the family $\pi'$.

Since we have already proved the \'etalness for  the smooth curve
$X^\nor$, it follows that it also holds  for $X$.

\subsubsection{The charts form an atlas}
\label{atlas-for-bM}
First of all, organize the charts $(U,A,\beta)$ constructed above
into a category as is explained in~\ref{orb-to-satake}. 
This gives a category 
$\cQ$ whose objects are triples $(U,A,\hat\beta)$, 
where  $\hat\beta:[A\bs U]\to\bsM$ is an open embedding
and whose morphisms consist of morphisms of such charts,
together with a $2$-isomorphism between their maps to $\bsM$.

Note that, according to our choice, each chart $(U,A)$ satisfies
the following property: $A$ has a fixed point in $U$. This implies,
in particular, that all maps of charts defined by arrows of $\cQ$, 
are injective. 

Let us show that the the category $\cQ$
together with the obvious  functor $$c:\cQ\to\Charts/\bM,$$
defines an orbifold atlas. The only thing to check is  the 
condition (ii) in the definition of 
orbifold atlas~\ref{satake-orbifold-atlas}.
Let 
$$[A_i\bs U_i]\to\bsM, \ i=1,2,$$ 
be two orbifold charts having a common
point $x\in\bM$ in the image. We can 
assume that $U_i$ is small enough so that $x_i$ is the only preimage
of $x$ in it. In this case the groups $A_1$ and $A_2$ can be identified 
with $A=\Aut(x)$. We will use this identification. 
Consider
$$W=\left[A\bs U_1\right]\times_{\bsM}U_2.$$
 The induced map $W\to U_2$ is an open embedding, equivariant with
 respect to the action of  $A$. 
This defines an abstract orbifold chart $(W,A)$ together with 
open embeddings
$[A\bs W]\to\relax[A\bs U_1]$ and $[A\bs W]\to\relax[A\bs U_2]$. 
Since $W$ is an open subset of $U_2$, the chart $(W,A)$
belongs to our collection.

The atlas of quasiconformal charts for $\bsM$ is constructed.

\section{Augmented \TS s from the complex-analytic point of view}
\label{sec:teich}
In this section we  study complex-analytic properties of Bers' augmented
\TS s  $\bT_{g,n}$. The space  $\bT_{g,n}$ is obtained by adding to
the classical \TS\ $\cT_{g,n}$ points  corresponding to Riemann
surfaces with nodal  singularities. 
Unlike $\cT_{g,n}$, the space $\bT_{g,n}$ is not a complex 
manifold (it is not even locally compact). 
However, as we show in this section, the quotient of
$\bT_{g,n}$ by any finite index subgroup $G$ of the \teich\ modular group  
$\Gamma_{g,n}$ is a normal complex space. 
More precisely, we prove (see  Theorem~\ref{thm:orbifoldstructure})
that $G\bs \bT_{g,n}$  has a canonical structure of a complex orbifold.  

\subsection{Complex structure on $G\bs \bT_{g,n}$: markings}
\label{chart-V}

Let $2g+n>2$ and $G$ be a finite index subgroup of the corresponding modular
group $\Gamma$.
Two markings $\phi,\phi'$ of a nodal curve $X$ are called $G$-equivalent
if there exists $g\in G$ such that $\phi'$ is isotopic to $g(\phi)$.
Points of the quotient space $G\bs\bT_{g,n}$ are pairs $(X,\phi)$,
where $X$ is a  stable curve of genus $g$ 
 with $n$ punctures and $\phi$ is
a $G$-equivalence class of markings (a $G$-marking).

We  are going
to construct an orbifold atlas for the quotient $G\bs\bT_{g,n}$. 

Shortly, the idea is the following.    We 
start with a quasiconformal orbifold atlas $\cQ$ atlas for 
the moduli stack $\bsM$ of stable curves (see Section~\ref{ss:qc}).

Then, for each chart $(U,A,\beta)\in\cQ$, endowed with an additional datum
(a marking of the singular fiber) we construct a chart $(V,H,\alpha)$
for $G\bs\bT_{g,n}$ making the diagram
$$
\begin{diagram}
V &  & &\rTo^\alpha & & &  G\bs\bT_{g,n} \\
\dTo & & & & & & \dTo \\
U &\rTo &[A\bs U] &\rTo^{\hat\beta} & \bsM & \rTo & \bM 
\end{diagram}
$$
commutative. Here $\hat\beta$ is the embedding of stacks determined by $\beta$.

Finally some work is needed to get everything arranged into an orbifold atlas 
and to prove various compatibilities.

In this subsection we 
present the construction of a chart $(V,H,\alpha)$ of $G\bs\bT_{g,n}$
based on a choice of $(U,A,\beta)\in\cQ$ and on a choice of a marking of
the special fiber of the family defined by $U$. 
We show that these charts can be arranged into an orbifold atlas 
$\cA\to\Charts(G\bs\bT)$.

As a result of the construction of the atlas, we get
a natural complex orbifold structure on the quotient $G\bs\bT_{g,n}$.
We denote the obtained orbifold by $[G\bs\bT_{g,n}]$. It is connected to
other spaces and orbifolds as shown in the diagram~(\ref{eq:smalldiag}) 
below. These connections are described in the following theorem whose 
proof occupies Sections~\ref{chart-V}--\ref{sec:projection}.

\begin{thm}\label{thm:orbifoldstructure}
Let $G$ be a finite index subgroup of the \teich\ modular group
$\Gamma=\Gamma_{g,n}$. Then the quotient space $G\bs\bT$ is the coarse
space of a naturally defined complex orbifold $[G\bs\bT]$\ so that
the quotient orbifold $[G\bs\cT]$ becomes its open substack.

The quotient map $\bT\to G\bs\bT$ factors through a map
$$\pi_G:\bT\to\relax[G\bs\bT];$$ 
the composition $[G\bs\cT]\to G\bs\bT$ factors
through $[G\bs\bT]\to G\bs\bT$ 
and the composition $[G\bs\cT]\to\overline{\sM}$ factors through 
a canonically defined morphism  $[G\bs\bT]\to\bsM$
(see the dashed arrows in the diagram~(\ref{eq:smalldiag}) below).

In particular, the quotient  $G\bs\bT$ has a natural structure of a
normal complex space extending that on  $G\bs\cT$.
\end{thm}

\begin{equation}\label{eq:smalldiag}
\begin{diagram}
 \cT & \rTo & & &\bT   \\ 
\dTo& & & \ldDashto^{\pi_G} &\dTo\\
[G\bs \cT] &\rDashto &[G\bs\bT] &  &\\
\dTo &        &\dDashto &\rdDashto &\\
G\bs\cT &\rTo & & & G\bs \bT \\
\dTo  & & & &\dTo\\
\sM & \rTo & \overline{\sM} & \rTo &  \overline{\cM} 
\end{diagram}
\end{equation}

\

\subsubsection{Space of markings of fibers}
\label{sss:marking-the-fibers}
\label{transfer-of-markings}

Let $(U,A,\beta)$ be a quasiconformal chart in $\cQ$. 
In what follows we adopt the notations of~\ref{sss:QCdef} where the 
notion of quasiconformal chart is discussed.

A collection of contractions $c_s^t:X_s\to X_t$ allows one to transfer
markings from $X_s$ to $X_t$. We will show in~\ref{prp:independence-of-c} 
below that, even though we do not fix the contractions but only require 
their existence, the transfer of markings in a quasiconformal chart is
defined uniquely.

Fix a quasiconformal contraction 
\begin{equation}\label{eq:qc-contraction}
c^t:\cX^t\to X_t
\end{equation}
and a marking 
$$\phi:S\to X_t.$$
We say that a
marking $\phi_s$ of $X_s,\ s\in U^t,$ is \emph{consistent} with the
given marking $\phi:S\to X_t$ via $c^t$
if the marking $c^t_s \circ \phi_s$ of $X_t$ is equivalent to $\phi$.

\

Fix $t$ and $\phi:S\to X_t$ as above.
For $s\in U_0^t=U_0\cap U^t$ denote by $P_s$ the set of all markings of 
$X_s$ and by $Q_s$ the subset of markings in $P_s$ consistent with $\phi$.
 The sets $P_s$ and $Q_s$ combine into coverings of $U_0^t$,

$$
p:P \rTo U_0^t  \text{\ \  and\ \ } q:Q \rTo U_0^t,
$$
so that $P_s=p^{-1}(s),\ Q_s=q^{-1}(s)$. 

The coverings $p$ and $q$ are torsors over $U_0^t$  respectively
for the groups $\Gamma$ and $\Gamma_0$,
the free abelian subgroup of $\Gamma$ generated by the Dehn
twists around the curves $\phi^{-1}(x_i)$, where $x_i,\ i=1,\ldots,r$
are the nodes of $X_t$.

The covering $q$ is   a universal covering of $U_0^t$ and
$p$ can be recovered from it as follows:
\begin{equation}\label{eq:pq}
P=\Gamma\times^{\Gamma_0}Q.
\end{equation}

\

The roles of the coverings $p$ and $q$ is explained by the
following.

\begin{Lem}
Let $\pi':\cX'\to Y$  be the family of curves induced from
$\pi:\cX\to U_0^t$ via a map $Y \to U_0^t$. Then markings of $\pi'$
correspond to sections of the covering $P'\to Y$ induced from $p$.
The sections of the covering $Q'\to Y$ induced from $q$ correspond to
the markings of $\pi'$ consistent with $\phi$.
\end{Lem}
\qed

One has a sequence of canonical maps $Q\to P\rTo^\alpha\cT$.

Note that $Q$ is a connected component of $P$; its choice depends on the choice
of the marking $\phi$. If $\gamma\in\Gamma$ then the marking 
$\phi'=\phi\circ\gamma$ corresponds to the component 
$Q'=\gamma(Q)$ of $P$.  
This gives the following geometric way of marking a curve $X_t$.
\begin{crl}
For a fixed quasiconformal contraction~(\ref{eq:qc-contraction})
$c^t:\cX^t\to X_t$,  there is a natural one-to-one correspondence
between markings of $X_t$  and components of $P$.
\end{crl}
\qed

We claim that this correspondence is independent of the choice of a
quasiconformal contraction.
To justify this we will present an independent characterization of a marking
defined by the choice of a component in $P$. We proceed as follows.

The point $t\in U^t$ admits a basis of neighborhoods $U^\delta$,
such that
$(U^\delta,A^t)$, where $A^t=\Stab_A(t))$,
is a subchart of $(U,A)$ satisfying the 
conditions (QC1)--(QC6) 
 in~\ref{satake-orbifold-atlas}.

Let $P^\delta$ and $Q^\delta$ be the spaces defined as above
with $U^\delta$ instead of $U$.
Since $U_0^t$ and $U_0^\delta=U_0^t\cap U^\delta$ have the same
fundamental groups, each component of $P$ contains precisely one component of 
$P^\delta$. Denote by $\ol{Q^\delta}$ the closure of $Q^\delta$ in the
augmented \TS\ $\bT$.

\begin{prp}
\label{prp:independence-of-c}
In the above notation, $\phi$ is the only marking of $X_t$ for which
$(X_t,\phi)$ belongs to the intersection 
$\bigcap_\delta\ol{Q^\delta}$. 
\end{prp}
The proposition immediately implies that the  notion of consistency 
of markings, defined in~\ref{transfer-of-markings} with the help of
contraction, is in fact independent of the choice of contraction.

\begin{proof}[Proof of the proposition]
First of all,  $(X_t,\phi)\in\ol{Q^\delta}$ for each $\delta$
since any neighborhood of $(X_t,\phi)$ contains $Q^\delta$ for
$U^\delta$ small enough.

Assume $(X_t,\phi')\in\bigcap_\delta\ol{Q^\delta}$. If $(X_t,\phi)$ and
$(X_t,\phi')$ represent different points of $\bT$, they have disjoint 
neighborhoods. On the other hand, by (QC2) there exist $U^\delta$ such
that   $Q^\delta$ belongs to both of them.

Thus, choosing a component $Q$ of $P$, we reconstruct the
transfer of markings from $X_s$ to $X_t$ for each $s\in U^t_0$.
The property (QC4) implies that the transfer is uniquely defined
also for any $s\in U^t$.
\end{proof}

From now on we will keep the notation  of~\ref{sss:marking-the-fibers}
for $t=z$. Thus, we have  $U^t=U,\ U^t_0=U_0$, and $p:P\to U_0,\ q:Q\to U_0$.

Note the following consequence of the above discussion.
\begin{crl}There is a one-to-one correspondence between the markings
of $X_z$ and the connected components of $P$.
\end{crl}

The following description of the covering space $P$ is very useful.

\begin{lem}\label{lem:pfibre}
We have the isomorphism
$$
P=U_0\times_\sM\cT,
$$
where the fiber product is taken in the 2-category of complex orbifolds.
\end{lem}

\begin{proof}
Let $f$ be the map $P\to U_0\times_\sM\cT$ given by the commutative diagram
$$
\begin{diagram}
P & \rTo^{\alpha_p} &  \cT \\
\dTo^p & & \dTo  \\
U_0 & \rTo^\beta & \sM
\end{diagram}
$$

Since  $p:P\to U_0$ and $U_0\times_\sM \cT\to U_0$ are coverings and
$f$ is a morphism of coverings over $U_0$,  to prove that $f$ is an 
isomorphism, it is sufficient to compare the action of the fundamental
group of $U_0$ on the fibers. 
After identification of $\pi_1(U_0)$ with $\Gamma_0$ both fibers can be
identified with $\Gamma$ and the action of $\pi_1(U_0)$ with the left action of
$\Gamma_0\subset \Gamma$ on $\Gamma$.
\end{proof}

\

As a result, the space $P$ acquires an action of the group $A$  commuting 
with the action of $\Gamma$.

\subsection{Complex structure on $G\bs\bT_{g,n}$: charts}
\label{ss:gbt-charts}

Let $G$ be a finite index subgroup of the modular group $\Gamma$.

Fix a quasiconformal orbifold chart $(U,A,\beta)$ of $\bsM$.
Fix a marking 
$$\phi:S\to X_z$$ 
(recall that this is equivalent to fixing
a connected component $Q$ of $P$). We will assign 
to the pair (chart, marking) an orbifold chart $(V,H,\alpha)$ of the quotient
$G\bs\bT$.

The marking $\phi$ determines the spaces $Q\subset P$, the isomorphism  
$\pi_1(U_0)\isom \Gamma_0$ and the presentation $P=\Gamma\times^{\Gamma_0}Q$.

The moduli stack $\bsM$ contains as an open substack the stack $\sM$ of 
non-singular curves. The triple $(U_0,A,\beta|_{U_0})$ is, of course,
a chart for $\sM$.

\subsubsection{A big commutative diagram}
\label{ABCD}
As a first step in the construction of our orbifold chart, 
we have to describe 
the spaces and the arrows of the diagram~(\ref{eq:bigdiag}) below.

The quotient $G\bs P$ can be described by the bijection $i$
\begin{equation}\label{eq:coprodexpression}
G\bs P = G\bs (\Gamma\times^{\Gamma_0} Q) \lTo^i
\coprod_{\gamma\in G\bs \Gamma/\Gamma_0}
(\gamma^{-1}G\gamma\cap\Gamma_0)\bs Q~,
\end{equation}
where $\gamma$ runs through a set of representatives
of double cosets $G\bs \Gamma/\Gamma_0$
and $i=\{i_\gamma\}$
is the collection of maps
$$
 i_\gamma: (\gamma^{-1}G\gamma\cap\Gamma_0)\bs Q \rTo G\bs
 (\Gamma\times^{\Gamma_0} Q), 
\ 
[x] \mapsto [\gamma x].
$$

Recall that $\Gamma_0$ is the free abelian group generated by the Dehn twists
$D_i,\ i=1,\ldots,r,$ around the curves $\phi^{-1}(x_i)$ of $S$,
where $x_1,\ldots,x_r$ are the nodes of $X_z$.
Let 
$$k_i=\min\{k|D_i^k\in G\}, \text{\ for\ } i=1,\ldots,r. $$
Denote by
$\Gamma'_0$ the subgroup of $G\cap\Gamma_0$ generated by $D_1^{k_1},\ldots,
D_r^{k_r}$. 
Let $Y=\Gamma'_0 \bs Q$.
The natural map 
$$Y=\Gamma'_0 \bs Q \to U_0=\Gamma_0 \bs Q$$ 
is a covering with the Galois group 
$\Gamma_0/\Gamma'_0=\Z/\Z_{k_1}\times\ldots\times\Z/\Z_{k_r}$.

We define $Z=(G\cap \Gamma_0)\bs Q$. This is the component of $G\bs P$
corresponding to $\gamma=1$.
The natural projection $Y\to Z$ gives a map 
$$u:Y\to G\bs P$$ 
commuting with the projections of $Y$ and of $G\bs P$ to $U_0$.

The projection $G\bs P\to U_0$ is a finite map. 
Now let $V$ be the normalization of $U$ 
in the field of meromorphic functions of $Y$.
The variety $V$ is a smooth; it looks locally like a polydisk
ramified over the components of the singular locus
with the ramification degree  $k_1,\ldots,k_r$.
We denote by $\kappa$ both the projection $V\to U$ and its
restriction to the smooth part $Y\to U_0$.
Let
\begin{equation}
\label{eq:induced-family}
\pi': \cX'\rTo V
\end{equation}
be the family of curves induced from the family
$\pi:\cX\to U$ via $\kappa$.

The manifold $V$ is the ``space component'' of the orbifold chart we
are building.

Let 
$$\alpha:G\bs P \to\relax[G\bs\cT]$$ 
be the map induced by 
$\alpha_p: P \to\cT$.
Now we will extend the composition 
$$Y\rTo^u G\bs P \rTo^\alpha \relax[G\bs \cT]$$ 
to get the dashed map $\alpha: V \to G\bs \bT$ in the following
big commutative diagram (do not pay attention to 
the other dashed arrows for the time being).

\begin{equation}\label{eq:bigdiag}
\begin{diagram}
& && & P & \rTo^{\alpha_p} & \cT & \rTo & & &\bT   \\ 
& && &\dTo& &\dTo& & & \ldDashto^{\pi_G} &\dTo\\
& &Y&\rTo^u&G\bs P& \rTo^\alpha & [G\bs \cT] &\rDashto &[G\bs\bT] &  &\\
&\ldTo& \dTo^\kappa& &\dTo& & \dTo & \rdTo(4,2) &\dDashto &\rdDashto &\\
V& & & & &\rDashto^{\alpha} & & & & & G\bs \bT \\
\dTo^{\kappa}& &U_0&\rTo^=&U_0& \rTo^{\hat\beta} & \sM & & & &\dTo\\
&\ldTo& & & & & &\rdTo&  & &\\
U& & & & \rTo^{\hat\beta} & & & &\overline{\sM}&
\rTo&\overline{\cM} 
\end{diagram}
\end{equation}
Here $\bsM=\bsM_{g,n}$ is the complex orbifold associated to 
the smooth algebraic stack of moduli of stable  curves of genus $g$ 
with $n$ punctures. According to Theorem~\ref{thm:gaga},
$\overline{\sM}$ represents  complex-analytic families of stable
curves of genus $g$  with $n$ punctures.  
The family of curves $\pi:\cX\to U$ defines therefore a map 
$\hat\beta:U\to\bsM$. The map from 
$\sM$ to $\bsM$ is the obvious open embedding. The complex space 
$\overline{\cM}=\overline{\cM}_{g,n}$ is the
coarse moduli space for $\bsM$ and the horizontal map 
$\overline{\sM}\to\overline{\cM}$
is obvious. The map from $[G\bs\cT]$ to $G\bs\bT$ is the composition
of the projection $[G\bs\cT]\to G\bs\cT$ to the ``na\"ive quotient'' space
and of the embedding $G\bs\cT\to G\bs\bT$. Finally, the projection 
$\bT\to\overline{\cM}$ forgetting the 
marking is continuous 
and factors through the quotient 
$G\backslash \bT$.

%% {\em Construction of $\alpha:V\to G\bs\bT$.}

The family of curves $\pi':\cX'\to V$ restricted to $Y$, admits a
canonical $G$-marking induced via $u$ from the canonical $G$-marking on the
family based on $G\bs P$. 

Choose a point  $x \in V$ and let $t=\kappa(x)\in U$. Since $G$ has finite 
index in $\Gamma$ and since the quotient $G\bs\bT$ is Hausdorff, there
exist a neighborhood $N$ of the nodes of $X_t$ and a positive $\varepsilon$
such that the standard neighborhoods $\cU_{N,\varepsilon}(X_t,\phi)$
have no intersection for different $G$-markings $\phi$.

Choose a neighborhood $U^t$ and a contraction $c^t:\cX^t\to X_t$. 
There exists a neighborhood $U^\delta$ of $t$ in $U^t$
such that $c^t_s$ is $(1+\varepsilon)$-quasiconformal outside $N$
for all $s\in U^\delta$. 
Define $\cO$ as the component of $\kappa^{-1}(U^\delta)$ containing $x$.

For $y\in\cO\cap Y$ let $\alpha(y)=(X_s,\psi)$ where we denote  $s=\kappa(y)$.
The $G$-marking  $c^t_s\circ\psi$ of $X_t$ does not depend on $y$.
This $G$-marking defines the image of $x\in V$ in $G\bs\bT$
and thus gives the required dashed map
\begin{equation}\label{eq:alpha-chart}
 \alpha:V \rTo G\bs \bT_{g,n}
\end{equation}
which is automatically continuous.
In Proposition~\ref{bar-alpha-is-open} below we will prove that $\alpha$
is open.

To get an orbifold chart $(V,H,\alpha)$ we have to specify the group $H$.

Recall that $Y=\kappa^{-1}(U_0)$ and the image of the map $u:Y\to G\bs P$
is a component $Z$ of $G\bs P$.
By Lemma~\ref{lem:pfibre} the groups $A$ and $\Gamma$ act
on $P$ and the actions commute. Thus, $A$ acts on the quotient $G\bs P$ as 
well.
Let
\begin{equation}
\label{eq:az}
A_Z=\{a\in A : a(Z)=Z \}
\end{equation}
be the stabilizer of the component $Z$.

We define $H$ as the group of pairs $(\wt{a},a),$ 
where $\wt{\alpha}:Y\to Y$ and $a:Z\to Z$ are automorphisms
with $a\in A_Z$ such that the diagram
$$
\begin{diagram}
Y & \rTo & Z \\
\dTo^{\wt{a}} & & \dTo^{a} \\
Y & \rTo & Z
\end{diagram}
$$
is commutative.

Another description of the groups $A_Z$ and $H$ is given in~\ref{another-H}.

The action of $H$ on $Y$ extends to $V$, since $V$ is the normalization
of $U$ in the field of meromorphic functions on $Y$, see~\cite[7.3]{gr}.
Now we will prove that $(V,H,\alpha)$ is an orbifold chart for $G\bs\bT$.

\begin{prp}
\label{bar-alpha-is-open}
The map 
$$\alpha:V\rTo G\bs\bT$$ 
is open.
\end{prp}
\begin{proof}

We will prove that the image $\alpha(V)$ is open in $G\bs\bT$.
Since we can replace the chart $(U,A,\beta)$ with a smaller
quasiconformal chart, this will prove that
$\alpha$ carries any open set to an open set.

Let $x\in V$ and let $\alpha(x)$ be presented by a marked curve $(X,\phi)$.
We have to prove that there is a pair $(N,\varepsilon>0)$ where $N$ is a 
neighborhood of the nodes of $X$, so that the standard neighborhood 
$\cU_{N,\varepsilon}(X,\phi)$ 
of $(X,\phi)$ in $G\bs\bT$ lies in $\alpha(V)$.

Let $t=\kappa(x)$. By (QC2) there exists an open neighborhood $U^t$ of $t$
in $U$ and a  contraction $c^t:\cX^t\to X_t$.
The map $\beta:U\to\bM$ is open; thus, a pair $(N,\varepsilon)$
can be chosen so that $\cU_{N,\varepsilon}(X,\phi)$ lies in 
$\pi^{-1}(\beta(U^t))$,
where $\pi:G\bs\bT\to\bM=\Gamma\bs\bT$ is the standard projection.

Since $G$ has finite index in $\Gamma$, we can also assume that
the neighborhoods $\cU_{N,\varepsilon}(X,\phi)$ and 
$\cU_{N,\varepsilon}(X,\phi')$ have no intersection if $\phi$ and $\phi'$
define different $G$-markings.

We claim that $\alpha(V)$ contains the neighborhood 
$\cU_{N,\varepsilon}(X,\phi)$.

In fact, let $(X',\phi')\in\cU_{N,\varepsilon}(X,\phi)$. 
By construction, there exists $s\in U^t$ with $\beta(s)$ represented by $X'$.
Let  $U^\delta$ be a small neighborhood of $s$ contained in 
$U^s\cap U^t$ and consider $U_0^\delta=U^\delta\cap U_0$. 
Choose a point $u\in U_0^\delta$, lift it to a point $y\in Y$ and consider
$\alpha(y)=(X'',\phi'')$. The contraction of $\phi''$ to $X$
gives $\phi$, therefore, the contraction of $\phi''$ to $X'$
gives $\phi'$ up to an element of the group $\Gamma_0$. This means that
we can find another lift $y_1$ of $u$ to $Y$ with $\alpha(y_1)=(X'',\phi''_1)$
such that the contraction of $\phi''_1$ to $X'$ will be $\phi'$.

Let $Y^\delta$ 
be the component of $\kappa^{-1}(U_0^\delta)$
containing $y_1$. Then the intersection 
$\bar{Y^\delta}\cap\kappa^{-1}(s)$
consists of one point $z$  such that  $\alpha(z)=(X',\phi')$.
\end{proof}

\begin{lem}\label{lem:lift}
The homomorphism $H \to A_Z $ is surjective with the kernel
$$
K = \Aut_Z (Y).
$$
\end{lem}

\begin{proof}
We have to verify that any automorphism $a:Z\to Z$ from $A_Z$ lifts
to an automorphism of $Y$. We have the following picture.
Three spaces, $U_0,\ Z$ and $Y$ have a common universal covering $Q$.
The fundamental group of $U_0$ is $\Gamma_0$, and the coverings
$Y$ and $Z$ correspond to the subgroups $\Gamma'_0$ and 
$G\cap\Gamma_0$ of $\Gamma_0$.

Let $a\in A_Z\subset A$.  Since $A$ is abelian, its action on $U_0$
induces an action on $\Gamma_0$. 
The action of $A$ on $U_0$ comes from an action on $U$,
therefore its action on $\Gamma_0$ must be
a signed permutation of the Dehn twists $D_1,\ldots,D_r$ which
generate $\Gamma_0$.

If an element $a\in A$ belongs to $A_Z$ then
$G\cap\Gamma_0$ is an $a$-invariant subgroup of $\Gamma_0$.
This implies that $\Gamma'_0$ is 
also $a$-invariant 
due to the specific form 
of the action of $A$ on $\Gamma_0$.

Note that the kernel $K$ of the epimorphism $H \to A_Z $ identifies
with $(G\cap \Gamma_0)/\Gamma'_0$.
\end{proof}

\

\begin{thm}\label{th:eq}
Let $x_1, x_2 \in V$, then $\alpha(x_1)=\alpha(x_2)$
if and only if $x_2 \in H x_1$.
\end{thm}

\begin{proof}
Let  $(X_i,\phi_i)$, $i=1,2$, be the marked curves representing
the points $\alpha(x_i)\in G\bs \bT$.

The equality $\alpha(x_1)=\alpha(x_2)$
gives an isomorphism of $G$-marked curves, that is a commutative diagram
\begin{equation}\label{diag:x12}
\begin{diagram}
S & \rTo^{\phi_1} & X_1\\
\dTo^g&&\dTo_\theta\\
S & \rTo^{\phi_2} & X_2
\end{diagram}
\end{equation}
for some $g\in G$.

We will show that for any two open sets $U_1\ni x_1$ and
$U_2\ni x_2$ in $V$ the intersection 
$U_2 \cap (H\cdot U_1)$ is nonempty.
Since $H$ is a finite group, this will imply that $x_2\in Hx_1$.

\

From now on we fix $x_i$ and $U_i$, $i=1,2$ as above.

\
Since the map $\alpha:V\to G\bs\bT$ 
is open by Proposition~\ref{bar-alpha-is-open},
we may assume that 
$\alpha(U_1)=\alpha(U_2)$. Choose a point $x'_1\in U_1$
which corresponds
to a smooth curve $(X'_1,\phi'_1)$. Since the images
of $U_i$ under $\alpha$ coincide, there exists $x'_2\in U_2$
having the same image. The corresponding $G$-marked curve $(X'_2,\phi'_2)$
is, obviously, smooth as well. Moreover, there exist $g'\in G$ and
an isomorphism $\theta': X_1'\to X_2'$ such that the diagram
\begin{equation}\label{diag:xprime}
\begin{diagram}
S & \rTo^{\phi'_1} & X'_1\\
\dTo^{g'}&&\dTo_{\theta'}\\
S & \rTo^{\phi'_2} & X'_2
\end{diagram}
\end{equation}
is commutative.

The images of $x'_1$ and  $x'_2$ in $Z$ have the same image in
$G\bs \cT$. Therefore,
$x_2'=a x_1'$ for some $a\in A_Z$.
By Lemma~\ref{lem:lift} we can lift the element $a$ to $h\in H$ acting on
$Y$. Therefore $x_2'$ and $h(x'_1)$ have the same image in $Z$.
This implies that there exists an element $h'$ in the kernel $K$ of the
epimorphism $H\to A_Z$ such that $x_2'=h' h x_1'$. This concludes the
proof of the theorem.
\end{proof}

\subsubsection{Second description of the groups $A_Z$ and $H$}
\label{another-H}
Here we will  present yet another interpretation of the groups $A_Z$
and $H$ which appear
in the description of the orbifold charts $(V,H,\alpha)$.
This description will be  needed in Section~\ref{sec:teich-vs-adm}  
where we use a slightly more general quotient of $\bT$ 
than  the one described here.

Let $A_Q$  be the group of pairs $(\wt{a},a)$, where $a\in A$ and 
$\wt{a}:Q\to Q$ satisfies the condition $q\circ\wt{a}=a\circ q$,
where $q:Q\to U_0$.

The natural map $[A_Q\bs Q]\to\relax{[A\bs U_0]}$ is an equivalence. 
Thus, the map $Q\to\cT$ induces a map of the quotients
$[A_Q\bs Q]\to\relax{[\Gamma\bs\cT]}$. By Lemma~\ref{lem:explicit-asso}
this gives rise to a homomorphism $\iota:A_Q\to\Gamma$.
This homomorphism is uniquely  determined by the requirement of 
commutativity of the diagram
$$
\begin{diagram}
Q & \rTo^{\alpha_Q} & \cT \\
\dTo^a & & \dTo^{\iota(a)} \\
Q & \rTo^{\alpha_Q} & \cT 
\end{diagram}
$$

Consider $A_{Q,G}=A_Q\times_\Gamma G$. We claim that the image of the 
composition 
$$A_{Q,G}\to A_Q\to A$$ 
is precisely $A_Z$, so that we have got a morphism of short 
exact sequences 
\begin{equation}
\label{eq:2ses}
\begin{diagram}
1& \rTo& \Gamma_0\cap G& \rTo& A_{Q,G}&\rTo& A_Z &\rTo& 1 \\
& & \dTo & & \dTo & & \dTo & & \\
1& \rTo& \Gamma_0& \rTo& A_Q&\rTo& A &\rTo& 1 
\end{diagram}.
\end{equation}
In fact, since $A_{Q,G}$ acts on $Q$, the quotient 
$A'_Z=A_{Q,G}/(\Gamma_0\cap G)$
acts on $Z=(\Gamma_0\cap G)\bs Q$ and is a subgroup of $A$. Thus
$A'_Z\subseteq A_Z$. Since the composition 
$$[A'_Z\bs Z]\to\relax[A_Z\bs Z] \to\relax[G\bs\cT]$$ 
is an open embedding, one should necessarily have
$A'_Z=A_Z$.

As it was explained in the proof of~\ref{lem:lift}, the group $\Gamma'_0$
is normal in $A_{Q,G}$. Passing to the quotient by $\Gamma'_0$
in the upper line of~(\ref{eq:2ses}), one gets the short exact
sequence
$$1\rTo  (\Gamma_0\cap G)/\Gamma'_0\rTo A_{Q,G}/\Gamma'_0 \rTo A_Z\rTo 1$$
which identifies with the sequence
\begin{equation}
\label{eq:1ses}
1\rTo K\rTo H\rTo A_Z\rTo 1
\end{equation}
defined by Lemma~\ref{lem:lift}.

Recall that $Y=\Gamma'_0\bs Q$. Thus, one has an open embedding
\begin{equation}
\label{eq:open-embedding-hy}
[H\bs Y]=[A_{Q,G}\bs Q]=[A_Z\bs Z]\rTo\relax[G\bs\cT].
\end{equation}

\subsection{Orbifold atlas for $G\bs\bT$}
\label{chartsarecompatible}
 
We now have a sufficient supply of orbifold charts for constructing 
an orbifold atlas for $G\bs\bT$. 
In order to arrange the constructed orbifold charts into an atlas,
we have to present a chart category $\cA$ and a functor
$c:\cA\to\Charts/(G\bs\bT)$  satisfying the properties
of~\ref{satake-orbifold-atlas}. 

The chart $(V,H,\alpha)$ constructed above depends on the 
following choices.
\begin{itemize}
\item[1)] A chart $(U,A,\beta)$ in $\cQ$. The singular locus
$U-U_0$, where $U_0=\sM\times_{\bsM}U$, is a normal crossing divisor.
The group $A$ acts on $U$ with a fixed point $z$ (belonging to the 
intersection of the components)
\item[2)] A marking $\phi$ of the curve $X_z$ (or, what is equivalent,
a choice of a component of $P$).
\end{itemize}

%\subsubsection{Atlas for $G\bs\bT$}
%\label{atlas-for-GbT}
The chart category $\cA$ will be constructed simultaneously
with a functor $p:\cA\to\cQ$, so that if $c(p(a))$ is the chart 
$(U,A,\beta)$ of $\bsM$, $c(a)$ is the chart $(V,H,\alpha)$
constructed in~\ref{chart-V}.

Recall some notation from Section~\ref{ss:gbt-charts}.
The chart $V$ contains a dense open $H$-equivariant subset $Y=\Gamma'_0\bs Q$
giving  
%which, together with therestriction of $\alpha:V\rTo G\bs\bT$, defines
an open embedding
\begin{equation}
\label{obj-tA}
\hat\alpha:[H\bs Y]\rTo \relax[G\bs\cT].
\end{equation}

We define $\cA$ as the category whose objects are open 
embeddings~(\ref{obj-tA}), where $Y$ and $H$ 
are obtained from an orbifold chart 
$(V,H,\alpha)$ described above. A morphism in $\cA$ from 
$(Y_1,H_1,\hat\alpha_1)$
to $(Y_2,H_2,\hat\alpha_2)$ is defined as a
morphism of abstract orbifold charts
$$ 
(f_Y,f_H):(Y_1,H_1)\rTo (Y_2,H_2)
$$
together with a $2$-morphism
$$ \theta_Y:\hat\alpha_2\circ f_Y\isom \hat\alpha_1.
$$

Note that any object $(Y,H,\hat\alpha)$ of $\cA$ gives rise to a 
$2$-commutative diagram
\begin{equation}
\label{YZU}
\begin{diagram}
Y & \rTo^\alpha & [G\bs\cT] \\
\dTo & & \dTo \\
U & \rTo^\beta & \bsM
\end{diagram}
\end{equation}
so that the assignment $(Y,H,\hat\alpha)\mapsto (U,A,\hat\beta)$
defines a functor $\cA\to\cQ$. In fact, $Y=\Gamma'_0\bs Q$ and the
morphism $\hat\alpha:[H\bs Y]\to\relax[G\bs\cT]$ can be realized
by a pair of morphisms $(q:Q\to\cT, A_{Q,G}\to G).$ The pair
$(q:Q\to\cT,\ A_Q\to\Gamma)$ is compatible it
on one side and with $\hat\beta:[A\bs U]\to\bsM$ on
 the other side.
 
Furthermore, any map 
$$\eta:(Y_1,H_1,\hat\alpha_1)\rTo(Y_2,H_2,\hat\alpha_2)$$
in $\cA$ lifts to a map 
$\eta:Q_1\to Q_2$ which by Proposition~\ref{prp:hom-asso}
defines a unique diagram
\begin{equation}
\begin{diagram}
Q_1 & \rTo^{q_1} & \cT \\
\dTo^\eta & & \dTo^g \\
Q_2 & \rTo^{q_2} & \cT \\
\end{diagram}
\end{equation}
where $g\in G$. This diagram induces a morphism of orbifold charts
$$\bar\eta:(U_{01},A_1,\hat\beta_1)\rTo(U_{02},A_2,\hat\beta_2)$$
compatible with $\eta$.

Note that since $U_{0i}$ is the smooth locus of $U_i$,
any map 
$(U_1,A_1,\hat\beta_1)\to (U_2,A_2,\hat\beta_2)$ 
in $\cQ$ 
carries $U_{01}$ to $U_{02}$, so that $\bar f$ extends uniquely to a
morphism $(U_1,A_1,\hat{\beta}_1)\to(U_2,A_2,\hat\beta_2)$.

Let us prove that the above morphisms in $\cA$ 
give an orbifold atlas
for $G\bs\bT$. First of all, the category $\cA$ is a chart category
since it is a full subcategory of the chart category defined by
the orbifold $[G\bs\cT]$ via~\ref{orb-to-satake}.

Let us show that an arrow of $\cA$ defined as above, gives a
morphism of the corresponding orbifold charts. First, $\eta_Y:Y_1\to Y_2$ 
uniquely defines
a map $\eta_V:V_1\to V_2$ since $V_i$ can be identified as the normalization
of $U_i$ in the field of meromorphic functions on $Y_i$. 
The map of abstract orbifold
charts $(V_1,H_1)\to (V_2,H_2)$ is automatically defined since $\cA$
is a chart category. To check that we have a map of orbifold charts over 
$G\bs\bT$, we have to check that the map $\eta_V: V_1\to V_2$ is compatible
with the projections $\alpha_i:V_i\to\ G\bs\bT$. This is enough
to check on the dense subset $Y_1$ of $V_1$ where compatibility follows from 
the definition.

The required collection of isomorphisms 
$\iota:\Aut(a)\to H(a),\ a\in \cA$,
comes from the construction of $\cA$ as a full subcategory
of the chart category for $[G\bs\cT]$.

Let us check that the images of the charts $(V,H,\alpha)$ cover the whole
space $G\bs\bT$. Let $(X,\psi)$, where $X$ is a curve and $\psi$ is a 
$G$-marking of $X$, represent a point of $G\bs\bT$. Choose a quasiconformal 
chart $(U,A,\beta)$ of $\bsM$ containing $x\in U$ with $\beta(x)=X$. 
A choice of representative for the $G$-marking $\psi$ defines a marking 
$\phi:S\to X_z$ and therefore a chart $(V,H,\alpha)$ for $G\bs\bT$.
If $y\in V$ is a lifting of $x$, its image $\alpha(y)$ is a pair $(X,\psi')$.
The $G$-markings $\psi$ and $\psi'$ define the same $G$-marking $\phi$ on 
$X_z$. Therefore, they differ by an element $\gamma\in\Gamma_0$.
Since $\Gamma_0$ acts on $V$,\footnote{via $\Gamma_0/\Gamma'_0$}
the point $\gamma(y)$ has the required image in $G\bs\bT$.

The last thing to be checked is the condition (ii) of 
definition~\ref{satake-orbifold-atlas}.

Let $x\in G\bs\bT$ belong to the images of the orbifold charts 
$(V_i,H_i,\alpha_i),\ i=1,2$. Then the image $y$ of $x$ in $\bM$
is covered by $(U_i,A_i,\beta_i),\ i=1,2$. If $y_i\in U_i$ are 
preimages of $y\in\bM$, we can assume as in~\ref{atlas-for-bM}
that there exists an isomorphism of charts
$\eta:(U_1,A_1,\beta_1)\to (U_2,A_2,\beta_2)$ sending $y_1$ to $y_2$. 

Let $(X_1,\phi_1)$ and $(X_2,\phi_2)$ represent the curves at the points
$y_1$ and $y_2$ of $U_1$ and $U_2$. Since both marked curves represent
the same point  $x\in G\bs\bT$, there exist an isomorphism $\theta:X_1\to X_2$
and an element $g\in G$ making the diagram ~(\ref{diag:x12}) commutative.

The isomorphism $\eta:U_1\to U_2$ commutes with the maps
$\hat\beta_i:U_i\to\bsM$. Therefore, $\eta$ induces an isomorphism
\begin{equation}\label{etasmooth}
\eta:U_{01}\rTo U_{02}
\end{equation}
The isomorphism~(\ref{etasmooth}) induces an isomorphism of the
corresponding fundamental groups so that the Dehn twists defined by
the nodes of $X_1$ map (up to sign) to the Dehn twists defined by 
the corresponding nodes of $X_2$. In particular, this implies that 
the numbers $k_1,\ldots,k_r$
defining the coverings $Y_1$ and $Y_2$, coincide.

Also, an isomorphism $\eta:Q_1\to Q_2$ of the covering spaces of $U_{0i}$
is induced so that the diagram
$$
\begin{diagram}
Q_1 & \rTo &  \cT \\
\dTo^\eta & & \dTo^g\\
Q_2 & \rTo &  \cT 
\end{diagram}
$$
is commutative. This defines an isomorphism of the factors
$$ Y_1=\Gamma'_0\bs Q_1\rTo \Gamma'_0\bs Q_2=Y_2.$$

This induces an isomorphism of $V_i$ since $V_i$ is the normalization of $U_i$
in the field of meromorphic functions of $Y_i$.

\section{Properties of orbifolds $G\bs \bT$}
\label{sec:properties}

In this section we establish some properties of the orbifold structure
on  $G\bs\bT$ introduced in the previous section.
We start with showing that for the subgroup
$$
\Gamma^{(\ell)}=\Ker(\Gamma\rTo\Aut(H_1(S,\Z/\ell))).
$$
the orbifold $[\Gamma^{(\ell)}\bs\bT]$
corresponds to the moduli  stack of 
curves with level-$\ell$ structures. 
Besides of providing an interesting example, this fact will be used
in~\ref{covered-by} to construct the canonical map
$$\pi_G:\bT\to \relax[G\bs \bT]$$ for an arbitrary finite-index 
subgroup $G \subset \Gamma.$

We also construct here gluing operations on the orbifolds
$[G\bs\bT]$ which are induced by gluing operations for bordered
surfaces. 
\subsection{Example: level-$\ell$ curves}
\label{levelcurves}
Let $\ell>2$ be a natural number. Define
$$\Gamma^{(\ell)}=\Ker(\Gamma\rTo\Aut(H_1(S,\Z/\ell))).$$
$\Gamma^{(\ell)}$-marking on a smooth curve $X$ is the same as a level-$\ell$ 
structure on $X$.

Let $(X,\phi:S\to X)$ represent a $\Gamma^{(\ell)}$-marked nodal Riemann
surface. 
Choose a quasiconformal neighborhood $(U,A,\beta)$ of $X\in\bsM$
and construct a corresponding chart $(V,H,\alpha)$ for 
$[\Gamma^{(\ell)}\bs \bT]$ as in~\ref{ABCD}. We assume that
$\alpha(s)=(X,\phi)$ for some $s\in V$.

Recall that the group $H$ is an extension
$$ 1\rTo K\rTo H\rTo A_Z\rTo 1,$$
where $A_Z$ is the subgroup of the automorphism group $A$ of 
$X\in\bsM$ stabilizing the component $Z$.

\begin{prp}
For $G=\Gamma^{(\ell)}$ one has  $A_Z=1$. 
\end{prp}
\begin{proof}
Let $h\in H$. We will check that the image $a\in A_Z$ of $h$ 
induces a trivial action on the homology $H_1(X,\Z/\ell)$.
This will imply that $a=1$ since the automorphism group $A$ of $X$
acts faithfully on $H_1(X,\Z/\ell)$. Since the map $H\to A_Z$ is surjective,
this will imply our claim.

Recall  that $Y=\varkappa^{-1}(U_0)\subset V$. Choose $x\in Y,\ y=h(x)$
and let $(X_x,\phi_x),\ (X_y,\phi_y)$ be the corresponding $G$-marked
Riemann surfaces. The map $\alpha$
from diagram~(\ref{eq:bigdiag}) induces an open embedding
$$ \left[H\bs Y\right]\rTo\relax[G\bs\cT].$$
This gives rise to the following commutative diagram
\begin{equation}
\label{eq:right-part}
\begin{diagram}
H_1(X_x,\Z/\ell) & \lTo^{H_1(\phi_x)} & H_1(S,\Z/\ell) \\
\dTo^a      & & \dEQ \\
H_1(X_y,\Z/\ell) & \lTo^{H_1(\phi_y)} & H_1(S,\Z/\ell) \\
\end{diagram}
\end{equation}

On the other hand, the family of curves with the base $V$ defines
a morphism $V\to U\to\bsM$. The element $h\in H$ induces an automorphism
$a\in A$ of the family. This implies the commutativity of the diagram
\begin{equation}
\label{eq:left-part}
\begin{diagram}
H_1(X,\Z/\ell) & \lTo^{v_x} & H_1(X_x,\Z/\ell) \\
\dTo^a      & & \dTo^a \\
H_1(X,\Z/\ell) & \lTo^{v_y} & H_1(X_y,\Z/\ell) \\
\end{diagram},
\end{equation}
where the horizontal arrows are the vanishing cycles maps.

We will show  later that
\begin{equation}
\label{compatibility-vc}
v_x\cdot H_1(\phi_x)=v_y\cdot H_1(\phi_y):H_1(S,\Z/\ell)\rTo
H_1(X,\Z/\ell).
\end{equation}
Then comparing the diagrams (\ref{eq:right-part}) and 
(\ref{eq:left-part}) we see that 
$$a:H_1(X,\Z/\ell)\to H_1(X,\Z/\ell)$$
is the identity, which yields the claim.

Let us now explain~(\ref{compatibility-vc}).

Let $\pi:\cX\to V$ be the family of curves  described in~(\ref{eq:induced-family}) 
(where the notation $\cX'$ was used instead of $\cX$).
One has $X=\pi^{-1}(s), \ X_x=\pi^{-1}(x),\ X_y=\pi^{-1}(y)$.
Let $j_s,\ j_x,\ j_y$ be the respective embeddings of $X,\ X_x,\ X_y$ into 
$\cX$. The space $\cX$ contracts to $X$, so $j_s$ is a homotopy equivalence. 
Consider the diagram
\begin{equation}
\label{1-homol}
\begin{diagram}
H_1(S,\Z/\ell) & \rTo^{H_1(\phi_y)} & H_1(X_y,\Z/\ell) & &  \\
\dTo^{H_1(\phi_x)} & & \dTo^{H_1(j_y)} & &\\
H_1(X_x,\Z/\ell) & \rTo^{H_1(j_x)} & H_1(\cX,\Z/\ell)& \lTo^{H_1(j_s)}_\sim & 
H_1(X,\Z/\ell)  
\end{diagram}.
\end{equation}
The vanishing cycle homomorphism is the composition $j_s^{-1}j_x$; therefore, 
the compatibility~(\ref{compatibility-vc}) is equivalent to the commutativity
of the diagram~(\ref{1-homol}). 

Finally, commutativity of~(\ref{1-homol}) can be shown as follows. The 
restriction of the family $\pi:\cX\to V$ to $Y$ is locally trivial; thus,
the assignment 
$$x\in Y\mapsto H_1(X_x,\Z/\ell)$$ is a local system on $Y$.
The maps 
$$H_1(S,\Z/\ell) \rTo^{H_1(\phi_x)}H_1(X_x,\Z/\ell)\rTo^{H_1(j_x)}H_1(\cX,\Z/\ell)$$
give rise to a map of constant local systems
$H_1(S,\Z/\ell) \to H_1(\cX,\Z/\ell)$ which therefore does not depend on $x\in Y$.

\end{proof}

\subsection{Functoriality with respect to $G$}
\label{sec:projection}
In this section we will prove that the orbifold structure on 
spaces $G\bs\bT$ is natural with respect to a subgroup $G$ of
the modular group $\Gamma$. Then we will use this fact
to produce in~\ref{covered-by} the map
$\pi_G:\bT\to\relax[G\bs\bT]$ 
from the big commutative diagram~(\ref{eq:bigdiag}).

\subsubsection{A canonical map  
$[G_1\bs \bT]\to[G_2\bs\bT]$} 
\label{sss:2groups}
Let $G_1\subset G_2$ be two finite index subgroups of $\Gamma$.
Then a canonical map
\begin{equation}
\label{eq:2groups}
 [G_1\bs \bT]\rTo\relax[G_2\bs\bT]
\end{equation}
can be constructed
as follows. Starting from an orbifold chart $(U,A,\beta)\in\cQ$ of 
$\bsM$, we get as in~\ref{ABCD} the charts $(V_i,H_i,\alpha_i),\ i=1,2,$ 
and a compatible pair
of maps $V_1\to V_2,\ H_1\to H_2$. Thus, a map of charts
$(V_1,H_1,\alpha_1)\to (V_2,H_2,\alpha_2)$ is canonically defined, 
giving finally a map of 
orbifolds~(\ref{eq:2groups}). 

The group homomorphism $H_1\to H_2$ appears in the
commutative diagram whose construction is obvious.
$$
\begin{diagram}
1 & \rTo & K_1 & \rTo & H_1 & \rTo & A_{Z_1} & \rTo & 1 \\
  &      & \dTo&      & \dTo&      & \dTo    &      & \\
1 & \rTo & K_2& \rTo & H_2 & \rTo & A_{Z_2} & \rTo & 1 \\
\end{diagram}
$$

The map $A_{Z_1}\to A_{Z_2}$ is injective. This implies that if, for instance,
$G_2=\Gamma^{(l)},\ l\geq 3$, then $A_{Z_1}=1$.

Note that the map~(\ref{eq:2groups}) is seldom \'etale.

The following result generalizes 
\cite[Proposition 3]{Loo}.

\begin{prp}
\label{asloo}
For each positive integer $k$ there exists a finite index subgroup 
$\Gamma_{(k)}$
of the modular group $\Gamma$ satisfying the following property.
For each collection $D_1,\ldots,D_m$ of independent Dehn twists
the intersection of $\Gamma_{(k)}$ with the group generated by 
$D_1,\ldots,D_m$ is 
generated by some powers $D_1^{k_1},\ldots,D_m^{k_m}$ where all $k_i$ 
are divisible by $k$.
\end{prp}
\begin{proof}
The case $n=0$ follows from Looijenga's  result~\cite[Proposition 3]{Loo}
Here is the definition of $\Gamma_{(k)}$ for $n=0$. Let $\wt{S}\to S$
be a universal Prym cover, i.e. a Galois cover with 
$\Gal(\wt{S}/S)=H^1(S,\Z/2)$ considered as the quotient of 
$\pi_1(S)$ by the normal subgroup generated by the squares of the elements.

Without loss of generality we can assume that $k$ is even and $k\geq 6$.
The group $\Gamma_{(k)}$ 
is then the group of $\gamma\in\Gamma$ whose
(arbitrary) lift $\wt{\gamma}:\wt{S}\to\wt{S}$ acts on $H^1(\wt{S},\Z/k)$ 
as an element of $\Gal(\wt{S}/S)$. 
By~\cite[Proposition 3]{Loo} the group $\Gamma_{(k)}$ satisfies the
requirements of the proposition:  its intersection with a group
generated by $D_1,\ldots,D_m$ is the group generated by
$D_1^{k_1},\ldots,D_m^{k_m}$ where  $k_i=k$ if $D_i$ disconnects $S$, and $k_i=2k$ otherwise.

The general case will be reduced to the case $n=0$.

Let $S$ be a  compact oriented surface of genus $g$ with 
 $n$ boundary components. 
We define a new surface $T$ as the result of gluing $S$ to $-S$ along
the boundary in an obvious way. Thus $T$ has no boundary and it is of
genus $2g+n-1$. Any diffeomorphism $\phi$ of $S$ preserving the boundary
defines a diffeomorphism $\Delta(\phi)$ of $T$ acting as $\phi$ on both $S$ 
and $-S$. This construction preserves isotopy, and, therefore, induces
a homomorphism of the modular groups
\begin{equation}
\label{map-of-gammas}
\Delta:\Gamma_S\rTo\Gamma_T
\end{equation}
of $S$ and of $T$ respectively.
Lemma~\ref{lem:mono-modgroups} below claims that $\Delta$ is injective.
Then we define the subgroup $\Gamma_{(k),S}$ of $\Gamma_S$ as
$\Delta^{-1}(\Gamma_{(k),T})$.
 
If $D_1,\ldots,D_m$ are independent Dehn twists in $\Gamma_S$, one has
$2m$ independent Dehn twists $D^{\pm}_i,\ i=1,\ldots,m$ in $T$ defined 
by the corresponding circles in $S$ and in $-S$. 
By~\cite[Proposition 3]{Loo}
an element $\prod (D^+_i)^{k^+_i}\prod (D^-_i)^{k^-_i}$ belongs
to $\Gamma_{(k),T}$ if and only if $k^{\pm}_i$ are divisible by $k$ or by $2k$,
depending on $i$.\footnote{More precisely, this is $k$ if $D^+_i$
  disconnects $T$, and $2k$ if it does not.}  
Since $\Delta(\prod D_i^{k_i})=\prod (D^+_i)^{k_i}\prod (D^-_i)^{k_i}$,
we get the required property.
\end{proof}

Now we will prove injectivity of $\Delta$.
\begin{lem}
\label{lem:mono-modgroups}
The map $\Delta:\Gamma_S\to\Gamma_T$ defined in~(\ref{map-of-gammas}) 
is injective.
\end{lem}
\begin{proof}
Denote $\pi=\pi_1(S),\ \Pi=\pi_1(T)$. We choose as the base point for both
$S$ and $T$ a boundary point of $S$. The embedding $S\to T$ admits an
obvious section which identifies $-S$ with $S$. Thus, the embedding
$i:\pi\to\Pi$ of fundamental groups induced by the embedding $S\to T$
splits by a projection $\rho:\Pi\to\pi$. 

The modular groups $\Gamma_S$ and $\Gamma_T$ act by outer automorphisms
on the corresponding fundamental groups $\pi$ and $\Pi$;
the canonical maps 
$$\alpha_S:\Gamma_S\to\Out(\pi),
\text{\ and \ } \alpha_T:\Gamma_T\to\Out(\Pi)$$ 
are well-known to be injective.

Define a map (this is not a group homomorphism!) 
$$\nabla:\Aut(\Pi)\to \Aut(\pi)$$ 
by the formula $\nabla(\phi)=\rho\circ\phi\circ i$.
One has for $\phi\in\Aut(\Pi),\ g\in\Pi$,
$$\nabla(\ad(g)\circ\phi)=
\rho\circ\ad(g)\circ\phi\circ i=
\ad(\rho(g))\circ\nabla(\phi).$$
Thus $\nabla$ induces a map 
$\nabla:\Out(\Pi)\to\Out(\pi)$. We claim that
the diagram 
$$
\begin{diagram}
\Gamma_S & \rTo^\Delta & \Gamma_T \\
\dTo &          & \dTo \\
\Out(\pi)& \lTo^\nabla & \Out(\Pi)
\end{diagram}
$$
is commutative. In fact, if $\gamma:S\to S$ defines an element of $\Gamma_S$
then $\alpha_S(\gamma)$ sends a loop $u\in\pi$ to
$\gamma(u)$.\footnote{$\alpha_S$ can be actually defined as a
  homomorphism to $\Aut(\pi)$ since we chose the base point preserved
  by any $\gamma$.} 
On the other hand, $\nabla\circ\alpha_T(\Delta(\phi)$ sends $u\in\pi$
to $\rho(\Delta(\phi)(u))=\phi(u)$.
Thus $\Delta$ is injective since $\alpha_S$ is injective.
\end{proof}

\subsubsection{Construction of the map $\pi_G:\bT\to [G\bs \bT]$}
\label{covered-by}

Now we will use the reasoning of~\cite[Corollary 2.10]{BoPi} 
to prove that for any finite index subgroup $G\subset \Gamma$ there
exists a smaller finite index subgroup $H\subset G\subset \Gamma$
such that  $[H\bs\bT]$ is a manifold. 
This will allow to construct the canonical map
$\pi_G:\bT\to\relax[G\bs\bT]$ 
from the big commutative diagram~(\ref{eq:bigdiag}).

Let $G$ be any finite index subgroup of the modular group $\Gamma$. 
The intersection
$G'=G\cap\Gamma^{(l)}$ has also finite index. 
Let 
$$G''=\bigcap_{g\in\Gamma}gG'g^{-1}.$$
This is a normal subgroup of $\Gamma$ contained in $G$ and having a finite 
index. Since there are only finitely many 
Dehn twists in $\Gamma$ up to conjugation, there exists $k$ such that
for each Dehn twist $D$ one has $D^k\in G''$.     
Finally, consider the subgroup $G'''=G''\cap\Gamma_{(k)} $. 
We claim that the quotient $[G'''\bs \bT]$ is a complex manifold.

Look at a chart $(V,H,\alpha)$ constructed as in in~\ref{ABCD}
for the quotient $G'''\bs\bT$. Recall that
the group $H$ appears as the extension of $A_Z$ with 
the quotient $K=(G'''\cap\Gamma_0)/\Gamma'_0$, see the notation
of~\ref{ABCD}--\ref{lem:lift}. The group $A_Z$ is 
trivial since $G'''\subseteq\Gamma^{(l)}$, see~\ref{sss:2groups}. 
Let us show the group $K$ is trivial. Let $\gamma=\prod_{i=1}^rD_i^{d_i}\in G'''\cap
\Gamma_{00}$. Then $\gamma\in\Gamma_{(k)}$ since $G'''\subseteq\Gamma_{(k)}$. 
Therefore, $d_i$ are all divisible by $k$ and $D_i^{d_i}\in\Gamma_{(k)}$.
By the choice of $k$ $D_i^{d_i}\in G''$ as well, so they belong
to $\Gamma'_0$.

Thus, we see that $H=1$. Therefore, $[G'''\bs \bT]$ is a complex manifold.

Assume now that $H_1,\ H_2$ are two finite index subgroups of
$G\subset\Gamma$ such that $[H_i\bs\bT]$ are manifolds for $i=1,2$.
Then the intersection $H_1\cap H_2$ contains as well a finite index subgroup
$H_3$ such that $[H_3\bs\bT]$ is a manifold. This proves that the 
compositions $\bT\to\relax[H_i\bs\bT]\to\relax[G\bs\bT]$
coincide.

We can now define the map $\pi_G:\bT\to \relax[G\bs \bT]$ 
as the composition 
$$\bT\rTo  \relax[H\bs \bT]\rTo\relax[G\bs \bT],$$
where $H$ is any finite index subgroup of $G$ such that
the quotient $[H\bs \bT]$ is a manifold.

\

\subsection{Gluing operations}
\label{ss:gluing-ATS}

For Riemann surfaces with parametrized boundary components,
as well for stable curves with punctures,
there exist natural gluing operations
which correspond to compositions in a {\em modular operad}.
Given two surfaces (resp., stable curves)
$S_i, \ i=1,2$ of genus  $g_i$  with $n_i$ parametrized 
boundary components (resp., with $n_i$ punctures), one can  
glue them along $a$th component (resp., puncture) of $S_1$ and
$b$th component (resp., puncture) of $S_2$ 
to get a surface (resp., a stable curve) of genus $g_1+g_2$
with $n_1+n_2-2$ boundary components (resp., punctures). 
Similarly we can glue two boundary components (resp., punctures) of 
$S_1$ and produce a surface (resp., a stable curve) of genus $g_1+1$
and $n_1-2$ boundary components (resp., punctures). 

These gluing operations on surfaces and on stable curves are compatible in the 
sense that for two marked stable curves $\phi_1:S_1\to X_1$ and $\phi_2:S_2\to X_2$ 
one can define a new marked curve $\phi:S\to X$, where $S$ is obtained by gluing 
$S_1$ and $S_2$ and $X$ is obtained by gluing $X_1$ and $X_2$.

All this is almost obvious. Note, however, that gluing stable curves
is canonical in the best possible way --- it defines the maps of the 
corresponding moduli stacks (as described in a more detail 
in~\ref{opera-sta-cu} below). 
To be justify our suggestion to interpret
augmented \TS s  as projective limits of complex 
orbifolds $[G\bs\bT]$
we have to show that  the gluing operations for the augmented \TS s 
descend to well-defined operations on complex orbifolds $[G\bs\bT]$. 
This is done in the current  subsection. 

Below we describe gluing operations for different types of objects:
first for surfaces with boundary in~\ref{opera-sur}, then
for  augmented \TS s---on the level of points---in~\ref{opera-T}.
After that in~\ref{opera-sta-cu} we recall the gluing operations
for the stacks of stable curves, and in the  last two
subsections, \ref{opera-ats-q1} and \ref{opera-ats-q2}, we describe the 
gluing operations  on the level of complex orbifolds --- quotients of
the augmented \TS s. 
Note that the description in \ref{opera-sur} and \ref{opera-sta-cu} 
contain nothing new and the construction in \ref{opera-T} is fairly obvious.

\subsubsection{Gluing bordered surfaces}
\label{opera-sur}
In what follows we denote by $\cS_{g,n}$ the groupoid whose objects
are oriented  surfaces of genus $g$ with $n$ labeled boundary
components together with a parametrization of each component. The
morphisms are diffeomorphisms preserving the parametrization of the boundary 
components, up to isotopy.%
\footnote{A version of this groupoid with non-numbered boundary components
is called {\em extended \teich\ groupoid} in~\cite{BaKi}.} 
In particular, for $S\in\cS_{g,n}$ the modular group of $S$ is just
$\Gamma(S):=\Aut_{\cS_{g,n}}(S)$.

The following gluing operations
are defined.
\begin{itemize}
\item Gluing two bordered surfaces: given $S_1\in\cS_{g_1,n_1}$
and $S_2\in\cS_{g_2,n_2}$, a choice of a pair of boundary components
in $S_1$ and in $S_2$ defines 
$$S_1\circ S_2\in\cS_{g_1+g_2,n_1+n_2-2}.$$
\item Gluing two boundary components: given $S\in\cS_{g,n}$,
a choice of a pair of boundary components defines a new surface
$\bar S\in\cS_{g+1,n-2}$. 
\end{itemize}
The gluing operations are functorial; in particular, for $S=S_1\circ S_2$
one has natural group homomorphisms $\Gamma_i\to\Gamma$ with
$\Gamma_i=\Gamma(S_i),\ \Gamma=\Gamma(S)$.

The operations described above satisfy standard axioms saying that
the collection 
$$g,n\mapsto\cS_{g,n}$$
gives a {\em modular operad} in the $2$-category of groupoids.

\subsubsection{Gluing augmented \TS s}
\label{opera-T}
It is convenient to consider the augmented \TS s as a collection
of functors
$$ \cS_{g,n}\rTo\Top$$
to topological spaces. The action of the modular groups on $\cT(S)$
is built in in this approach. The gluing operations described above
extend to the following maps connecting different $\bT(S)$:

\begin{equation}\label{eq:opera-T1}
\bT(S_1)\times\bT(S_2)\rTo\bT(S_1\circ S_2).
\end{equation}
\begin{equation}\label{eq:opera-T2}
\bT(S)\rTo\bT(\bar S).
\end{equation}
The result of gluing $(X_i,\phi_i:S_i\to X_i)$ gives the pair
$$(X,\phi:S_1\circ S_2\to X),$$ 
where $X=X_1\vee X_2$ is obtained by
gluing $X_i$ along the corresponding punctures, with $\phi=\phi_1\vee\phi_2$ 
defined by $\phi_1$ and $\phi_2$ in an obvious way.

The second operation is defined similarly.

\subsubsection{Gluing stable curves}
\label{opera-sta-cu}
     The famous modular operad 
     $$ (g,n)\mapsto\bsM_{g,n}$$
of moduli of stable curves is a close relative of the
above.\footnote{Since $\bsM_{g,n}=\relax[\Gamma(S)\bs\bT(S)]$ for
  $S\in\cS_{g,n}$} 
In order to define the gluing operations
$$\bsM_{g_1,n_1}\times\bsM_{g_2,n_2}\rTo\bsM_{g_1+g_2,n_1+n_2-2},$$
one has to be able to glue two families of punctured stable curves 
$$\cX_i\to V,\ i=1,2,$$ 
of types $(g_i,n_i)$ along a chosen
pair of punctures 
$$s_1:V\to\cX_1,\ s_2:V\to\cX_2.$$ 
This is much easier than one could have imagined:
the result is given by the colimit of the diagram
$$\cX_1\lTo V\rTo \cX_2$$
defined by the choice of the punctures. 
The existence of such (very special) colimit is easily verified.

The second type gluing operation 
$$\bsM_{g,n}\rTo\bsM_{g+1,n-2}$$
is defined similarly. Let $\cX\to V$ be a family of punctured stable curves
of type $(g,n)$ and let $s_{1,2}:V\to\cX$ be a pair of punctures. Then 
the corresponding family $\bar\cX\to V$ of type $(g+1,n-2)$ is defined
by the coequalizer of the pair $(s_1,s_2)$.

\subsubsection{Gluing quotients of $\bT_{g,n}$. Disconnected case.}
\label{opera-GT}
\label{opera-ats-q1}

The gluing operations on augmented \TS s  described in~\ref{opera-T}
 are just continuous maps of topological spaces. In this section
 we will show that they can be lifted to the level of orbifold maps
for corresponding quotient orbifolds $[G\bs\bT]$.

Let us consider first the operation that corresponds 
to gluing two different surfaces. 

\begin{Prp}
Consider two surfaces $ S_i\in\cS_{g_i,n_i}$, \ $i=1,2$.
Set  as above 
 $$S=S_1\circ S_2, \ \bT_i=\bT(S_i), \ \Gamma_i=\Gamma(S_i), \
 \Gamma=\Gamma(S).$$

Let $G\subset\Gamma$ be a finite index subgroup.
Set $G_i=\Gamma_i\times_{\Gamma}G, \ i=1,2$.
Then  there exists a natural map of complex orbifolds
\begin{equation}
\label{gluing-q-1}
[G_1\bs\bT_1]\times[G_2\bs\bT_2]\rTo\relax[G\bs\bT]
\end{equation}
which is  compatible with the gluing operation~(\ref{eq:opera-T1}) 
of topological spaces.
\end{Prp}

\begin{proof}
Choose an arbitrary  pair of marked Riemann surfaces
$$((X_1,\phi_1),(X_2,\phi_2))\in\bT_1\times\bT_2.$$ 
By~\ref{ss:qc}, there  exists a quasiconformal orbifold chart
$(U,A,\beta)$  for $\bsM_{g,n}$, where
$(g,n)=(g_1+g_2,n_1+n_2-2)$,
with $X_z=X_1\vee X_2,\ A=\Aut(X_z)$
and a pair of quasiconformal charts $(U_i,A_i,\beta_i),\
i=1,2,$ for  $\bsM_{g_i,n_i}$, with $X_i=X_{z_i}$, $A_i=\Aut(X_i)$,
such that $U_1\times U_2$ belongs to the preimage of $U$ under the
gluing map~(\ref{eq:opera-T1}). This induces maps 
  $$  f_U:U_1\times U_2\rTo U, \ \mathrm{and} \ f_A:A_1\times A_2\rTo  A. $$
Since by~(\ref{ss:coarse}) open substacks of a stack correspond to open
subsets of  its coarse space, we obtain the following $2$-commutative
diagram of stacks
\begin{equation}
\label{eq:glue-charts-M}
\begin{diagram}
[A_1\bs U_1]\times[A_2\bs U_2]& \rTo^{(f_U,f_A)}& [A\bs U] \\
\dTo & & \dTo \\
\bsM_{g_1,n_1}\times\bsM_{g_2,n_2}& \rTo & \bsM_{g,n}
\end{diagram}
\end{equation}
where the vertical arrows are embeddings of open substacks.

Recall that each node of $X_1\vee X_2$ defines a component of the 
singular locus $U-U_0$ of $U$; in particular, the node $x$ obtained by 
gluing the punctures of $X_1$ and $X_2$, defines a component $D_x$.
The image of $f_U$ lies in $D_x$ and, moreover,
$f_U$ is an open embedding of $U_1\times U_2$ to $D_x$.

Let $x_1,\ldots,x_r$ be the nodes of $X_1$, and
$x_{r+1},\ldots,x_{r+s}$ be the nodes of $X_2$. Then the nodes of 
$X_1\vee X_2$ are
$$ x_1,\ldots,x_r,x_{r+1},\ldots,x_{r+s},x.$$
The corresponding circles in $S_1\circ S_2$ consist of the circles in $S_1$
of the form $C_i=\phi_1^{-1}(x_i),\ i=1,\ldots,r$, the circles in $S_2$ 
of the form $C_i=\phi_2^{-1}(x_i),\ i=r+1,\ldots,r+s$, and the common 
boundary component of $S_1$ and $S_2$.

Let $D_i$  be the Dehn twists around $C_i,\ i=1,\ldots,r,$ in $\Gamma_1$,
and around $C_i,\ i=r+1,\ldots,r+s,$ in $\Gamma_2$,
let $\bar D_i$ be their images in $\Gamma$. Let 
$$k_i=\left\{\begin{array}{ll}
\min\{d|(D_i)^d\in G_1\}&\textrm{ for }i=1,\ldots,r,\\
\min\{d|(D_i)^d\in G_2\}&\textrm{ for }i=r+1,\ldots,r+s.
\end{array}
\right.
$$
By the choice of $G_i$ the same values
$k_i$ define the ramification indices of $V$ over $U$ 
around the corresponding components of the 
singular locus.
This implies that the fiber product $(U_1\times U_2)\times_UV$ is isomorphic 
to $V_1\times V_2$. Choose a morphism $f_V:V_1\times V_2\to V$ so that
the diagram 
$$
\begin{diagram}
V_1\times V_2 & \rTo^{f_V} & V \\
\dTo & & \dTo \\
U_1\times U_2 & \rTo^{f_U} & U
\end{diagram}
$$
is Cartesian.
Let us show that $f_V$ in the diagram above can be chosen to be compatible 
with the maps 
$$\alpha:V\rTo G\bs\bT,\ \alpha_i:V_i\rTo G_i\bs\bT_i\ (i=1,2),$$
together with the operations
$$\bT_1\times\bT_2\rTo \bT$$
defined in~\ref{opera-T}.
Let $v_i,\ i=1,2,$ be the (only) preimages of $z_i\in U_i$ in $V_i$.
Any choice of $f_V$ sends the pair $(v_1,v_2)\in V_1\times V_2$
to the only preimage $v\in V$ of $z\in U$. Both $\alpha(f_V(v_1,v_2))$
and the result of gluing $\alpha_i(v_i)$ give the element
$(X_1\vee X_2,\phi_1\vee\phi_2)\in G\bs\bT$. If now 
$(y_1,y_2)\in Y_1\times Y_2\subset V_1\times V_2$ with 
$\alpha_i(y_i)=(X'_i,\phi'_i)$, the image $\alpha(f_V(y_1,y_2))$
has form $(X'_1\vee X'_2,\phi')$ where the $G$-markings $\phi'$ and 
$\phi'_1\vee\phi'_2$ are both consistent with $\phi_1\vee\phi_2$, that is
differ by an element $\gamma\in\Gamma_{01}\times\Gamma_{02}$,
where, as in~\ref{sss:marking-the-fibers}, $\Gamma_{01}$ and $\Gamma_{02}$
are generated by the Dehn twists $D_1,\ldots,D_r$ and 
$D_{r+1},\ldots,D_{r+1}$.
The element $\gamma$ is unique modulo the intersection 
$(\Gamma_{01}\times\Gamma_{02})\cap (G_1\times G_2)$. 
The dependence of $\gamma$ on the choice of the point $(y_1,y_2)$ is 
continuous; therefore, $\gamma$ is constant. Replacing now $f_V$ to its 
composition with $\gamma$, we get a new $f_V$ with the required property.

This allows one to lift the maps $f_U$ and $f_A$ to maps
$$ f_V:V_1\times V_2\rTo V,\ f_H:H_1\times H_2\rTo H$$
connecting the orbifold charts of $[G_i\bs\bT_i]$ and $[G\bs\bT]$
and giving rise to a $2$-commutative diagram
\begin{equation}
\label{eq:glue-GT} 
\begin{diagram}
[H_1\bs V_1]\times[H_2\bs V_2]& \rTo& [H\bs V] \\
\dTo & & \dTo \\
[G_1\bs\bT_1]\times[G_2\bs\bT_2]& \rTo & [G\bs\bT]
\end{diagram}.
\end{equation}

The collections $(V_1\times V_2,H_1\times H_2,\alpha_1\times\alpha_2)$
form an orbifold atlas for the product $[G_1\bs\bT_2]\times[G_2\bs\bT_2]$.
The diagram~(\ref{eq:glue-GT}) gives, in particular, a collection
of maps
$$[H_1\bs V_1]\times[H_2\bs V_2]\rTo\relax[G\bs\bT].$$
Any morphism of charts 
$$(V_1\times V_2,H_1\times H_2,\alpha_1\times\alpha_2)\rTo
 (V'_1\times V'_2,H'_1\times H'_2,\alpha'_1\times\alpha'_2)$$
can be uniquely completed to a $2$-commutative diagram
\begin{equation}
\label{} 
\begin{diagram}
[H_1\bs V_1]\times[H_2\bs V_2]& \rTo^{\hat f_V}& [H\bs V] \\
\dTo & & \dTo \\
[H'_1\bs V'_1]\times[H'_2\bs V'_2]& \rTo^{\hat f_{V'}} & [H'\bs V']
\end{diagram}.
\end{equation}

This gives the required map~(\ref{gluing-q-1}).
\end{proof}

\subsubsection{Gluing quotients of $\bT_{g,n}$. Connected case.}
\label{opera-GT-2}
\label{opera-ats-q2}
Now we will describe gluing operation of the second type 
which corresponds to gluing two boundary components of the same surface.

\begin{Prp}
Let $S\in\cS_{g,n}$ and let $\bar S\in\cS_{g+1,n-2}$ be obtained from $S$
by gluing two chosen boundary components. Let $\Gamma=\Gamma(S)$, 
$\bar\Gamma=\Gamma(\bar S)$. One has a group homomorphism
$\Gamma\to\bar\Gamma$. Choose a finite index subgroup $\bar G$ of $\bar\Gamma$
and let $G=\Gamma\times_{\bar\Gamma}\bar G$.
Then there exists a natural map of complex orbifolds
\begin{equation}
\label{gluing-q-2}
[G\bs\bT(S)]\rTo\relax[\bar G\bs\bT(\bar S)]
\end{equation}
compatible with the continuous map~(\ref{eq:opera-T2}) 
of topological spaces.
\end{Prp}
\begin{proof}
Let $(X,\phi)\in\bT(S)$ and let $(\bar X,\bar\phi)$ be the corresponding
point in $\bT(\bar S)$.  
By~\ref{ss:qc}, there  exists a quasiconformal orbifold chart
$(\bar U,\bar A,\bar\beta)$ for
$\bsM_{g+1,n-2}$ with the exceptional curve $X_{\bar z}=\bar X$, \
$\bar A=\Aut(\bar X)$ and 
a quasiconformal chart  $(U,A,\beta)$ for  $\bsM_{g,n}$, 
with the exceptional curve $X_z=X$, $A=\Aut(X)$, 
such that $U$ is contained in the preimage of $\bar U$ under the gluing
map~(\ref{eq:opera-T2}).
This induces a pair of maps $$ f_U:U\rTo \bar U,\ f_A:A\rTo\bar A,$$
and gives  the following $2$-commutative diagram of stacks
\begin{equation}
\label{eq:glue-charts-M-2}
\begin{diagram}
[A\bs U]& \rTo^{(f_U,f_A)}& [\bar A\bs\bar U] \\
\dTo & & \dTo \\
\bsM_{g,n}& \rTo & \bsM_{g+1,n-2}
\end{diagram}
\end{equation}
where the vertical arrows are open embeddings of stacks
which are defined by~(\ref{ss:coarse}).

Similarly to~\ref{opera-GT}, the maps $f_U,\ f_A$  can be lifted 
to maps $f_V:V\to \bar V$ and $\ f_H:H\to \bar H$ defining the maps of
orbifolds 
  $$
 V\rTo\relax[H\bs V]\rTo\relax[\bar H\bs\bar V]\rTo\relax
[\bar G\bs\bT_{\bar S}].
$$

Here $(V,H,\alpha)$ is a chart of $[G\bs\bT(S)]$ and  
$(\bar V,\bar H,\bar\alpha)$ is the corresponding chart of 
$[\bar G\bs\bT(\bar S)]$. 
A morphism 
    $$ (V,H,\alpha)\rTo (V',H',\alpha')$$
between the charts defines a canonical 
$2$-morphism  connecting $V\to\relax[\bar G\bs\bT_{\bar S}]$ with  
$V'\to\relax[\bar G\bs\bT_{\bar S}]$. This defines ~(\ref{gluing-q-2}).
\end{proof}
\

\section{Augmented \TS s and admissible coverings}
\label{sec:teich-vs-adm}

Let $S$ be a compact oriented 
surface $S$ of genus $g$ with $n$ boundary 
components. Fix an unramified covering $\rho:\wt{S}\to S$ of degree $d$. 
To each marked stable curve $(X,\phi)\in\bT(S)$  a very natural
construction (described below in~\ref{adm-pointwise}) assigns an
admissible covering  $\phi_*(\rho):\wt{X}\to X$. The 
goal of this section is to show that this leads to
a continuous map   
$$v_\rho:\bT(S)\to\Adm_{g,n,d}.$$

The morphism $v_\rho:\bT(S)\to\Adm_{g,n,d}$ 
is defined as the composition
 $$ \bT(S)\rTo^{\pi_{\wt{G}}}\relax[\wt{G}\bs\bT]\rTo^{v_\rho^{\wt{G}}}\Adm_{g,n,d}$$
where $\wt{G}$ is a group defined in~\ref{ss:choice-g} below
and   $v_\rho^{\wt{G}}$ is a morphism of complex orbifolds.
Thus, using our interpretation of  $\bT$ as a projective system
of complex orbifolds, $v_\rho$ may be viewed as 
a projective system of morphisms of complex orbifolds.

The group $\wt{G}$ consists of pairs $(\wt{\gamma},\gamma)$ where 
$\gamma\in\Gamma$ and $\wt{\gamma}$ is a lifting of $\gamma$ to $\wt{S}$.
It is not a subgroup of $\Gamma$ since such lifting is not unique. 
Instead, one has a group homomorphism $\wt{G}\to\Gamma$  whose kernel is finite
and whose image $G$ has finite index in $\Gamma$. Thus,
our standard definition of the quotients $[G\bs\bT]$ and of the
canonical maps $\pi_G$ given in~\ref{chartsarecompatible}, \ref{covered-by}
does not meet  our needs; 
the quotient $[\wt{G}\bs\bT]$ and the canonical map $\pi_{\wt{G}}$
are defined in~\ref{mod-wt-G}.
The resulting orbifold is a gerbe over the quotient $[G\bs\bT]$
which is of the type we studied in Section~\ref{sec:teich}.
The orbifold charts for $[\wt{G}\bs\bT]$ have form 
$(V,\wt{H},\alpha)$ where $(V,H,\alpha)\in \cA$ is a chart for 
$[G\bs\bT]$ and $\wt{H}$ is a group endowed with a surjective
map to $H$.

The definition of the morphism $v_\rho^{\wt{G}}:[\wt{G}\bs\bT]\to\Adm_{g,n,d}$
amounts to the construction of a compatible collection 
of admissible coverings for the families of curves corresponding
to each orbifold chart of $[\wt{G}\bs\bT]$.

There is an equivariant version of the construction: if 
$\rho:\wt{S}\to S$ is an $H$-covering where $H$ is a finite group,
a continuous map $v_{\rho,H}:\bT(S)\to\Adm_{g,n}(H)$ is defined.
This is done in~\ref{ss:verho}.

The morphisms $v_\rho$ and $v_{\rho,H}$ have some important factorization
properties with respect to gluing bordered surfaces, see~\ref{sss:vrho_prop}.
The factorization properties of the maps $v_\rho$ and $v_{\rho,H}$ 
follow from the comparison of the corresponding admissible
coverings for the orbifold charts of $[\wt{G}\bs\bT]$.

The maps $v_\rho:\bT(S)\to\Adm_{g,n,d}$ are of ultimate importance in
the construction of correction classes for the definition of stringy
cohomology, see~\cite{HV}  and Section~\ref{sec:cohom}. 

\subsection{Pointwise construction}
\label{adm-pointwise}
Fix a bordered surface $S$ of genus $g$ with $n$ boundary components
and a finite covering $\rho : \wt{S} \to S$.

Let $(X,\phi:S\to X)$ be a point of $\bT(S)$.
Using the marking $\phi:S\to X$ one can
push the covering $\rho:\wt{S}\to S$ forward (see~\ref{sss:pushforward})
to get an admissible covering $\phi_*(\rho):\wt{X}\to X$.

In the case  $\rho:\wt{S}\to S$ is an $H$-covering where $H$ is a 
finite group, the map $\phi_*(\rho):\wt{X}\to X$ acquires an action of $H$
which is automatically balanced as we show in Lemma~\ref{lem:balanced} below.
 
Thus, the map $\phi_*(\rho):\wt{X}\to X$
becomes an admissible $H$-covering in the sense of Definition 4.3.1 
of~\cite{ACV}.

\subsubsection{Pushforward of $\rho$}
\label{sss:pushforward}
Here is the construction of $\phi_*(\rho)$.
Outside the nodes and the punctures of $X$
the covering  $\phi_*(\rho)$ is the pullback
of $\rho$ via $\phi^{-1}$ with the complex structure on $\wt{X}$ induced
from $X$. By passing to the normalization we get  a ramified 
covering $\beta$ of the normalization $X^\nor$ of $X$. Let $p_1$ and $p_2$ be
two points of  $X^\nor$ that correspond to a node $p$ of $X$.
The fibers of $\beta$ at $p_1$ and $p_2$ are canonically identified with the
orbits of monodromy of $\rho$ around the loop $\phi^{-1}(p)$. Thus we
obtain an admissible covering $\phi_*(\rho):\wt{X}\to X$.

Assume now that $\rho$ is an $H$-covering where $H$ is a finite group.
The group $H$ in this case acts upon the map $\phi_*(\rho):\wt{X}\to X$.
  
\begin{lem}
\label{lem:balanced}
The action of $H$ on $\phi_*(\rho):\wt{X}\to X$ is balanced.
\end{lem}
\begin{proof}
Let $y\in\wt{X}$ be a node over $x\in X$ and let $h\in H_y$ stabilize $y$.
Let $\wt{D}_+\vee \wt{D}_-$ be a small neighborhood of $y$ consisting of a 
pair of unit disks glued at $y$ and let $D_+\vee D_-$ be the corresponding
neighborhood of $x\in X$. An element $h\in H_y$ acts on $\wt{D}_+$ and on 
$\wt{D}_-$ by multiplication by primitive $n$-th roots of unity, $\zeta_\pm$.
Balancedness condition means that $\zeta_+\zeta_-=1$. One can read out the 
values of $\zeta_\pm$ from the action of $h$ on the nearby fiber of 
$\phi_*(\rho)$ at $x_\pm\in D_\pm$. Let $C=\phi^{-1}(x)$ and $\wt{C}$ be
the component of $\rho^{-1}(C)$ corresponding to $y$. The annulus
$\phi^{-1}(D_+\vee D_-)$ in $S$ admits an involution
identifying the fibers at $x_+$ and $x_-$; the corresponding
involution identifying $D_+$ and $D_-$ is antiholomorphic. Therefore, 
$\zeta_+$ and $\zeta_-$ are complex conjugate.
\end{proof}

\subsection{Modular group and some other automorphism groups}

The classical Dehn-Nielsen-Baer theorem states that the modular group 
$\Gamma(S)$ embeds into the outer automorphism group $\Out(\pi_1(S))$.
The latter group has an especially nice interpretation in terms of the 
fundamental groupoid $\Pi(S)$. In this subsection we present 
a groupoid interpretation for the modular group $\Gamma$ and for some of its
relatives. 

\subsubsection{Fundamental groupoid and the modular group}
\label{sss:fund-gr}
Recall that for a topological space $X$ its fundamental groupoid $\Pi(X)$
has the points of $X$ as the objects, and the homotopy classes of paths 
connecting the points as the arrows.
We will be especially interested in $\Pi=\Pi(S)$
  where $S$ is a fixed  oriented surface with boundary.

For (any) groupoid $\Pi$ let $\SEQ(\Pi)$ denote the groupoid of
self-equivalences of $\Pi$ and  let $\Aut(\Pi)$ denote
the corresponding group of isomorphism classes of objects of $\SEQ(\Pi)$.

For a connected groupoid $\Pi$ the group $\Aut(\Pi)$ is nothing but $\Out(\pi)$
where $\pi$ is the automorphism group of an object of $\Pi$. Thus,
for $\Pi=\Pi(S)$ the natural homomorphism from the modular group 
$\Gamma$ to $\Aut(\Pi)$ is injective.

\subsubsection{Variations}
\label{sss:var}

More generally, for a pair of groupoids $\Pi_1,\ \Pi_2$ we denote by 
$\EQ(\Pi_1,\Pi_2)$ the groupoid of equivalences $f:\Pi_1\to\Pi_2$,
so that $\EQ(\Pi,\Pi)=\SEQ(\Pi)$. We write $\Iso(\Pi_1,\Pi_2)$
for the set of isomorphism classes of objects of $\EQ$.

For a pair of functors $j_{1,2}:\Pi_{1,2}\to\Pi$ a groupoid
$\EQ(j_1,j_2)$ has as objects pairs of equivalences, 
$$ f:\Pi_1\rTo\Pi_2,\quad g:\Pi\rTo\Pi,$$
together with an isomorphism $\theta:g\circ j_1\isom j_2\circ f$.
Similarly to the above, $\Iso(j_1,j_2)$ is the set of isomorphism
classes of objects of $\EQ(j_1,j_2)$. As a special case we get a groupoid
$\SEQ(j)$ and a group $\Aut(j)$.

\subsubsection{Variations with  coverings}
\label{sss:cov}

Let $X$ be a topological space with the fundamental groupoid $\Pi$.
A covering $\rho:\wt{X}\to X$ can be described by a functor
$\Sigma:\Pi\to\Ens$ given by $\Sigma(x)=\rho^{-1}(x)$. This is a
``basepoint-free'' version of the usual description of a covering by 
the action of the fundamental group of $X$ on a fiber.

We can define now more groupoids similarly to~\ref{sss:var}.
Thus given  
$$\Sigma_i:\Pi_i\to\Ens, \ i=1,2,$$
 one defines
$\EQ((\Pi_1;\Sigma_1),(\Pi_2;\Sigma_2))$ as the groupoid whose
objects are pairs $(f,\phi)$ where $f:\Pi_1\to\Pi_2$ is an equivalence
and $\phi:\Sigma_1\to f^*(\Sigma_2)$ is an isomorphism. Similarly,
for a pair of functors $j_{1,2}:\Pi_{1,2}\to\Pi$ and $\Sigma:\Pi\to\Ens$
one defines $\EQ(j_1,j_2;\Sigma)$ to be the groupoid whose
objects are quadruples $(f,g,\theta,\phi)$ where
$$ f:\Pi_1\rTo\Pi_2,\quad g:\Pi\rTo\Pi,\quad
\theta:g\circ j_1\isom j_2\circ f,\quad
\phi:\Sigma\isom g^*\Sigma. $$
Isomorphism classes of objects of $\EQ(j_1,j_2;\Sigma)$ are denoted
by $\Iso(j_1,j_2;\Sigma)$. The notations
$$ \EQ((\Pi_1;\Sigma_1),(\Pi_2;\Sigma_2)),\
\Iso((\Pi_1;\Sigma_1),(\Pi_2;\Sigma_2)),\
\SEQ(\Pi;\Sigma),\ \Aut(\Pi;\Sigma)$$
are self-evident.

The above defined groups and sets  are connected by a bunch of forgetful
maps which are all seen in the following commutative
diagram corresponding to a pair $j_{1,2}:\Pi_{1,2}\to \Pi$
and to a functor $\Sigma:\Pi\to\Ens$
\begin{equation}
\label{eq:seqs}
\begin{diagram}
\Iso((\Pi_1;j^*_1\Sigma),(\Pi_2;j_2^*\Sigma)) & \lTo & \Iso(j_1,j_2;\Sigma) 
& \rTo & \Aut(\Pi;\Sigma)\\
\dTo  & & \dTo & & \dTo \\
\Iso(\Pi_1,\Pi_2) & \lTo & \Iso(j_1,j_2) & \rTo & \Aut(\Pi)
\end{diagram}.
\end{equation}

Note that the right-hand side square of the diagram is Cartesian.

\subsection{Choice of the group}
\label{ss:choice-g}
In this subsection we present the group $\wt{G}$ which will appear
in the decomposition 
$$ \bT(S)\rTo^{\pi_{\wt{G}}}\relax[\wt{G}\bs\bT]\rTo^{v_\rho}\Adm_{g,n,d}.$$

The group $\wt{G}$ will be chosen as a certain subgroup of 
$\wt{\Gamma}$ which is defined as the group of pairs $(\wt{\gamma},\gamma)$ 
where $\gamma\in\Gamma$ and $\wt{\gamma}$ is a lifting of $\gamma$ to $\wt{S}$.

In the notation of~\ref{sss:cov} $\wt{\Gamma}$ is just 
the fiber product $\Gamma\times_{\Aut(\Pi)}\Aut(\Pi;\Sigma)$
where $\Pi$ is the fundamental groupoid of $S$ and $\Sigma$ is defined by
$\rho$.

Let $C$ be a circle in $S$. We denote by $\rho_C$ the pullback 
      $$\rho_C: C\times_S\wt{S}\to C,$$
and by $\rho_C^k$ the pullback of $\rho_C$ with respect to 
the $k$-sheeted covering  $C\to C$. The covering $\rho_C$ is 
determined up to isomorphism, by a monodromy operator acting on a
$d$-element set;   the covering $\rho_C^k$  corresponds to  the $k$-th
power of  this operator. 

\begin{prp}There exists a subgroup $\wt{G}$ of $\wt{\Gamma}$ satisfying the
following properties.
\begin{itemize}
\item The kernel of the map $\wt{G}\to\wt{\Gamma}\to \Gamma$ is finite.
\item The image $G$ of the map $\wt{G}\to\wt{\Gamma}\to \Gamma$ has finite 
index.
\item For any circle $C$ in $S$ with the Dehn  twist $D\in\Gamma$, if
for some $k$ $D^k\in G$, then $\rho_C^k$ is trivial.
\end{itemize}
\end{prp}
\begin{proof}
The kernel of the map $\wt{\Gamma}\to\Gamma$ identifies with 
$\Aut(\wt{S}/S)$; it is, therefore, finite. Thus, the first property
of $\wt{G}$ is automatically fulfilled for any subgroup $\wt{G}$ of 
$\wt{\Gamma}$.
Let us show  that the image $\bar\Gamma$ of the map $\wt{\Gamma}\to\Gamma$
has finite index in $\Gamma$.

The covering $\rho:\wt{S}\to S$ is uniquely defined by the action of the 
fundamental group $\pi_1(S,s)$ at a point $s\in S$ on the finite set 
$\Sigma=\rho^{-1}(s)$.
Since $\pi_1(S,s)$ is finitely generated, there are finite number of
isomorphism classes of such coverings.
An element $g\in\Gamma$ belongs to $\bar\Gamma$ if and only if the inverse image 
$g^*(\wt{S})$ is isomorphic to $\wt{S}$. 
Thus, $\Gamma$ acts on a finite set (the set of
isomorphism classes of  coverings of degree $d$)
and $\bar\Gamma$ is the stabilizer of one of its elements.

We will now prove that there is a finite index subgroup $\wt{G}$ of 
$\wt{\Gamma}$ satisfying the third property. Then the second property will
be automatically fulfilled for $\wt{G}$.

The group $\Gamma$ has a finite number of orbits on the set of 
(free homotopy classes of) non-trivial circles in $S$. 
Since $\bar\Gamma$ has finite index in $\Gamma$, it has as well
a finite number of orbits. This implies that there exists an integer $K$
such that for each non-trivial circle $C$ 
one has $\rho_C^K=\id$. By Proposition~\ref{asloo} one can choose
a finite index subgroup $G$ in $\bar\Gamma$ such that
for each non-trivial circle $C$ in $S$ the corresponding Dehn twist
$D\in\Gamma$ satisfies the condition
$$ D^k\in G\Longrightarrow k\textrm{ is divisible by }K.$$

We can now define $\wt{G}=G\times_\Gamma\wt{\Gamma}$. Clearly,
the map $\wt{G}\to G$ is surjective.

\end{proof}

\subsection{The quotient $[\wt{G}\bs\bT]$.}
\label{mod-wt-G}
In this subsection we construct an orbifold atlas for the quotient
of $\bT$ modulo the group $\wt{G}$. The orbifold so defined is endowed
with a canonical projection $\pi_{\wt{G}}:\bT\to\relax[\wt{G}\bs\bT]$. 

Recall that our construction of the quotient $[G\bs\bT]$ described 
in~Section~\ref{sec:teich} is valid only for finite index subgroups of 
$\Gamma$. We lack a general 
construction of the quotient modulo a group $\wt{G}$ acting on $\bT$
via $f:\wt{G}\to\Gamma$ such that $\Ker(f)$ and $[\Gamma:\Image(f)]$
are finite. 
Our construction is specifically tailored for the groups $\wt{G}$ described 
in~\ref{ss:choice-g}.

The orbifold atlas for the quotient $[\wt{G}\bs\bT]$
is a slight modification of the atlas for $[G\bs\bT]$ where $G$ is the image
of $\wt{G}$ in $\Gamma$. For each orbifold
chart $(V,H,\alpha)\in\cA$ we construct a group epimorphism $\wt{H}\to H$
which will give rise to a chart $(V,\wt{H},\alpha)$ for the quotient
modulo $\wt{G}$. Here is how to get $\wt{H}$.

\subsubsection{Construction of the chart $(V,\wt{H},\alpha)$}
\label{mod-wt-G-chart}
Recall~\ref{another-H} that the group $H$ of symmetries of an orbifold
chart $(V,H,\alpha)$ appears as the quotient
\begin{equation}
\label{eq:h=}
 H=A_{Q,G}/\Gamma'_0
\end{equation}
where $A_{Q,G}=A_Q\times_\Gamma G$ and $\Gamma'_0=\langle D_1^{k_1},
\ldots,D_r^{k_r}\rangle$ is generated by appropriate powers of the Dehn
twists $D_i$ around the circles $C_i=\phi^{-1}(x_i)$ which are the
preimages in $S$ of the nodes of $X_z$. Define
\begin{equation}
A_{Q,\wt{G}}=A_Q\times_\Gamma\wt{G}.
\end{equation}

One has a natural projection $A_{Q,\wt{G}}\to A_{Q,G}$. We claim that
the subgroup $\Gamma'_0$ of $A_{Q,G}$ canonically lifts to $A_{Q,\wt{G}}$.

Let $\wt{C}_{ij}$ be the components of $\rho^{-1}(C_i)$ and let 
$d_{ij}$ denote the degree of $\wt{C}_{ij}$ over $C_i$. By the choice of 
$\wt{G}$, 
$k_i$ is divisible by all $d_{ij}$. Therefore, $D_i^{k_i}$ can be lifted to 
$\prod_jD_{ij}^{\frac{k_i}{d_{ij}}}$, where $D_{ij}$ denotes the Dehn twist
around $\wt{C}_{ij}$.

\

The image of $\Gamma'_0$ in $A_{Q,\wt{G}}$ will be denoted $\wt\Gamma'_0$.

Define now $\wt{H}=A_{Q,\wt{G}}/\wt\Gamma'_0$.

\
The formula~(\ref{eq:h=}) immediately gives
a canonical surjection $\wt H\to H$ with the
kernel isomorphic to $\Ker(\wt\Gamma\to\Gamma)=\Aut(\wt{S}/S)$.

The group $\wt{H}$ acting on $V$ via $H$, we have got a (highly non-effective)
orbifold chart $(V,\wt{H},\alpha)$ of $G\bs\bT=\wt{G}\bs\bT$.

Recall that $V$ contains an open dense $H$-equivariant subset 
$Y=\Gamma'_0\bs Q$. The group $\wt{H}$ acts on $Y$ via $H$.

\begin{Lem} The map $\alpha:Y\to G\bs\cT$ defines an orbifold chart
$ (Y,\wt{H},\alpha)$ for the quotient $[\wt{G}\bs\cT]$.
\end{Lem}
\begin{proof}We have to check that the map 
$\alpha:[\wt{H}\bs Y]\to\relax[\wt{G}\bs\cT]$
is an open embedding. Making the base change 
with respect to the map $\cT\to\relax[\wt{G}\bs\cT]$, we get the map
\begin{equation}
\label{eq:open-embedding-hty}
             f:  [A_{Q,\wt{G}}\bs \wt{G}\times Q] \rTo\cT,
\end{equation}
where the group $A_{Q,\wt{G}}$ acts on the product $\wt{G}\times Q$
by 
$a(g,q)=(ga^{-1},aq)$.
\footnote{Here 
we identify an element $a\in A_{Q,\wt{G}}$ with its image in $\wt{G}$.}
We have to check that~(\ref{eq:open-embedding-hty}) is an open embedding.
Since the map $[H\bs Y]\to\relax[G\bs\cT]$
(see~(\ref{eq:open-embedding-hy}) ) is an open embedding,
the base change with respect to
$\cT\to[G\bs\cT]$ gives an open embedding
$$[A_{Q,G}\bs G\times Q] \rTo\cT$$
which is equivalent to the map $f$ in~(\ref{eq:open-embedding-hty}).
\end{proof}

We will show now how to organize the charts $(\wt{H},V,\alpha)$ into an atlas.

\subsubsection{Construction of  an atlas for $[\wt{G}\bs\bT]$}
The atlas category $\wt{\cA}$ for $[\wt{G}\bs\bT]$ is defined very 
similarly to the definition of $\cA$, see~\ref{chartsarecompatible}.

The category $\wt{\cA}$ is a full subcategory of the chart category 
corresponding to the orbifold $[\wt{G}\bs\cT]$ via~\ref{orb-to-satake}.
Its objects are the triples $(Y,\wt{H},\hat\alpha)$
coming from the charts $(V,\wt{H},\alpha)$.\footnote{So these are basically
  the same objects as in $\cA$}

Note that, similarly to~~\ref{chartsarecompatible}, every object in $\wt{\cA}$
defines canonically a commutative diagram

\begin{equation}
\label{wtUYZ}
\begin{diagram}
Y    & \rEQ    &  Y  &  \rTo   & U \\
\dTo^\alpha & & \dTo & & \dTo^\beta \\
[\wt{G}\bs\cT] & \rTo &  [G\bs\cT] &\rTo &\bsM
\end{diagram},
\end{equation}
giving rise to the functors $\wt{\cA}\to\cA\to\cQ$.

The functor $c:\wt{\cA}\to
\Charts/(\wt{G}\bs\bT)$
assigns a chart $(V,\wt{H},\alpha)$ to a triple $(Y,\wt{H},\hat\alpha)$.

The required collection of isomorphisms 
$\iota:\Aut(a)\to\wt{H}(a),\ a\in\wt{\cA}$,
comes from the construction of $\wt{\cA}$ as a full subcategory
of the chart category for $[\wt{G}\bs\cT]$.

Verification of the axioms of Definition~\ref{satake-orbifold-atlas}
is immediate.

Thus, we  proved the following result.
\begin{Prp}
The category $\wt{\cA}$ defined above, together with
 the  charts $(V,\wt{H},\alpha)$, gives an orbifold 
    atlas for the quotient  $\wt{G}\bs\bT$.
 The realization of the atlas denoted as $[\wt{G}\bs\bT]$
contains the quotient $[\wt{G}\bs\cT]$ as an open dense suborbifold. 
\end{Prp}
\qed

\subsubsection{The canonical projection $\pi_{\wt{G}}:\bT\to\relax[\wt{G}\bs\bT]$}
\label{sss:pi-wtg}

In Section~\ref{sec:teich} the canonical projection $\pi_G:\bT\to[G\bs\bT]$ was
constructed in the case $G$ is a finite index subgroup of $\Gamma$. The idea
was to find a smaller group $H$ in $G$ so that the quotient $[H\bs\bT]$ is a
complex manifold, and to present the quotient map as the composition
$$ \bT\rTo\relax [H\bs\bT]\rTo\relax [G\bs\bT].$$
This approach will not work for the quotient modulo $\wt{G}$ since the action of
$\wt{G}$ on $\bT$ is not effective. 
To construct the canonical projection 
\begin{equation}
  \label{eq:canproj}
\pi_{\wt{G}}:\bT\to\relax[\wt{G}\bs\bT]
  \end{equation}
we will use the already constructed map $\pi_G:\bT\to [G\bs\bT]$.

The canonical map $[\wt{G}\bs\bT]\to\relax[G\bs\bT]$ is a gerbe.
Its base change with respect to the map $\pi_G:\bT\to\relax[G\bs\bT]$
gives a gerbe $\wt{\bT}\to\bT$. Since $\bT$ is contractible, the gerbe
$\wt{\bT}\to\bT$ is trivial, i.e. is isomorphic to the gerbe
$\Aut(\wt{S}/S)\times\bT\pile{\rTo \\ \rTo }\bT$.

Fortunately, we can point out to a canonical trivialization of this gerbe.
In fact, the base change of this gerbe with respect to the embedding
$\cT\to\bT$ gives a gerbe $\wt{\cT}\to\cT$ which is 
canonically trivialized by the fact that
$$\wt{\cT}=\cT\times_{[G\bs\cT]}[\wt{G}\bs\cT].$$
This trivialization defines a unique trivialization of the
gerbe $\wt{\bT}\to\bT$. In particular, we have a canonical splitting
$s:\bT\to\wt{\bT}$ (``zero section'').

Now we can define the map $\pi_{\wt{G}}$ as the composition
$$\bT\rTo^s \wt{\bT}\rTo\relax[\wt{G}\bs\bT].$$

\subsection{Construction of the map $v_\rho:\bT\protect\to\Adm_{g,n,d}$}
\label{ss:verho}

The canonical map $[\wt{G}\bs\bT]\to[G\bs\bT]$ gives rise to
a family of marked nodal curves over $[\wt{G}\bs\bT]$.
In order to obtain a morphism of orbifolds
\begin{equation}\label{eq:vrhoG} 
 v_\rho^{\wt{G}}:[\wt{G}\bs\bT]\rTo\Adm_{g,n,d},
\end{equation}
we will construct below an admissible covering of this family
corresponding to   $\rho:\wt{S}\to S $. 
Then, composing  $v_\rho^{\wt{G}}$
 with the canonical projection $\pi_{\wt{G}}$ constructed in~\ref{sss:pi-wtg},  
we will finally produce the  desired map $$v_\rho:\bT\rTo\Adm_{g,n,d}.$$

\subsubsection{Admissible covering of $\cX_V\protect\to V$}
\label{sss:adm-v}

Let $(V,\wt{H},\alpha)\in \wt{\cA}$ be the orbifold chart for $[\wt{G}\bs\bT]$
corresponding to a chart $(U,A,\beta)\in\cQ$ and to a marking
$\phi$ of the fiber $X_z$ of the universal family $\pi:\cX\to U$
at $z\in U$. 

We intend to construct an admissible 
covering of the induced family 
$\pi_V:\cX_V\to V$
corresponding to $\rho:\wt{S}\to S$.

Choose a contraction $c:\cX\to X_z$ (the result will not depend on the 
choice). This induces a contraction $c:\cX_V\to X_z$.

Let $x_1,\ldots,x_r$ be the nodes of $X_z$. Choose small neighborhoods
$\cO_i$ of $x_i$ in $X_z$ as in (QC3).
The manifold  $\cX_V$ is covered by the following open subsets.
\begin{itemize}
\item[1.] 
$\cY=c^{-1}(X_z-\{x_1,\ldots,x_r\}).$
\item[2.]
$
\cP_i =c^{-1}(\cO_i).
$
\end{itemize}
The sets $\cP_i$ are disjoint; one has
$$ \cY\cap \cP_i=c^{-1}(\cO_i-\{x_i\}).$$

A fiber of $\cP_i$ at $x$ looks as follows: if $x$ does not belong to the
$i$-th component of the singular locus,
it is a small annulus 
around the circle $c_x^{-1}(x_i)$. Otherwise it is a standard neighborhood 
of the node $zw=0$.

An admissible covering of $\cX_V$ is uniquely described by admissible 
coverings on $\cY$ and on $\cP_i$ together with isomorphisms on the 
intersections $\cY\cap\cP_i$.

$\cY$ is a family of (non-compact) Riemann surfaces on $V$. Admissible
covering of $\cY$ is the same as a unramified covering; it is defined
uniquely 
up to unique isomorphism by a unramified covering of
$X_z-\{x_1,\ldots,x_r\}$. 

In particular, the marking $\phi:S\to X_z$ uniquely defines a unramified
covering on $\cY$.

We denote $\cC_\cY\to\cY\to V$ the resulting admissible covering.

The intersection of each $\cP_i$ with $\cY$ is (homotopically) 
a union of two annuli. The induced unramified coverings on these
annuli are determined by the restriction of $\rho:\wt{S}\to S$ 
to the circle $C_i=\phi^{-1}(x_i)\subset S$. 
The latter is a  degree-$d$ covering $\wt{C}_i=\coprod_j \wt{C}_{ij} \to C_i$,
see~\ref{mod-wt-G-chart}.

Note that the constructed covering $\cC_\cY\to\cY$ is endowed
with a canonical isomorphism of the restriction 
$\cC_\cY|_{\cY\cap\cP_i}\to \cY\cap\cP_i$ with the one defined by $\wt{C}_i$.

\begin{Prp}
The covering $\cC_\cY\to\cY\to V$ extends uniquely up to unique isomorphism
 to an admissible covering $\cC_V\to\cX_V\to V$.
\end{Prp}
\begin{proof}
We have to construct admissible coverings $\cC_i$ of $\cP_i\to V$ endowed 
with isomorphisms of the restrictions 
$\cC_i|_{\cY\cap\cP_i}\to \cY\cap\cP_i$ with the coverings 
defined by $\wt{C}_i$. This will allow to canonically glue the coverings 
into an admissible covering of $\cX_V\to V$.

Recall that by the choice of $G$ (see~\ref{ss:choice-g}) $k_i$ are
divisible by the degrees $d_{ij}$ of the components $\wt{C}_{ij}$ of $\wt{C}_i$
over $C_i$.

We are now looking for an admissible covering $\cC_i$ of $\cP_i\to V$ 
inducing $\wt{C}_i$ on both components of the intersection $\cY\cap\cP_i$.  

Let $V_{0i}=V-D_{x_i}$ 
be the collection of points which do not belong to the
$i$-th component of the singular locus, see condition (QC6)
of the quasiconformal charts.

Let  $\cP^0_i$ be the preimage of $V_{0i}$ in $\cP_i$.
Since $\cP^0_i$ is smooth,
the restriction of an admissible covering to $\cP^0_i$ is unramified;
it is therefore determined by the action of the fundamental group of 
$\cP^0_i$ on a typical fiber of $\rho$. 

\

The fundamental group of $\cP^0_i$
is the free abelian group generated by two loops:
\begin{itemize}
\item[(Lp1)]
around the annulus in any fiber of $\pi:\cP^0_i\to V_{0i}$, and 
\item[(Lp2)]  around the singular locus
of $V_{0i}$.
\end{itemize}

The first loop is homotopic 
to each one of the components of $\cY\cap\cP_i$.
The second loop is contractible in $\cY$. 

Thus, the covering $\cC_\cY$ of $\cY$
uniquely extends to a non-ramified covering $\cC_i^0$ of the open part 
$\cP^0_i$ of $\cP_i$  so that its restriction to (Lp1)
canonically identifies with $\wt{C}_i$, whereas
the restriction  to (Lp2) is trivial.

Now we have to show that the covering $\cC_i^0$ of $\cP^0_i$
uniquely extends to an admissible covering $\cC_i$ 
of the family $\cP_i\to V$. 

Let us start with the uniqueness. The admissible covering $\cC_i$ of $\cP_i$, 
if it exists, is normal\footnote{$k[x,y,t]/(xy-t^r)$ 
is normal by Serre's criterion $R_1+S_2$}
and finite over $\cP_i$. It can therefore
be described as the normalization of $\cP_i$ in the field of meromorphic
functions on $\cC_i^0$. This gives the uniqueness.

To prove the existence, note that the projection $\pi:\cP_i\to V$
is analytically isomorphic by (QC3)(b) to the standard projection of the space
$$
P_i=\{(u,v,t_1,\ldots,t_m)\in D^2\times D^m|uv=t^{k_i}_i\}
$$
to $D^m$. Here $D$ is the standard polydisk and $k_i$ is defined by the 
condition $k_i=\min\{k|D_i^k\in G\}$ where $D_i$ is the Dehn twist around
$C_i$.

\

The generators of the fundamental group are now presented by the loops
\begin{itemize}
\item[(Lp1)] $\theta\mapsto (u\exp(2\pi i\theta),v\exp(-2\pi i\theta),
t_1,\ldots,t_m)$.
\item[(Lp2)] $\theta\mapsto (u,v,t_1,\ldots,t_i\exp(2\pi i\theta),
\ldots,t_m)$.
\end{itemize}
We have to present an admissible covering $C_i$ of $P_i$ which induces 
$\wt{C}_i$ on (Lp1)   and a trivial covering on (Lp2).

The covering $\wt{C}_i$ of the circle $C_i$ is uniquely determined by the
degrees $d_{ij}$ of each component.  We know that $d_{ij}$ divides $k_i$ for 
each $j$.   Thus, it is sufficient to present for each divisor 
$d$ of $k_i$ an admissible  covering of $P_i\to D^m$ of degree $d$, 
such that the monodromy around (Lp1) acts transitively on the generic
fiber of the covering, whereas the monodromy  around (Lp2) acts
trivially on it. 

Consider
$$
\wt{P}_i=\{(\wt{u},\wt{v},t_1,\ldots,t_m)\in D^2\times V|\wt{u}\wt{v}=
t_i^{\frac{k_i}{d}}\},
$$
and define the map $\wt{P}_i\to P_i$ by 
$u=\wt{u}^d,\ v= \wt{v}^d$.   This gives the required admissible covering.

As we have already mentioned, the admissible coverings
$\cC_\cY\to\cY\to V$ and $\cC_i\to\cP_i\to V$
glue uniquely to get an admissible covering of
the family $\cX_V\to V$. 
\end{proof}
The resulting admissible covering will be denoted as
$$\cC_V\rTo \cX_V\rTo V.$$

\begin{thm}
\label{sss:gluingup}
The admissible coverings $\cC_V\to\cX_V\to V$ constructed above canonically
glue into an admissible covering of the universal curve $\cX$ of
 $[\wt{G}\bs\bT]$.
\end{thm}

\begin{proof}

To get an admissible covering over the whole quotient $[\wt{G}\bs\bT]$,
we have to construct a canonical isomorphism
$$\cC_{V_1}\rTo u^*\cC_{V_2}$$
for each morphism $u:a_1\to a_2$ in $\wt{\cA}$,
where $c(a_i)=(V_i,\wt{H}_i,\alpha_i)$.

Let $a_i=(Y_i,\wt{H}_i,\hat\alpha_i)$. A morphism $u:a_1\to a_2$
is given by a triple 
$$u_Y:Y_1\rTo Y _2,\quad u_H:\wt{H}_1\rTo\wt{H}_2,\quad
\theta:\hat\alpha_1\rTo\hat\alpha_2\circ \hat u,$$
where $\hat u:[\wt{H}_1\bs Y_1]\to\relax[\wt{H}_2\bs Y_2]$ is induced by
$(u_Y,u_H)$.

The admissible coverings $\cC_{V_i},\ i=1,2,$ are uniquely determined by their
restrictions $\cC_{Y_i},\ i=1,2,$ to $Y_i$.
Thus, it is enough to present a canonical isomorphism
\begin{equation}
\label{gluing-c}
\cC_{Y_1}\rTo\ u_Y^*(\cC_{Y_2})
\end{equation}
of coverings of $Y_1$.
Since $u_Y$ is always injective, we can consider separately two cases:
$u$ is an embedding and $u$ is an isomorphism. 
The case when $u$ is an embedding is obvious. Let us assume now 
that $u$ is an isomorphism.

Lift a map $u_Y$ to a map $u_Q:Q_1\to Q_2$ of the universal coverings.
The obvious equivalences
$$ [A_{Q_i,\wt{G}}\bs Q_i]\rTo\relax[\wt{H}_i\bs Y_i]$$
of the orbifolds allow one to translate a morphism $u$ into
a pair of commutative diagrams
\begin{equation}
\label{dgrm:2sides}
\begin{diagram}
Q_1 & \rTo & \cT    & &&  &  A_{Q_1,\wt{G}} & \rTo & \wt{G} \\ 
\dTo^{u_Q}&&\dTo^g&&\textrm{and}&& \dTo^{\ad(u_Q)} & & \dTo^{\ad(\wt{g})} \\
Q_2 & \rTo & \cT    & &&  &  A_{Q_2,\wt{G}} & \rTo & \wt{G} 
\end{diagram}
\end{equation}
for some $g\in G$ and a lifting $\wt{g}$ of $g$ in $\wt{G}$.

The element $\wt{g}\in\wt{G}$ defines an isomorphism (\ref{gluing-c})
as follows. We assume that $z_i\in V_i$ satisfy the condition $u_V(z_1)=z_2$.
Let $\Pi_1$ (resp., $\Pi_2$) be the fundamental groupoid of
$S-\cup C^1_i$ (resp.,  $S-\cup C^2_i$), where $C^1_i=\phi_1^{-1}(x_i)$
and similarly for $C^2_i$, and let $j_i:\Pi_i\to\Pi,\ i=1,2,$ be the
obvious embeddings.

The element $g\in G$ appearing in the left-hand side of (\ref{dgrm:2sides})
represents an element of $\Iso(j_1,j_2)$ (see~\ref{sss:var}); its lifting
$\wt{g}$ gives an element of $\Iso(j_1,j_2;\Sigma)$ as in the diagram
(\ref{eq:seqs}). This defines an isomorphism between $(\Pi_1;j_1^*\Sigma)$
and $(\Pi_2;j_2^*\Sigma)$ which is precisely the isomorphism 
$\cC_{Y_1}\to u_Y^*(\cC_{Y_2})$ we need.

Another choice of lifting  $u_Q:Q_1\to Q_2$ of $u_Y$ leads to
different $g$ and $\wt{g}$. The difference is, however, not very serious.
If $u'_Q$ is another lifting, one has $u'_Q=u_Q\circ\gamma$ where
$\gamma\in\Gamma_0'=\langle D_1^{k_1},\ldots,D_r^{k_r}\rangle$. Thus
the lifting $u'_Q$ gives rise to the pair $g'\in G,\ \wt{g}'\in \wt{G}$
where
$$ g'=g\gamma,\ \wt{g}'=\wt{g}\wt{\gamma}$$
and $\wt{\gamma}$ is the canonical lifting of $\gamma$.

Since $\wt{\gamma}$ is a product of Dehn twists along the components $C'_{ij}$
of $\rho^{-1}(C_i)$, the induced element of 
$\Iso((\Pi_1;j_1^*\Sigma),(\Pi_2;j_2^*\Sigma))$
is the same.

The continuous map
$$v_\rho:\bT\rTo\Adm_{g,n,d}$$
is constructed.
\end{proof}

\subsubsection{The map $v_\rho$ on the level of points}
\label{v-rho-points}

To make sure we constructed exactly what was announced at the beginning of
Section~\ref{sec:teich-vs-adm}, let us describe the image
$v_\rho(X,\phi)$ for arbitrary $(X,\phi)\in\bT$.

We can assume that $(X,\phi)$ belongs to the image 
$\alpha(V)$ of an orbifold chart $(V,H,\alpha)$ of $G\bs\bT$.

The admissible covering $\cC_V$ of $V$ was constructed  by gluing
admissible coverings $\cC_{\cY}$ and $\cC_i$ of $\cY$ and of $\cP_i$
respectively, see~\ref{sss:adm-v}.
Let $X=X_v$ for $v\in V$. The intersection $\cY\cap X$ is 
$X-c_v^{-1}\{x_1,\ldots,x_r\}$ where $c_v:X=X_v\to X_z$ is the restriction
of the contraction to $X_v$. An admissible covering of $X$ is uniquely
determined by its restriction to $X\cap\cY$; Since the $G$-markings of $X$ 
and of $X_z$ are compatible, the restriction of the admissible covering
on $X\cap\cY$ induced from $\cC_V$ is the same as the one described
in~\ref{adm-pointwise}. Therefore, $v_\rho(X,\phi)$ is presented
by the admissible covering of $X$ described in~\ref{adm-pointwise}.

\subsubsection{Admissible $H$-coverings}
\label{sss:adm}
If $\rho:\wt{S}\to S$ is an $H$-covering, the resulting admissible
coverings $\cC_V$ of $(V,\wt{H},\alpha)$ acquire an $H$-action. 
Since the balancedness condition is verified at each point
by \ref{v-rho-points} and \ref{lem:balanced}, the admissible covering
of $[\wt{G}\bs\bT]$ becomes an admissible $H$-covering. Thus, a map
\begin{equation}
  \label{eq:adm}
v_{\rho,H}:\bT\rTo\Adm_{g,n}(H)  
\end{equation}
is defined.

\subsection{Compatibilities}
\label{sss:vrho_prop}

The augmented \TS s $\bT_{g,n}$ as well as the stacks of admissible 
coverings $\Adm_{g,n,d}$ have various gluing operations giving
rise to (a sort of) modular operads, see~\ref{ss:gluing-ATS}.

In this subsection we will describe
the compatibility of these structures with the map $v_\rho$.

We also describe functoriality of $v_{\rho,H}$ with respect to
the change of $H$.

The proofs of the properties \ref{sss:fun}--\ref{sss:fact2} are given 
in~\ref{sss:proofs}--\ref{sss:proofs-2}. Basically, the properties follow 
directly from the construction of an admissible covering of 
$[\wt{G}\bs \bT_S]$ described in~\ref{sss:gluingup}.

\subsubsection{Functoriality for $v_{\rho,H}$} 
\label{sss:fun}
We will now describe functoriality for the maps $v_{\rho,H}$.
Let $\rho:\wt{S}\to S$ be an $H$-covering and let $f:H\to H'$ be a finite
group homomorphism. This defines an $H'$-covering $\rho':\wt{S}'\to S$
obtained by {\em induction} along $H\to H'$. If $f$ is injective, $\wt{S}'$
consists of $[H':H]$ copies of $\wt{S}$. If $f$ is surjective, $\wt{S}'$
is the quotient of $\wt{S}$ by the group $\Ker(f)$.

One has
\begin{Prp}
A group homomorphism $f:H\to H'$ induces a map of the stacks
$$ f_*:\Adm_{g,n}(H) \rTo \Adm_{g,n}(H').$$
Moreover, the following diagram
$$
\begin{diagram}
\bT(S) & \rTo^{=} & \bT(S)\\
\dTo^{v_{\rho,H}} & & \dTo^{v_{{\rho',H'}}} \\
\Adm_{g,n}(H) & \rTo^{f_*} & \Adm_{g,n}(H')
\end{diagram}
$$
is 2-commutative. 
\end{Prp}

\subsubsection{Factorization (gluing two bordered surfaces)}
\label{sss:fact1}

Let $S_1\in\cS_{g_1,n_1},\ S_2\in\cS_{g_2,n_2}$ be two bordered surfaces. 
Choose a boundary component in each one of $S_i$ and let 
$S=S_1\circ S_2\in\cS_{g,n}$ where $g=g_1+g_2,\ n=n_1+n_2-2$.

Fix a finite covering
$\rho:C\to S$ and let $\rho_i:C_i\to S_i$ 
be the induced covering of $S_i,\ i=1,2$.

Let $\Upsilon_d$ denote the (discrete) groupoid of finite multisets
of weight $d$: its objects are pairs $(X,w)$ where $X$ is a finite set
and $w:X\to \Z_{>0}$ satisfies $\sum w(x)=d$.

Let $\pi:C\to X$ be  an admissible covering of  degree $d$.
Then any $x\in X$
defines an object of $\Sigma_d$: this is the set-theoretic preimage
$\pi^{-1}(x)$ with the weight function defined by the multiplicities
of the points of $\pi^{-1}(x)$. The covering $C$ is non-ramified at $x$
 if and only if all points of $\pi^{-1}(x)$ have weight one.
An admissible covering $\cC\rTo^\pi\cX\to V$ of degree $d$
and a choice of a puncture $s:V\to\cX$ 
defines a map $V\to\Upsilon_d$ 
which assigns
to $v\in V$ the fiber of the map $\cC_v\to\cX_v$ at $\pi(v)$.
This map is locally constant.
Thus, the map
$$ F_s:\Adm_{g,n,d}\rTo\Upsilon_d$$
of orbifolds is defined. In particular, a choice of boundary components
of $S_i,\ i=1,2,$ defines a pair of maps 
$\Adm_{g_i,n_i,d}\to\Upsilon_d,\ i=1,2.$

\begin{Prp}
\begin{itemize}
\item[(1)]Gluing of $S_i$  defines a canonical operation
\begin{equation}
\label{gluing-adm-1} 
\Adm_{g_1,n_1,d}\times_{\Upsilon_d}\Adm_{g_2,n_2,d}\rTo^\iota  
\Adm_{g,n,d}.
\end{equation}
\item[(2)]The product of maps $v_{\rho_1}$ and $v_{\rho_2}$
$$v_{\rho_1}\times v_{\rho_2}:\bT(S_1)\times\bT(S_2)\rTo
\Adm_{g_1,n_1,d}\times\Adm_{g_2,n_2,d}$$
is canonically factored through the map
$$\Adm_{g_1,,n_1,d}\times_{\Upsilon_d}\Adm_{g_2,n_2,d}\rTo
\Adm_{g_1,n_1,d}\times\Adm_{g_2,n_2,d}.$$
\item[(3)] The following diagram
\begin{equation}\label{diag:comp-bound}
\begin{diagram}
\bT(S_1)\times\bT(S_2) & \rTo^\iota & \bT(S) \\
\dTo^{v_{1,2}} & & \dTo^{v_\rho} \\
\Adm_{g_1,n_1,d}\times_{\Upsilon_d}\Adm_{g_2,n_2,d}& \rTo^\iota & 
\Adm_{g,n,d}
\end{diagram}
\end{equation}
is $2$-commutative. Here $v_{1,2}$ is defined by 
$v_{\rho_1}\times v_{\rho_2}$ via (2). 
\end{itemize}
\end{Prp}

\subsubsection{Factorization (gluing two boundary components)}
\label{sss:fact2}

Let $S\in\cS_{g,n}$ be a bordered surface. 
Gluing a pair of boundary components
in $S$ we get a surface $\bar{S}\in\cS_{g+1,n-2}$
together with a canonical map $S\to\bar{S}$.
Fix a finite covering
$\rho:C\to \bar{S}$ and let $\rho_S:C_S\to S$ 
be the induced covering of $S$.

The choice of two boundary components in $S$ defines
a map 
$$\Adm_{g,n,d}\to\Upsilon_d\times\Upsilon_d$$ 
as in~\ref{sss:fact1}.

\begin{Prp}
\begin{itemize}
\item[(1)]
Gluing of two boundary components of $S$ defines a canonical operation
\begin{equation}
\label{adm-gluing-2}
\Adm_{g,d,n}\times_{\Upsilon_d\times\Upsilon_d}\Upsilon_d\rTo\Adm_{g+1,n-2,d}.
\end{equation}
\item[(2)]The map
$$v_{\rho_S}:\bT(S)\rTo\Adm_{g,n,d}$$
is canonically factored through the projection onto the first factor
$$ \Adm_{g,d,n}\times_{\Upsilon_d\times\Upsilon_d}\Upsilon_d\rTo\Adm_{g,d,n}.
$$
\item[(3)]
The following diagram
\begin{equation}\label{diag:comp-bound2}
\begin{diagram}
\bT(S) & \rTo^\iota & \bT(\bar{S}) \\
\dTo^{v'_{\rho_S}} & & \dTo^{v_\rho} \\
\Adm_{g,n,d}\times_{\Upsilon_d\times\Upsilon_d}\Upsilon_d& \rTo^\iota & 
\Adm_{g+1,n-2,d}
\end{diagram}
\end{equation}
is $2$-commutative. Here $v'_{\rho_S}$ is defined by $v_{\rho_S}$ via (2).
\end{itemize}
\end{Prp}

\subsubsection{Operations for $\Adm$}
\label{sss:proofs} 
\label{opera-A} 

The induction operation 
$$f_*:\Adm_{g,n}(H)\to\Adm_{g,n}(H')$$
can be constructed separately for the case $f:H\to H'$ is injective or
surjective. 

If $f$ is injective and if $\cC\to\cX\to V$ is an admissible $H$-covering,
its image under $f_*$ is given by $\cC'\to\cX\to V$ where $\cC'$
is disjoint union of $[H':H]$ copies of $\cC$. 

If $f$ is surjective, $\cC'$ is the quotient of $\cC$ by the action of 
$\Ker(f)$. 

To define the gluing operation~(\ref{gluing-adm-1}), we have to construct,
given two families $\cC_i\to\cX_i\to V$, $i=1,2,$ of admissible coverings, 
together
with a choice of punctures $s_i:V\to \cX_i$ and an isomorphism
$\theta:F_{s_1}\to F_{s_2}$ in $\Upsilon_d$, a glued up family
$\cC\to\cX\to V$. We have already described  (see~\ref{opera-sta-cu})
how to get $\cX$ as the colimit of the diagram $\cX_1\lTo V\to \cX_2$. 
Similarly, the choice
of punctures $s_1,\ s_2$ and of $\theta$ define a one-to-one correspondence
between the punctures of $\cC_1$ over $s_1$ and the punctures of $\cC_2$
over $s_2$. The coproduct of $\cC_1$ and $\cC_2$ under an appropriate
number of copies of $V$ gives the admissible covering $\cC$. 

The gluing operation~(\ref{adm-gluing-2}) is defined similarly.

\subsubsection{Proof of \ref{sss:fun}--\ref{sss:fact2}}
\label{sss:proofs-2}

Proposition~\ref{sss:fun} results from the following
obvious observation. Let $\cC_V\to\cX_V\to V$ (resp.,
$\cC'_V\to\cX_V\to V$)
be an admissible covering constructed as in~\ref{sss:adm-v} for
$\rho:\wt{S}\to S$ (resp., for $\rho':\wt{S}'\to S$). Then
$\cC'=f_*(\cC)$.

To prove \ref{sss:fact1}, let $G$ be the finite index subgroup of
$\Gamma(S)$ chosen as in~\ref{ss:choice-g} for the covering
$\rho:C\to S=S_1\circ S_2$; the groups $\Gamma(S_i)$ embed into $\Gamma(S)$;
define $G_i=\Gamma_i\times_\Gamma G$.

The gluing operation
$$ [G_1\bs\bT_1]\times[G_2\bs\bT_2]\rTo\relax[G\bs\bT]$$
defined in~\ref{opera-GT} extends trivially to its gerbe-version
$$ [\wt{G}_1\bs\bT_1]\times[\wt{G}_2\bs\bT_2]\rTo\relax[\wt{G}\bs\bT],$$
where the groups $\wt{G}_i,\ \wt{G}$ are defined as in~\ref{ss:choice-g}.

The property (3) of~\ref{sss:fact1} results from
$2$-commutativity of the following diagram of complex orbifolds.

\begin{equation}\label{diag:comp-bound-2}
\begin{diagram}
[\wt{G}_1\bs\bT_1]\times[\wt{G}_2\bs\bT_2] & \rTo^\iota & [\wt{G}\bs\bT] \\
\dTo^{v_{1,2}} & & \dTo^{v_\rho} \\
\Adm_{g_1,n_1,d} \times_{\Upsilon_d}\Adm_{g_2,n_2,d}& \rTo^\iota & 
\Adm_{g,n,d}
\end{diagram}
\end{equation}

The latter results from the following observation.
Let $(V_i,\wt{H}_i,\alpha_i),\ i=1,2,$ and $(V,\wt{H},\alpha)$
be the charts for the quotients 
$[\wt{G}_1\bs\bT_1]$, $[\wt{G}_2\bs\bT_2]$ and $[\wt{G}\bs\bT]$ respectively,
so that a map              
$$ f_V:V_1\times V_2\rTo V$$
realizes the gluing operation~(\ref{diag:comp-bound-2})
as in~\ref{opera-GT}. The spaces $V_1,\ V_2$ and $V$ are bases of families
of admissible coverings $\cC_i\to\cX_i\to V_i$ and $\cC\to\cX\to V$
constructed as in~\ref{sss:adm-v}. Then the inverse image $f_V^*(\cC)$
identifies with the admissible covering based on $V_1\times V_2$
obtained by gluing of 
$$\cC_1\times V_2\rTo\cX_1\times V_2\rTo V_1\times V_2$$
and 
$$V_1\times\cC_2\rTo V_1\times\cX_2\rTo V_1\times V_2$$
as described in~\ref{opera-A}.

The observation follows from the fact that an admissible covering of   
the family $(\cX_1\times V_2)\vee(V_1\times\cX_2)$ is uniquely defined
by its restriction to the smooth locus of the exceptional curve $X_1\vee X_2$
--- see~\ref{sss:adm-v}. 

The proof of \ref{sss:fact2} goes along the same lines. 

\subsection{Associativity of the stringy orbifold cup-product} 
\label{sec:cohom}
As an application of the results proved earlier in this section,
we will show how they can be used in the study of orbifold cohomology.

Based on work of string theorists, Chen and Ruan in~\cite{CR} 
(see also~\cite{FG}) introduced a new invariant of almost complex orbifolds   
called the stringy orbifold cohomology ring.
Multiplication in this ring is defined in a very non-trivial way
and the proof of its associativity given in~\cite{CR} and~\cite{FG}
involves various moduli spaces of stable Riemann surfaces with punctures.

In the forthcoming work~\cite{HV} we will show that
augmented \TS s and their properties established in this
paper provide a very natural tool for dealing with 
various orbifold cohomology theories. 

Here we will only illustrate this by showing how
to fix some problems in the proofs of associativity 
of the stringy orbifold cup-product given in~\cite{CR} and in~\cite{FG}
(we elaborate on this in Remark~\ref{rem:gap} below).

Let $X=[Y/G]$ be an almost complex global quotient orbifold,
i.e.\  $Y$ is an almost complex manifold and $G$ a finite group which 
acts on $Y$ by diffeomorphisms preserving the almost complex structure.
Proof of associativity of stringy orbifold cohomology cup-product 
reduces to the following statement.

Let $g_1,g_2,g_3,g_4$ be a quadruple of elements in $G$ 
with $g_1 g_2 g_3 g_4= 1$. Let 
$$H=\langle g_1,g_2,g_3,g_4\rangle\subset G$$
be the subgroup in $G$ generated by these elements. 
Define two representations
$V_L$ and $V_R$ of the group $H$ as follows.
Let $(C,p_1,p_2,p_3,p_4)\in \cM_{0,4}$ 
be a nodal Riemann surface obtained by gluing two Riemann spheres
$C_1$ and $C_2$ at a point $p$ with punctures $p_1,p_2$ on the component
$C_1$ and $p_3,p_4$ on $C_2$.
Let 
$$
  \pi:\wt{C}\to C
$$ 
be the Galois covering of $C$ with the Galois group $H$ 
unramified outside of the punctures $\{p_1,p_2,p_3,p_4\} \subset C$,
such that the     monodromy
around $p_i$ is given by the action of 
$g_i\in H$.
Let 
$$
  V_L=H^1(\wt{C}, \cO_{\wt{C}})
$$
 be the representation of $H$ given by the action
of $H$ on $\wt{C}$.

Note that this covering depends on a choice of a marking of $C$,
i.e.\ of an identification of the fundamental group 
$\pi_1(C - \{p_1,p_2,p_3,p_4\})$ with the free group 
$$
F_3=\langle x_1,x_2,x_3,x_4 | \prod x_i =1\rangle.
$$

Another representation of $H$ denoted $V_R$ is 
constructed by relabeling the marked points.
Now we put the points $p_1$ and $p_3$ on $C_1$ and
the points $p_2$ and $p_4$ on $C_2$.

The proof of associativity of stringy orbifold cup-product
in~\cite{FG} reduces to the following statement.

\begin{lem}
 \label{lem:fg}
The representations $V_L$ and $V_R$ of the group $H$ are isomorphic.  
\end{lem}
\begin{proof}
Let $S$ be a surface obtained by removing four open disks (holes) from $S^2$.  
The fundamental group of $S$ can be identified 
with 
$$F_3=\langle x_1,x_2,x_3,x_4 | \prod x_i =1\rangle,$$
where $x_i$ corresponds to the path going around the boundary of the $i$th
hole. This gives an epimorphism $F_3 \to H$ and with it a canonical
$H$-covering $\rho: \wt{S}\to S$. 

Due to the result of Section~\ref{sss:adm}  there exists a map~(\ref{eq:adm})
$$
v_{\rho,H}: \bT (S) \rTo \Adm_{g,n}(H)
$$
for certain $g$ and $n$.

 The tautological family of curves
$$
\wt{C}\rTo^\pi C \rTo^ \sigma \Adm_{g,n}(H)
$$
gives an $H$-equivariant vector bundle $\cV$ on $\Adm_{g,n}(H)$
defined by
$$
\cV = R^1(\sigma \pi)_*(\cO_{\wt{C}})
$$
which induces via $v_{\rho,H}$ an $H$-equivariant vector bundle 
    $$\cW=v^*_{\rho,H}(\cV)$$ on $\bT (S)$.

Representations $V_L$ and $V_R$ of $H$ constructed above
appear as fibers of $\cW$ at two different boundary points 
of $\bT(S)$ (they correspond to curves $C$ and $\bar{C}$ with
specific choices of marking). The desired isomorphism
between $V_L$ and $V_R$
now follows from connectedness of $\bT(S)$.
\end{proof}

\begin{rem}
  \label{rem:gap}
The proof of this lemma given in~\cite{FG}  uses the moduli stack $\sM_{g,n}$
instead of the augmented \TS\ $\bT$.

This does not allow to take into account  the dependence of the
construction of relevant coverings  on the choice of markings 
(which is equivalent to an identification of the fundamental group 
of the     punctured surface with the free group).    
An attempt to resolve the issue by replacing $\sM_{g,n}$ with the
stack $\Adm$ hits the problem of high non-connectivity of $\Adm$.
\end{rem}


\begin{thebibliography}{99}

\bibitem{ab1} W.~Abikoff,
Augmented \TS s.
Bull. AMS  82, 333--334 (1976).

\bibitem{ab2} W.~Abikoff,
Degenerating families of Riemann surfaces.  
Ann. Math. 105,  29--44 (1977).

\bibitem{ab3} W.~Abikoff, 
The real analytic theory of \TS.
Lect. Notes in Math. 820 (1980).

\bibitem{ACV} D.~Abramovich, A.~Corti, A.~Vistoli, 
Twisted bundles and   admissible covers.
Comm. Algebra 31, 3547--3618 (2003).

\bibitem{BaKi} B.~Bakalov, A.~Kirillov, Jr., 
Lectures on tensor categories and modular functors.
Amer. Math. Soc. (2001).

\bibitem{berxu} K.~Behrend, P.~Xu, 
Differentiable stacks and gerbes.
Preprint, arXiv: math.DG/0605694.

\bibitem{bers} L.~Bers, 
Spaces of degenerating Riemann surfaces.
Discontinuous groups and Riemann surfaces, 
Ann. Math. Stud. 79, 43--55  (1973).

\bibitem{bers2} L.~Bers, 
On spaces of Riemann surfaces with nodes.
Bull. AMS 80, 1219--1222 (1974). 

\bibitem{bers3} L.~Bers,
Deformations and moduli of Riemann surfaces with nodes and signatures.
Math. Scand. 36, 12--16 (1975). 

\bibitem{BoPi} M.~Boggi, M.~Pikaart, 
Galois covers of moduli of curves.
Compositio Math. 120, 171--191 (2000).

\bibitem{brau} V.~Braungardt, 
\"Uberlagerungen von Modulr\"aumen f\"ur   Kurven.
Ph.D. thesis, Karlsruhe (2001).

\bibitem{brock} J.~Brock, 
The Weil-Petersson metric and volumes of 3-dimensional hyperbolic convex cores.
J. Amer. Math. Soc. 16, 495--535 (2003).

\bibitem{brock-marg}
J.~Brock, D.~Margalit, 
Weil-Petersson isometries via the pants complex.
Preprint, arXiv:math.DG/0412499.

\bibitem{C} B.~Conrad, 
Grothendieck duality and base change.
Lect. Notes in Math. 1750 (2000).

\bibitem{CR}W.~Chen, Y.~Ruan, 
A new cohomology theory of orbifolds.
Comm. Math. Phys. 248, 1--31 (2004).

\bibitem{dm} P.~Deligne, D.~Mumford, 
The irreducibility of the space of curves of a given genus.
Publ. Math. IHES 36, 75--109 (1969).

\bibitem{earle} C.J.~Earle, 
On holomorphic families of pointed Riemann surfaces.
Bull. AMS 79, 163--166 (1973).

\bibitem{engber} M.~Engber, 
\TS s and representability of functors.
Trans. AMS 201, 213--226 (1975).

\bibitem{FG} B.~Fantechi, L.~G\"ottsche, 
Orbifold cohomology for global quotients.
Duke Math. J. 117, 197--227 (2003).

\bibitem{gr} H.~Grauert, R.~Remmert, 
Coherent analytic sheaves.
Grundlehren der mathematischen Wissenschaften 265, Springer (1984).

\bibitem{groth}  A.~Grothendieck, 
Technique de construction en g\'eom\'etrie analytique. I.
Description axiomatique de l'espace de \teich\ et de ses variants,  
S\'eminaire H.~Cartan 13' 1960/61, Exp.~7--8, 1--33.

\bibitem{hakim} M.~Hakim, 
 Topos annel\'es et sch\'emas relatifs.
Springer (1972).

\bibitem{ha} R.~Hartshorne, 
 Residues and duality.
Lect. Notes in Math.  20 (1966).


\bibitem{hs}
F.~Herrlich, G.~Schmith\"usen, 
On the boundary of Teichm\"uller disks in Teichm\"uller and in Schottky space.
Handbook of Teichm\"uller theory. Vol. I,  293--349, 
IRMA Lect. Math. Theor. Phys. 11, Eur. Math. Soc. (2007). 

\bibitem{cex} V.~Hinich, 
Plumbing coordinates on Teichmuller space: a counterexample.
Israel J. Math. (to appear), arXiv:math.CV/0606240.

\bibitem{double} V.~Hinich, Drinfeld double for orbifolds. 
Quantum groups. Contemp. Math. 433, 251--265,
Amer. Math. Soc. (2007), arXiv:math.QA/0511476.

\bibitem{HV} V.~Hinich, A.~Vaintrob, 
Stringy orbifold cohomology via \TS s.
In preparation.

\bibitem{ImaTani} Y.~Imayoshi, M.~Taniguchi, 
An introduction to \TS s. 
Springer (1992).

\bibitem{Iv} N.~Ivanov,
Mapping class groups.  
Handbook of geometric topology,  523--633,
North-Holland (2002). 

\bibitem{JKK} T.~J.~Jarvis, R.~Kaufmann, T.~Kimura, 
Pointed admissible $G$-covers and $G$-equivariant cohomological field theories.
Compos. Math. 141, 926--978 (2005).

\bibitem{keelmori} S.~Keel, S.~Mori, 
Quotients by groupoids.
Ann. Math. 145, 193--213 (1997). 

\bibitem{knudsen} F.~Knudsen, 
The projectivity of the moduli space of stable curves. III. 
The line bundles on $M_{g,n}$ and the projectivity of $\overline{M}_{g,n}$.
Math. Scand. 52, 200--212 (1983).

\bibitem{knutson} D.~Knutson, 
Algebraic Spaces.
Lect. Notes in Math. 203 (1971).

\bibitem{Kod}K.~Kodaira, 
Complex manifolds and deformations of complex structures. 
Grundlehren der mathematischen Wissenschaften 283, Springer (1986).

\bibitem{LMB} G. Laumon, L. Moret-Bailly, 
 Champs Alg\'ebriques.
Springer (2000).

\bibitem{Loo} E.~Looijenga, 
Smooth Deligne-Mumford compactifications by means of Prym level structures.
J. Alg. Geom. 3, 283--293 (1994).

\bibitem{marden} A.~Marden, 
Geometric complex coordinates for \TS.
Mathematical aspects of string theory, 341--354,
Adv. Ser. Math. Phys. 1, World Sci. (1987).

\bibitem{masur} H.~Masur, 
Extension of the Weil-Petersson metric to the boundary of \TS.
Duke math. J. 43, 623--635 (1976).

\bibitem{metzler} D.~Metzler, 
Topological and smooth stacks.
Preprint, arXiv:math.DG/0306176. 

\bibitem{Mochizuki} S.~Mochizuki, 
The geometry of the compactification of the Hurwitz scheme.
Publ. RIMS, 31, 355--441 (1995).

\bibitem{moepronk} I.~Moerdijk, D.~Pronk,  
Orbifolds, sheaves and  groupoids.
 K-theory 12, 3--21 (1997).

\bibitem{noohi}
B.~Noohi, Foundations of topological stacks. I.
Preprint,  arXiv: math.AG/0503247. 

\bibitem{RRV}  J.-P.~Ramis, G.~Ruget, J.-L.~Verdier,
Dualit\'e relative en g\'eom\'etrie analytique complexe.
Invent. Math. 13, 261--283 (1971).

\bibitem{satake} I.~Satake, 
The Gauss-Bonnet theorem for $V$-manifolds.
J. Math. Soc. Japan 9, 464--492 (1957).

\bibitem{sga1} 
SGA 1.  Rev\^etements \'etales et groupe fondamental,
S\'eminaire de g\'eom\'etrie alg\'ebrique du Bois Marie, 1960--1961.
Lect. Notes in Math.  224 (1971).

\bibitem{sga4} 
SGA 4.  Th\'eorie de topos et cohomologie \'etale des sch\'emas
S\'eminaire de g\'eom\'etrie alg\'ebrique du Bois Marie, 1963--1964.
Lect. Notes in Math. 269, 270, 305 (1972--1973).


\bibitem{wolpert} S.~Wolpert,  
Geometry of the Weil-Petersson completion of \TS. 
Surveys in differential geometry VIII,  357--393,  
International Press (2003).
\end{thebibliography}
\end{document}